%% file: main.tex
\documentclass[11pt]{article}

\usepackage{enumitem} 

\input{premeable}

\begin{document}
\title{Signal-to-noise ratio aware minimax analysis of sparse linear regression}

\author[1]{Shubhangi Ghosh}
\author[1]{Yilin Guo}
\author[2]{Haolei Weng}
\author[1]{Arian Maleki}
\affil[1]{Columbia University}
\affil[2]{Michigan State University}

\date{\vspace{-5ex}}
\maketitle





\begin{abstract}

We consider parameter estimation under sparse linear regression -- an extensively studied problem in high-dimensional statistics and compressed sensing. While the minimax framework has been one of the most fundamental approaches for studying statistical optimality in this problem, we identify two important issues that the existing minimax analyses face: (i) The signal-to-noise ratio appears to have no effect on the minimax optimality, while it shows a major impact in numerical simulations. (ii) Estimators such as best subset selection and Lasso are shown to be minimax optimal, yet they exhibit significantly different performances in simulations. In this paper, we tackle the two issues by employing a minimax framework that accounts for variations in the signal-to-noise ratio (SNR), termed the SNR-aware minimax framework. We adopt a delicate higher-order asymptotic analysis technique to obtain the SNR-aware minimax risk. Our theoretical findings determine three distinct SNR regimes: low-SNR, medium-SNR, and high-SNR, wherein minimax optimal estimators exhibit markedly different behaviors. The new theory not only offers much better elaborations for empirical results, but also brings new insights to the estimation of sparse signals in noisy data. 
    
\end{abstract}

\section{Introduction}
\label{intro:label}

\subsection{Limitations of the minimax framework}
For i.i.d.~samples $\{(y_i,x_i)\}_{i=1}^n$, we consider the linear regression model under isotropic Gaussian design,
\begin{equation}\label{model::gaussian-model}
    y_{i} = x_{i}^{T}\beta + \sigma z_{i}, \qquad i=1, \ldots, n,
\end{equation}
where $y_{i}\in\R $ denotes the response, $x_{i}\sim \mathcal{N}(0, \frac{1}{n}I_p)$ represents the covariate vector, $\beta\in \R^{p}$ is the unknown signal vector to be estimated, and $z_i \sim \mathcal{N}(0, 1)$ is the standard normal error independent of $x_i$. As in the rich literature of sparse linear regression \cite{hastie2009elements, buhlmann2011statistics, hastie2015statistical, wainwright2019high, fan2020statistical}, we are interested in studying this problem in the high-dimensional setting where $p$ is comparable with or much larger than $n$, and $\beta$ belongs to 
\begin{equation} \label{param::sparse}
\Theta(k):= \Big\{ \beta \in \R^{p}: \|\beta\|_{0} \leq k \Big \},
\end{equation}
where $\|\beta\|_{0}$ denotes the number of non-zero components of $\beta$. The minimax framework aims to calculate the minimax risk defined as
\begin{equation}\label{eq::sparse-minimax}
    R(\Theta(k), \sigma) := \inf_{\hat{\beta}} \sup_{ \beta \in \Theta(k)} \E_{\beta} \|\hat{\beta} - \beta\|_2^{2}.
\end{equation}
Here, $\mathbb{E}_{\beta}(\cdot)$ is the expectation taken under the model defined in \eqref{model::gaussian-model} with the true parameter vector $\beta$ being fixed, and $\|\cdot\|_2$ denotes the Euclidean norm. Obtaining the exact minimax risk is mathematically challenging and has remained open. Hence, researchers have explored approaches that aim to approximate the minimax risk. One major approach focuses on the rate-optimal minimaxity with the goal to derive the order of $R(\Theta(k), \sigma)$. A representative result implied by the work \cite{verzelen2012minimax} states that as $k/p \rightarrow 0$ and $(k\log(p/k))/n  \rightarrow 0$,
\begin{align}
\label{rate:optimal}
R(\Theta(k), \sigma) \sim  \sigma^2 k \log (p/k),
\end{align}
where the notation ``$a_n\sim b_n$" means $a_n/b_n$ remains bounded away from zero and infinity. Furthermore, it has been shown in the literature \cite{bickel2009simultaneous, raskutti2011minimax, verzelen2012minimax, bellec2018slope} that many estimators, such as best subset selection \cite{hocking1967selection, beale1967discarding}, Dantzig selector \cite{candes2007dantzig} and Lasso \cite{tibshirani1996regression}, achieve this minimax optimal rate, meaning that their supremum risks (under optimal tuning) divided by $\sigma^2 k \log (p/k)$ remain bounded\footnote{In some of these results, the risk is stated with high probability and the rate is $k\log p$ instead of $k
\log(p/k)$.}. 

\begin{figure}[t!]
    \centering
    \begin{subfigure}[b]{0.495\linewidth}
        \centering
        \includegraphics[width=\linewidth]{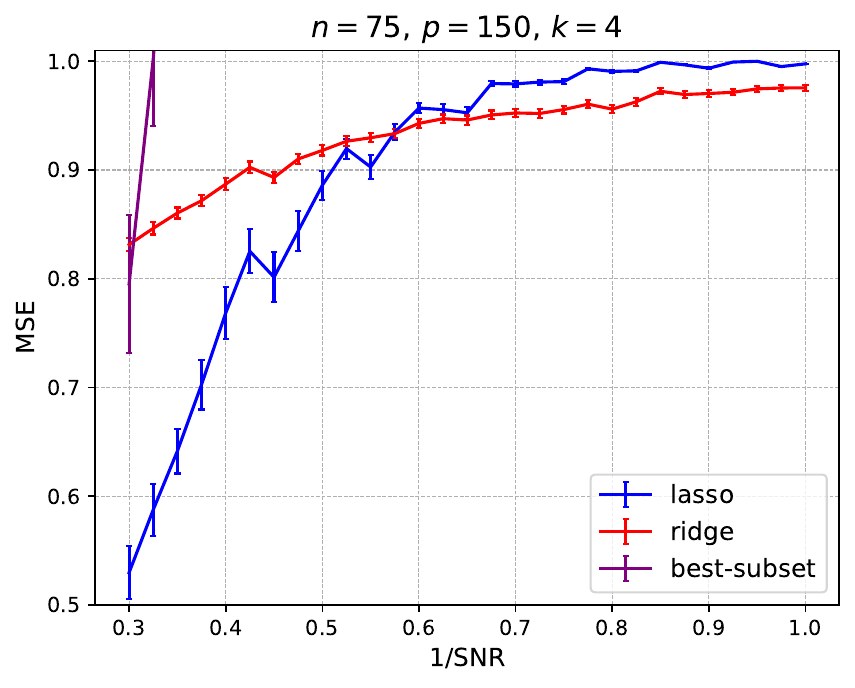}
    \end{subfigure}
    \hfill
    \begin{subfigure}[b]{0.495\linewidth}
        \centering
        \includegraphics[width=\linewidth]{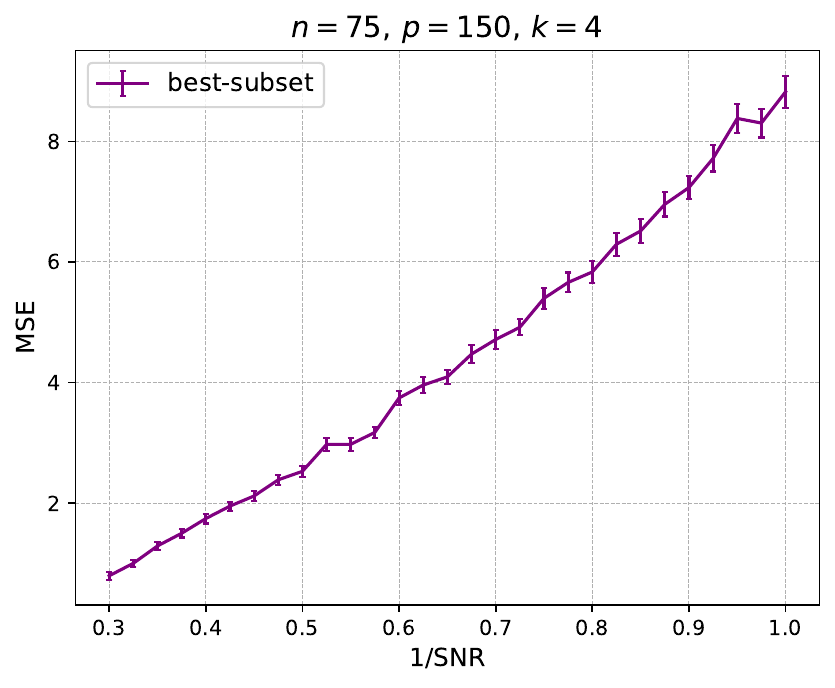}
    \end{subfigure}
    \caption{Mean-squared error comparison at different SNR values. Data is generated according to~\eqref{model::gaussian-model}. MSE is defined as the average of $\|\hat{\beta}-\beta\|_2^2/\|\beta\|_2^2$ over 50 experiments. We set $\beta$ as, for a randomly sampled index set \(S\) of cardinality \(k\), \(\beta_i = \tau \,\forall i\in S; \, \beta_i=0\) otherwise. SNR \(\coloneqq \tau/\sigma\). More simulation results are provided in Section \ref{simulation:sec}. To facilitate a clearer comparison between ridge, LASSO, and subset selection, we limit the MSE range to \([0, 1]\) in the left panel. However, the right panel displays the MSE of best subset selection across the full range of \(1/{\rm SNR}\). }
    \label{fig:bss}
\end{figure}


Despite the minimax rate optimality of the aforementioned estimators, extensive simulation results reported in \cite{hastie2020best, wang2020bridge} have confirmed that when the signal-to-noise ratio (SNR) is low, the best subset selection and Lasso are both consistently outperformed by the ridge regression estimator. See also Figure \ref{fig:bss} for a quick demonstration. Therefore, the rate-optimal minimax results can lead to misleading guidelines for practitioners. 

There could be two explanations for the mismatch between the result of rate-optimal minimaxity and the simulation studies:

\begin{itemize}
    \item Explanation 1: As is clear, the rate-optimal minimax results do not evaluate the minimax risk exactly. It ignores the constant in the minimax risk approximation and only captures the rate behavior in view of $k$ and $p$ for mathematical simplicity. If we can calculate the sharp constant in the minimax risk and supremum risk of estimators, the differences between constants may explain the discrepancies between the simulation studies and the rate-optimal minimax results.   

    \item Explanation 2: It could be that since the minimax framework only focuses on the spots of the parameter space that are hard for the estimation problem, it cannot capture the intricacies that happen with the variations of the signal-to-noise ratio. Hence, the framework needs to be amended to provide more informative results.  
\end{itemize}

 To pinpoint the correct explanation for the discrepancy between the minimax and simulation results, we first cite  a result that obtains a more accurate approximation of the minimax risk.

\begin{theorem}[\cite{guo2024minimaxlr}]\label{thm::first-order-sparse-minimax}
    Assume model \eqref{model::gaussian-model} and parameter space \eqref{param::sparse}. As $k/p \rightarrow 0$ and $\big(k \log p\big)/n \rightarrow 0$, the minimax risk defined in \eqref{eq::sparse-minimax} satisfies\footnote{The high-probability version of this result appeared earlier in \cite{su2016slope}.}
    \begin{equation}
    \label{non:snr:first}
        R(\Theta(k), \sigma) = 2\sigma^{2}k\log(p/k) \big(1+o(1) \big).
    \end{equation}
This minimax risk is (asymptotically) achieved by switching between the best subset selection and Lasso based on a certain switching rule.
\end{theorem}

 Compared to the rate-optimal minimax result in \eqref{rate:optimal}, Theorem \ref{thm::first-order-sparse-minimax} has the advantage of characterizing the (asymptotically) exact constant of the minimax risk. However, it still suffers from the same issue as the rate-optimal minimaxity. According to Theorem \ref{thm::first-order-sparse-minimax}, the estimation procedure using the best subset selection and Lasso is minimax optimal irrespective of the signal-to-noise ratio.  This implies that Explanation 1 does not offer the right reason.

\subsection{Our approach: SNR-aware minimaxity}

In light of Explanation 2, as will be clarified later in our paper, there are two main issues causing the discrepancy between the theoretical and simulation results: 
\begin{enumerate}
\item Since the parameter space \eqref{param::sparse} does not impose any constraint on the signal strength, the minimax analysis only focuses on a particular signal-to-noise ratio that makes the estimation problem the hardest. As a result, the factor of SNR affecting practical results is masked by the minimax framework.

\item The approximations we obtain for the minimax risk in rate-optimal minimaxity, and even in Theorem \ref{thm::first-order-sparse-minimax} are not accurate enough for distinguishing performances of different estimators.

\end{enumerate}

Our paper makes two main contributions, each addressing one of the issues outlined above. To address the first issue, we will incorporate the notion of SNR into the minimax analysis, and develop an SNR-aware minimaxity. To address the second issue, our paper considers and analyzes a more accurate higher-order approximation of the minimax risk. These two contributions will create a more insightful minimax framework that not only can offer results consistent with the existing simulation studies but also provides new insights to the estimation of sparse signals. The remainder of the paper is organized as follows. Section \ref{sec::snr-aware-minimaxity} discusses in detail the proposed SNR-aware minimax framework. Section \ref{simulation:sec} presents some simulations to support our theoretical findings. Section \ref{discussion:sec} discusses related works and gives some concluding remarks. All the proofs are relegated to the supplement.

\subsection{Notations}
We collect the notations used throughout the paper for convenience. For a scalar $a\in \mathbb{R}$, $a_+$ and ${\rm sign}(a)$ denote the positive part of $a$ and its sign respectively. For an integer $n$, let $[n]=\{1,2,\ldots, n\}$. We use $\mathbbm{1}_{\mathcal{A}}$ to represent the indicator function of the set $\mathcal{A}$. For a given vector $v = (v_{1}, \ldots, v_{p})\in \mathbb{R}^{p}$, $\norm{v}_{0} = \# \{i: v_{i}\neq 0\}$, $\norm{v}_{\infty} = \max_{i}|v_{i}|$, $\norm{v}_{q} = \left(\sum_{i=1}^{p} |v_{i}|^{q}\right)^{1/q}$ for $q \in (0,\infty)$. The inner product of two vectors $a,b$ is written as $\langle a, b\rangle$. We use $\{e_j\}_{j=1}^p$ to denote the natural basis in $\mathbb{R}^p$. For a matrix $X\in \mathbb{R}^{n\times p}$, $X_j$ represents its $j$th column and $X_S\in \mathbb{R}^{n\times |S|}$ is the submatrix consisting of columns indexed by $S\subseteq [p]$. The $p\times p$ identity matrix is denoted by $I_p$. For two real numbers $a$ and $b$, $a \vee b$ and $a \wedge b$ represent $\max(a, b)$ and $\min(a, b)$, respectively. For two non-zero real sequences $\{a_n\}_{n=1}^{\infty}$ and $\{b_n\}_{n=1}^{\infty}$, we use $a_n = o(b_n)$ (or $a_n \ll b_n$) to represent $|a_n/b_n| \rightarrow 0$ as $n \rightarrow \infty$, and $a_n=\omega(b_n)$ (or $a_n \gg b_n$) if and only if $b_n=o(a_n)$; $a_n = O(b_n)$ means $\sup_n|a_n/b_n| < \infty$ and $a_n=\Omega_n(b_n)$ if and only if $b_n=O(a_n)$. The notation $x\overset{d}{=}y$ means that the random variables $x$ and $y$ have the same distribution. For a random vector $x$, the notation $\|x\|_{\psi_2}$ denotes its sub-Gaussian norm. Finally, we reserve the notations $\phi(y)$ and $\Phi(y)$ for the standard normal density and its cumulative distribution function respectively.



\section{SNR-aware minimaxity}
\label{sec::snr-aware-minimaxity}

\subsection{Definition of the SNR-aware framework}
\label{snr:minimax:def}
As discussed in the previous section, one of the main reasons that the existing minimax framework produces misguidance for practitioners, is that the signal strength is not controlled and hence is set to a level that makes the estimation problem the hardest. As a result, the classical minimax approach, in an indirect way, becomes blind to the changes in the signal-to-noise ratio. To develop the SNR-aware minimax framework, we start by inserting a notion of signal-to-noise ratio in the minimax setting. To this end, we consider the following SNR-aware parameter space:
\begin{equation}\label{param::snr_aware}
    \Theta(k,\tau) := \Big\{ \beta \in \R^{p}: \|\beta\|_{0}\leq k, ~ \|\beta\|_{2}^{2} \leq k \tau^{2} \Big\}.
\end{equation}
The new parameter $\tau$, viewed as a constraint on the averaged magnitude of non-zero signal components, is a measure of signal strength. Compared to the vanilla sparse parameter space in \eqref{param::sparse}, $\Theta(k,\tau)$ can monitor the changes in SNR. Hence, the minimax framework we will develop with this parameter space can reveal the impact of the SNR on the sparse linear regression problem.

Given this new parameter space, the corresponding minimax risk is defined as 
\begin{equation}\label{eq::snr_minimax}
    R(\Theta(k,\tau), \sigma) := \inf_{\hat{\beta}} \sup_{\beta \in \Theta(k,\tau)} \E_{\beta} \|\hat{\beta} - \beta\|_2^{2}.
\end{equation}
As mentioned in the last section, it has remained open to characterize the exact minimax risk for $R(\Theta(k), \sigma)$. Given that it is even more challenging to obtain exact $R(\Theta(k,\tau), \sigma)$, instead we pursue asymptotically accurate approximations for this quantity. Following a standard way of modeling sparse signals \cite{donoho1992maximum, su2016slope, guo2022signaltonoise}, we consider the sparsity parameter defined as
\[
\epsilon := \frac{k}{p},
\]
and assume that $\epsilon \rightarrow 0$ as $n, p \rightarrow \infty$. Given that we have introduced the notion of signal strength $\tau$ to our framework, we expect the SNR level, defined as 
\[
\mu := \frac{\tau}{\sigma},
\]
to play an important role. Our goal is to derive accurate asymptotic approximations of $R(\Theta(k,\tau), \sigma)$ and discover optimal estimators, under different levels of SNR (i.e. different values of $\mu$). This leads us to consider the following regimes: as $n,p \rightarrow \infty$, and $\epsilon\rightarrow 0$ 
\begin{itemize}[label={}]
\item \textbf{Regime (\rom{1})}  Low signal-to-noise ratio: $\mu\rightarrow 0$; 
\item \textbf{Regime (\rom{2})} Moderate signal-to-noise ratio: $\mu\rightarrow \infty$, $\mu=o(\sqrt{\log \epsilon^{-1}})$;
\item \textbf{Regime (\rom{3})} High signal-to-noise ratio: $\mu=\omega(\sqrt{\log \epsilon^{-1}})$.
\end{itemize}
We should emphasize that all the quantities $\{k, p, \sigma, \tau\}$ are allowed to scale with $n$ in the above asymptotic regimes. We have suppressed the dependency on $n$ for notational simplicity. As will be shown in Section \ref{second:order:sec}, each regime exhibits unique minimaxity, and distinct minimax optimal estimators emerge in different regimes. Such fine-grained characterization requires delicate higher-order asymptotic analyses. To shed light on this point, in the next section, we first obtain the usual first-order asymptotic approximation and reveal its limitations in the SNR-aware minimax framework.

\subsection{First-order asymptotic approximation}

In this section, we present our first approximation for the minimax risk $R(\Theta(k,\tau), \sigma)$, in the three asymptotic regimes specified in Section \ref{snr:minimax:def}. After discussing the weaknesses of this approximation, in the next section, we will show how the approximation can be improved.

\begin{theorem}\label{thm::first-order-snr-minimax}
    Assume model \eqref{model::gaussian-model} and parameter space \eqref{param::snr_aware}. Recall the SNR level $\mu=\frac{\tau}{\sigma}$ introduced in Section \ref{snr:minimax:def}. As $k/p \rightarrow 0$, $(k \log (p))/n\rightarrow 0$, the following hold:
    \begin{itemize}
        \item Regime $(\rom{1})$: $\mu \rightarrow 0$, 
        \begin{equation*}
            R(\Theta(k,\tau), \sigma) = k\tau^{2} \Big( 1+o(1)\Big).
        \end{equation*}
        \item Regime $(\rom{2})$: $\mu\rightarrow \infty$ and $\mu = o\big(\sqrt{\log (p/k)} \big)$,
        \begin{equation*}
            R(\Theta(k,\tau), \sigma) = k\tau^{2} \Big( 1+o(1) \Big).
        \end{equation*}
        \item Regime $(\rom{3})$: $\mu = \omega\big(\sqrt{\log (p/k)} \big)$,
        \begin{equation*}
            R(\Theta(k,\tau), \sigma) = 2\sigma^2k\log(p/k) \Big(1+o(1)\Big).
        \end{equation*}
    \end{itemize}
\end{theorem}

The proof of Theorem \ref{thm::first-order-snr-minimax} can be found in Section \ref{sec::first-order-snr-minimax}.

\begin{remark}
 An intriguing aspect of Theorem \ref{thm::first-order-snr-minimax} is that it shows the difference of the minimax risk between Regimes I-II and Regime III, thus capturing the SNR effect to some extent. More specifically, the minimax result \eqref{non:snr:first} of Theorem \ref{thm::first-order-sparse-minimax} closely resembles that of Regime III in Theorem \ref{thm::first-order-snr-minimax}. This implies that if we do not consider any constraint on the signal strength, the minimax framework will set the SNR to a value that falls within Regime III. That is why Theorem \ref{thm::first-order-sparse-minimax} is agnostic to SNR. On the other hand, the minimax risk in fact undergoes a phase transition as SNR decreases from Regime III to Regimes I-II. Hence, a minimax analysis that is not considering the SNR will only offer a partial perspective, potentially leading to incorrect conclusions.  
\end{remark}

\begin{remark}
Another interesting feature of Theorem \ref{thm::first-order-snr-minimax} is that the minimax risk it provides in Regimes I-II enables us to evaluate the performance of estimators in moderate and low SNRs. Consider the well-known best subset selection estimator:
\begin{align}\label{eq:MLE_1}
\hat{\beta}^{BS} \in \argmin_{b: \|b\|_0 \leq k} \sum_{i=1}^n(y_i-x_i^Tb)^2. 
\end{align}
We iterate that the existing minimax framework has showed $\hat{\beta}^{BS}$ is minimax rate-optimal \cite{raskutti2011minimax}, while simulation studies \cite{hastie2020best} demonstrated its suboptimal performance when SNR is not high. The following proposition characterizes the suboptimality of $\hat{\beta}_{\rm BSS}$ in moderate-SNR and low-SNR regimes, thus offering a better explanation for the related empirical results in \cite{hastie2020best}.
\end{remark}

\begin{proposition}
\label{sub:bss}
    Under the same conditions of Theorem \ref{thm::first-order-snr-minimax}, in Regimes I-II where $\tau/\sigma=o(\sqrt{\log(p/k)})$, the best subset selection $\hat{\beta}^{BS}$ satisfies
    \[
    \frac{\sup_{\beta\in \Theta(k,\tau)}\mathbb{E}_{\beta} \|\hat{\beta}^{BS}-\beta \|_2^2 }{R(\Theta(k, \tau),\sigma)} \rightarrow \infty. 
    \]
\end{proposition}

The proof of Proposition \ref{sub:bss} is presented in Section \ref{proof:sub:bss}.

While Theorem \ref{thm::first-order-snr-minimax} reveals some interesting aspects of the SNR impact on the minimax risk, the first-order asymptotic approximations it obtains do not seem to always offer accurate information. In particular, as the signal-to-noise ratio significantly increases from Regime \rom{1} to Regime \rom{2}, the first-order analysis falls short of capturing any difference. In both cases, the approximated minimax risk remains $k \tau^2$ which is the supremum risk of the naive zero estimator. As will be shown in the following section, this issue arises because the current approximation for the minimax risk is not sufficiently accurate, and our objective is to present a second-order analysis to reach accurate enough approximations.

\subsection{Second-order asymptotic approximation}
\label{second:order:sec}

As we discussed in the previous section, the first-order asymptotic approximation of the SNR-aware minimax risk is not informative enough to distinguish between Regime \rom{1} (low SNR) and Regime \rom{2} (moderate SNR). Moreover, it leads to the peculiar conclusion that the naive zero estimator is optimal in both regimes. To resolve this issue, in this section, we aim to unpack the $o(1)$ term of the first-order asymptotic approximation to derive the second-order approximation of the minimax risk, in Regimes \rom{1} and \rom{2}. These higher-order expansions, by offering more accurate approximations, not only enable us to see the impact of the SNR in the first two regimes, but also provide new and more accurate information about the optimal estimators. In this way, the second-order asymptotics makes the SNR-aware minimax framework fully effective.

\subsubsection{Second-order approximation in Regime \rom{1}}
\label{subsec::ridge}

Our first result presents a more accurate approximation of the SNR-aware minimax risk in Regime I. Define the ridge regression estimator \cite{hoerl1970ridge}:
\begin{equation}
\label{ridge:def:1970}
    \betaR(\lambda) \in \argmin_{b\in \R^{p}} \sum_{i=1}^n(y_i-x_i^Tb)^2 + \lambda \|b\|_{2}^{2},
\end{equation}
where $\lambda \geq 0$ is the tuning parameter.

\begin{theorem}\label{thm::second-order-low-snr-minimax}
    Assume model \eqref{model::gaussian-model} and parameter space \eqref{param::snr_aware}. Suppose $k/p \rightarrow 0$ and $k/n \rightarrow 0$. In Regime $\rom{1}$ where $ \tau/\sigma \rightarrow 0$, the minimax risk defined in \eqref{eq::snr_minimax} satisfies
    \begin{equation*}
        R(\Theta(k,\tau),\sigma) = k \tau^{2} \Big( 1 - \frac{k\tau^{2}}{p\sigma^2} \big(1+o(1)\big) \Big).
    \end{equation*}
    In addition, the ridge estimator $\hat{\beta}^{R}(\lambda)$ with tuning $\lambda=p\sigma^2/(k\tau^{2})$ is asymptotically minimax optimal up to the second order, i.e.
\[
 \sup_{\beta \in \Theta(k,\tau) } \mathbb{E}_{\beta} \|\hat{\beta}^{R} (\lambda) - \beta\|_2^{2} = k \tau^{2} \Big( 1 - \frac{k\tau^{2}}{p\sigma^2} \big(1+o(1)\big) \Big).
\]
\end{theorem}
The proof of this theorem is presented in Section \ref{proof:reg1:ridge}. The condition $k/n\rightarrow 0$ is very mild in the high-dimensional sparse regression (even weaker than the standard condition $(k\log p)/n\rightarrow 0$). This theorem provides theoretical insights not covered by Theorem \ref{thm::first-order-snr-minimax}, as we clarify below.

\begin{remark}
The second-order approximation of the minimax risk, which is more accurate than the first-order approximation, offers much more refined information. First, we start seeing the impact of the SNR ($\mu=\tau/\sigma)$ in the minimax risk. The higher the SNR, the larger the absolute value of the second-order term, and the lower the minimax risk will be. Second, we start seeing the difference between different estimators. We see that the ridge estimator outperforms the zero estimator and is in fact the optimal estimator. 
\end{remark}

\begin{remark}
The optimality of the ridge estimator established in Theorem \ref{thm::second-order-low-snr-minimax} provides a sound theoretical justification, for the superior performance of the ridge estimator (compared with estimators such as best subset selection and Lasso) in low-SNR scenarios as reported elsewhere based on simulation studies \cite{hastie2020best, wang2020bridge}.
\end{remark}

\subsubsection{Second-order approximation in Regime \rom{2}}
\label{subsec::enet}
Define the following elastic-net regularized estimator \cite{zou2005regularization}:
\begin{align}
\label{elastic:reg:est}
\hat{\beta}^E(\lambda,\gamma)\in \argmin_{b\in \mathbb{R}^p}\Big\|\sum_{i=1}^ny_ix_i-b\Big \|_2^2+\lambda \|b\|_1+\gamma \|b\|_2^2,
\end{align}
where $\|\cdot\|_1$ is the $\ell_1$-norm, and $\lambda,\gamma\geq 0$ are the tuning parameters.

\begin{theorem}\label{thm::second-order-med-snr-minimax}
    Assume model \eqref{model::gaussian-model} and parameter space \eqref{param::snr_aware}. Consider Regime $\rom{2}$ where $k/p \rightarrow 0, \tau/\sigma \rightarrow \infty$ and $\tau/\sigma = o(\sqrt{\log (p/k)})$. 
    \begin{itemize}
    \item[(i)] If $(\log(p/k))/n \rightarrow 0$ and $\tau^4/(n\sigma^4)\rightarrow 0$, the minimax risk defined in \eqref{eq::snr_minimax} satisfies
    \begin{equation}
    \label{mmlower:reg:II}
        R(\Theta(k,\tau),\sigma) \geq k \tau^{2} \Big( 1 -\frac{1+o(1)}{2}\cdot \frac{k}{p} \cdot e^{\frac{\tau^2}{\sigma^2}} \Big).
    \end{equation}
    \item[(ii)] If $(k(\log(p/k))^2)/n \rightarrow 0$, and there exists a constant $\alpha>0$ such that $\left(\frac{p}{k} \right)^{\alpha} \leq n$, then the estimator $\hat{\beta}^E(\lambda,\gamma)$ with tuning $\lambda=4\tau,\gamma=\frac{p\sigma^2}{2k\tau^2}e^{-\frac{3\tau^2}{2\sigma^2}}-1$ satisfies
    \begin{equation}
    \label{mmupper:reg:II}
   \sup_{\beta\in \Theta(k,\tau)}\mathbb{E}_{\beta} \|\hat{\beta}^E(\lambda,\gamma)-\beta \|_2^2 \leq k \tau^{2} \Big( 1 - \frac{2+o(1)}{\sqrt{2\pi}}\cdot \frac{k}{p} \cdot \frac{\sigma}{\tau}e^{\frac{\tau^2}{\sigma^2}} \Big).
    \end{equation}
    \end{itemize}
\end{theorem}

The proof of this theorem can be found in Section \ref{proof:thm:regime2}. 

\begin{remark}
The condition required in the lower bound of Theorem \ref{thm::second-order-med-snr-minimax} can be easily satisfied if we make the weak assumption $(\log(p/k))^2/n\rightarrow 0$, because Regime II considers $\tau^2/\sigma^2 \ll \log(p/k)$. The condition $(k(\log(p/k))^2)/n \rightarrow 0$ in the upper bound is also mild, since it is comparable to the standard condition $(k\log(p/k))/n \rightarrow 0$ in the sparse linear regression literature. The other condition $\left(\frac{p}{k} \right)^{\alpha} \leq n$ is stronger, as it rules out the possibility of exponential growing of $p$ with $n$. That said, it allows $p$ to grow with $n$ at any polynomial rate. 
\end{remark}

\begin{remark}
In Theorem \ref{thm::second-order-med-snr-minimax}, the upper and lower bounds do not exactly match. However, the gap is only up to an order of $\tau/\sigma$ in the second-order term, which is very small given the presence of the exponentially large term $e^{\tau^2/\sigma^2}$. Hence, the elastic-net regularized estimator $\hat{\beta}^E(\lambda,\gamma)$ can be concluded to be nearly minimax optimal in Regime II. As in Regime I, being much more informative than the first-order approximation, the second-order approximation reveals the impact of SNR on the risk. As SNR increases from Regime I to Regime II, the second-order term in the minimax risk becomes smaller implying that the zero estimator can be outperformed by a larger margin. 
\end{remark}

\begin{remark}
Compared to the optimal estimator $\hat{\beta}^R(\lambda)$ in Regime I, the (nearly) minimax optimal estimator $\hat{\beta}^E(\lambda,\gamma)$ in Regime II employs $\ell_1$-regularization in addition to the $\ell_2$-regularization. To shed light on the importance of $\ell_1$-regularization in Regime II, the following proposition reveals the suboptimality of $\hat{\beta}^R(\lambda)$ in this regime. 
\end{remark}

\begin{proposition}\label{prop::ridge-suboptimality-second-regime}

    Assume model \eqref{model::gaussian-model} and parameter space \eqref{param::snr_aware}. In Regime $\rom{2}$ where $k/p \rightarrow 0, \tau/\sigma \rightarrow \infty$ and $\tau/\sigma = o(\sqrt{\log (p/k)})$, if $(k\log(p/k))/n \rightarrow 0$, the ridge estimator \(\hat{\beta}^R(\lambda)\) defined in \eqref{ridge:def:1970} is suboptimal in the minimax sense, because
    \begin{equation*}
        \inf_{\lambda>0}\sup_{\beta \in \Theta(k,\tau)}\E\|\hat{\beta}^R(\lambda) - \beta\|_2^2 \ge k \tau^{2} \Big( 1 - \frac{k\tau^{2}}{p\sigma^2} \big(1+o(1)\big) \Big).
    \end{equation*}
\end{proposition}
The proof of this proposition can be found in Section \ref{ssec:proof:ridge:suboptimal}.   

\begin{remark}
Proposition \ref{prop::ridge-suboptimality-second-regime} shows that the supremum risk of (optimally tuned) ridge estimator has a much larger second-order term than the minimax risk $R(\Theta(k,\tau),\sigma)$, hence it becomes suboptimal in Regime II. This proposition together with Theorems \ref{thm::second-order-low-snr-minimax} and \ref{thm::second-order-med-snr-minimax} provide insights into the estimation of sparse signals when the SNR level is not high: (1) In low-SNR scenarios where the variance dominates the estimation error, $\ell_2$-regularization is (minimax) optimal even though it does not produce sparse estimators; (2) In moderate-SNR scenarios in which the sparse signal strength becomes stronger, adding the sparsity-promoting $\ell_1$-regularization helps attain optimality.  
\end{remark}

\section{Simulations} \label{simulation:sec}

Sections \ref{intro:label} and \ref{sec::snr-aware-minimaxity} discussed the inadequacy of classical minimax results for characterizing the role of signal-to-noise ratio (SNR), and then developed the SNR-aware minimax framework. The simulation results in this section will offer further support to the theoretical results above. 

For our simulation studies, we generate the signal $\beta\in \mathbb{R}^p$ in the following way: for a randomly sampled index set \(S\) of cardinality \(k\), \(\beta_i = \tau \,\forall i\in S; \, \beta_i=0\) otherwise. The design matrix \(X \in \mathbb{R}^{n \times p}\) is generated such that each row $x_i$ is sampled independently from $\mathcal{N}(0, \frac{1}{n}I_p)$. Then \(y \in \mathbb{R}^n\) is generated as \(y_{i} = x_{i}^{T}\beta + \sigma z_{i}, \, i=1, \ldots, n\), where \(z_i\)'s are the standard normal errors independent of \(X\). The signal-to-noise ratio (SNR) is defined as SNR $:=\frac{\tau}{\sigma}$. We fix $\tau$ and change the value of $\sigma$ to vary the SNR level. The (scaled) MSE is defined as the average of \(\|\hat{\beta}-\beta\|_2^2/\|\beta\|_2^2\) over 150 repeated experiments.

We consider the following four estimators in our experiments.
\begin{enumerate}
\item The best subset selection estimator in \eqref{eq:MLE_1}:
\begin{align*}    
\hat{\beta}^{BS} \in \argmin_{b: \|b\|_0 \leq k} \sum_{i=1}^n(y_i-x_i^Tb)^2. 
\end{align*}
Computing the best subset selection estimator is an NP-hard optimization problem, and certifying the optimality of the solution is also of exponential time complexity. We use the branch-and-bound style algorithm proposed by~\cite{hastie2020best} for our implementation. We have only considered experiments where the algorithm provides a certification of \emph{optimality} of the computed estimator. The performance of best subset selection estimator in our plots is denoted in purple with the caption `best-subset'.
    \item The Lasso estimator:
\begin{equation*}
    \beta^L(\lambda) \in \argmin_{b\in \R^{p}} \sum_{i=1}^n(y_i-x_i^Tb)^2 + \lambda \|b\|_{1},
\end{equation*}
where \(\lambda \ge 0\) is a tuning parameter. It is implemented by the \texttt{glmnet} package in R. The performance of the Lasso estimator in our plots is denoted in blue with the caption `lasso'. 
    \item The elastic-net regularized estimator in \eqref{elastic:reg:est}:
\begin{align*}
\hat{\beta}^E(\lambda,\gamma)\in \argmin_{b\in \mathbb{R}^p}\Big\|\sum_{i=1}^ny_ix_i-b\Big \|_2^2+\lambda \|b\|_1+\gamma \|b\|_2^2,
\end{align*}
    where $\lambda$ and $\gamma$ are the tuning parameters. The implementation is according to the closed-form expression provided in \eqref{els:reform}. The performance of the elastic-net regularized estimator in our plots is denoted in green with the caption `enet'. 
    \item The ridge regression estimator in \eqref{ridge:def:1970}:
\begin{equation*}
    \betaR(\lambda) \in \argmin_{b\in \R^{p}} \sum_{i=1}^n(y_i-x_i^Tb)^2 + \lambda \|b\|_{2}^{2},
\end{equation*}
where $\lambda \geq 0$ is the tuning parameter. The implementation is performed by computing the closed form $\hat{\beta}^R(\lambda) = ( X^TX + \lambda I_p )^{-1}X^Ty$. The performance of the ridge estimator in our plots is denoted in red with the caption `ridge'. 
\end{enumerate}

All the hyperparameters, e.g. $\lambda, \gamma$, are optimally tuned for a fair comparison. The experiments are run on 28 cores of Intel Xeon Gold 6226 2.9 Ghz CPUs, each CPU with 16 cores and memory 192 GB.\footnote{For code and implementation details, please see \href{https://github.com/shubhangighosh/SNR-aware-minimaxity-for-linear-regression}{https://github.com/shubhangighosh/SNR-aware-minimaxity-for-linear-regression}.}

We summarize our simulation results in the following:

\begin{itemize}
\item We set $n=500$ and \(p=500, 1000\) for different undersampling fractions. Two possible sparsity levels are considered for each setting. For \(p=500\), we consider the sparsity levels \(k=12,35\). For \(p=1000\), we consider the sparsity levels, \(k=25,68\). The best subset selection estimator is excluded in the comparison due to its computational infeasibility under the considered scales. Referring to Figure \ref{fig:snr-aware-minimax-full}, in each of the plots, we observe three regimes based on the SNR. In the high-SNR regime, the Lasso dominates the other estimators. There exists a medium-SNR regime, akin to the theoretical results in Section \ref{subsec::enet}, where our elastic-net regularized estimator dominates. Then, in the low-SNR regime, as consistent with the optimality result in Section \ref{subsec::ridge}, the ridge estimator performs the best.

\item To make the best subset selection estimator computationally feasible, we also consider smaller-scale experiments where $n=75, p=75, 150$. As shown in Figure \ref{fig:snr-aware-minimax-full-bss}, the best subset selection estimator outperforms the other estimators for very high SNRs. But as the SNR deceases, its performance quickly deteriorates. This is aligned with our theoretical characterization in Proposition \ref{sub:bss}. Moreover, in this small-scale setup, we are still able to observe the medium-SNR regime where the elastic-net regularized estimator performs best, and the low-SNR regime where the ridge estimator dominates.

\end{itemize}

\begin{figure}[t!]
    \centering
    \begin{subfigure}[b]{0.495\linewidth}
        \centering
        \includegraphics[width=\linewidth]{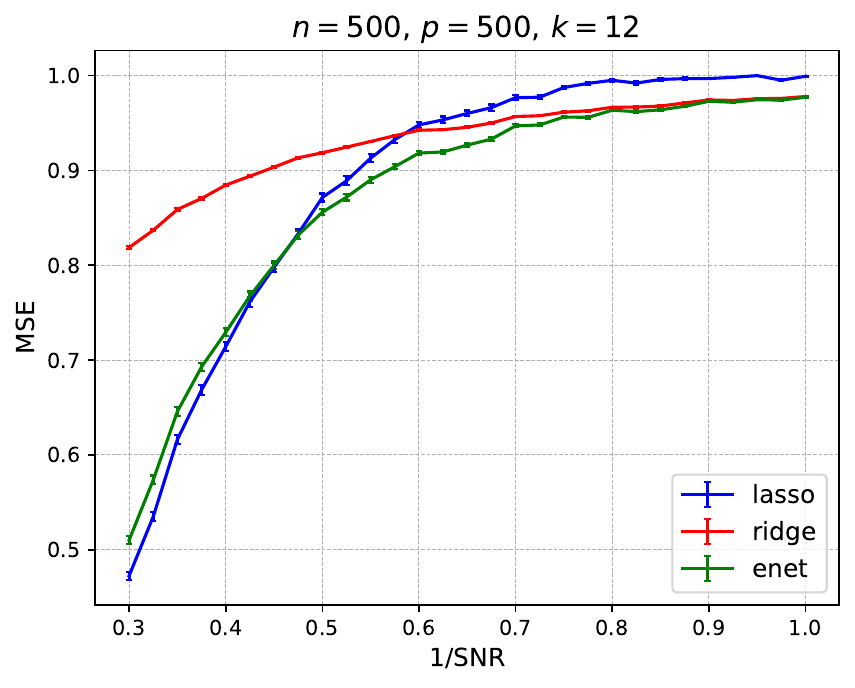}
    \end{subfigure}
    \hfill
    \begin{subfigure}[b]{0.495\linewidth}
        \centering
        \includegraphics[width=\linewidth]{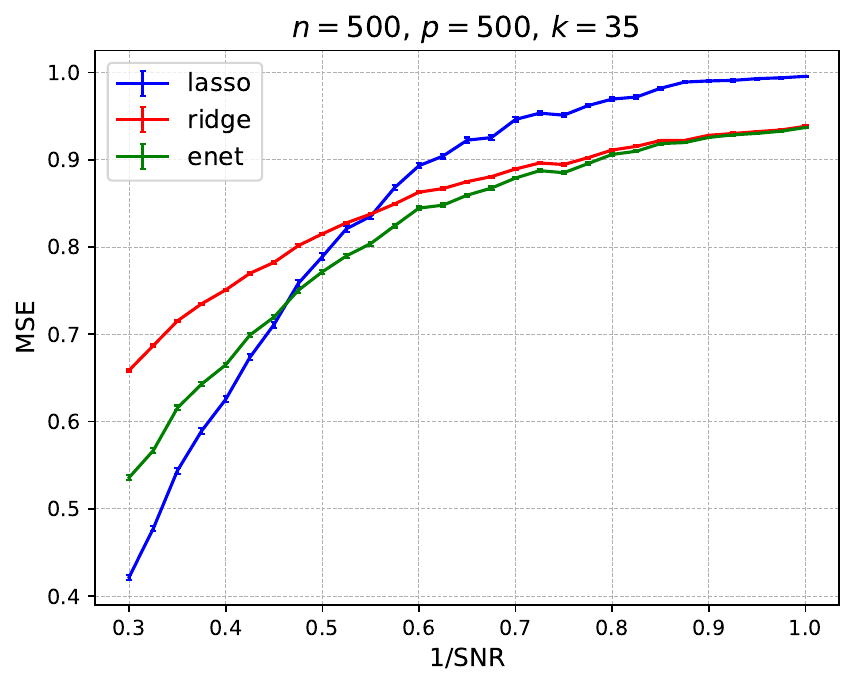}
    \end{subfigure}
    \vskip\baselineskip
    \begin{subfigure}[b]{0.495\linewidth}
        \centering
        \includegraphics[width=\linewidth]{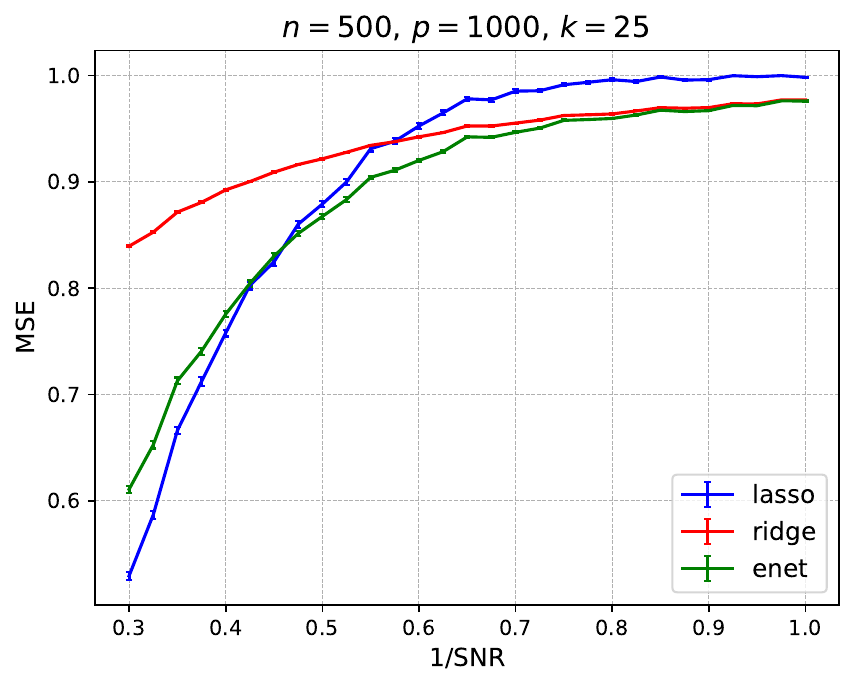}
    \end{subfigure}
    \hfill
    \begin{subfigure}[b]{0.495\linewidth}
        \centering
        \includegraphics[width=\linewidth]{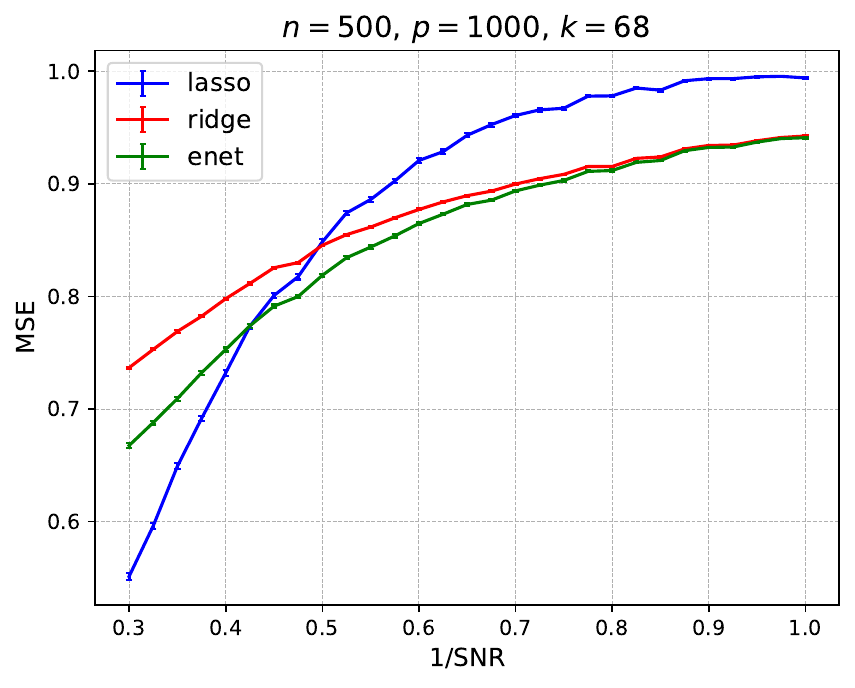}
    \end{subfigure}
    \captionsetup{justification=centering}
    \caption{Mean-squared error comparison at different SNR values. }
    \label{fig:snr-aware-minimax-full}
\end{figure}

\begin{figure}[t!]
    \centering
    \begin{subfigure}[b]{0.495\linewidth}
        \centering
        \includegraphics[width=\linewidth]{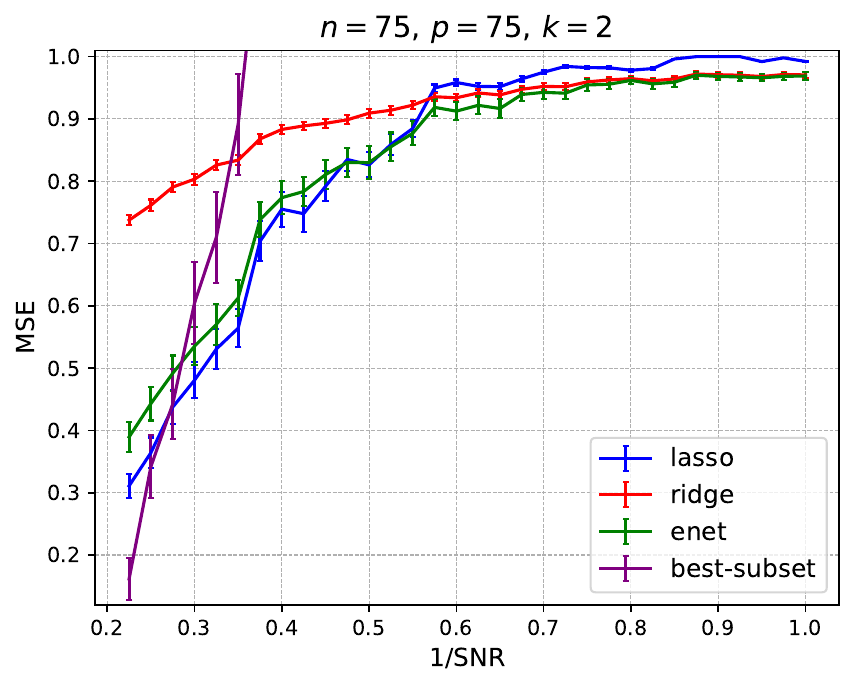}
    \end{subfigure}
    \hfill
    \begin{subfigure}[b]{0.495\linewidth}
        \centering
        \includegraphics[width=\linewidth]{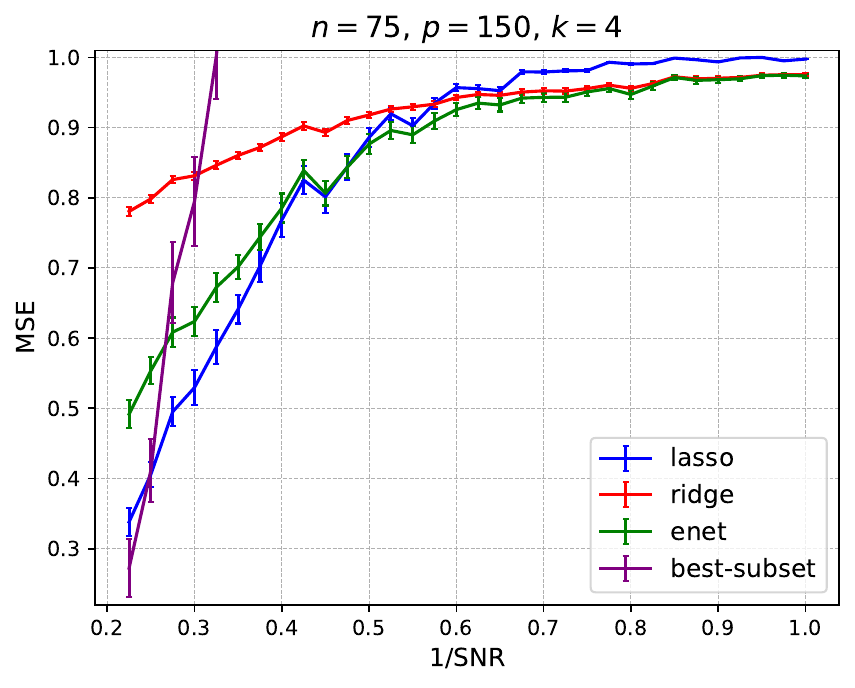}
    \end{subfigure}
    \vskip\baselineskip
    \begin{subfigure}[b]{0.495\linewidth}
        \centering
        \includegraphics[width=\linewidth]{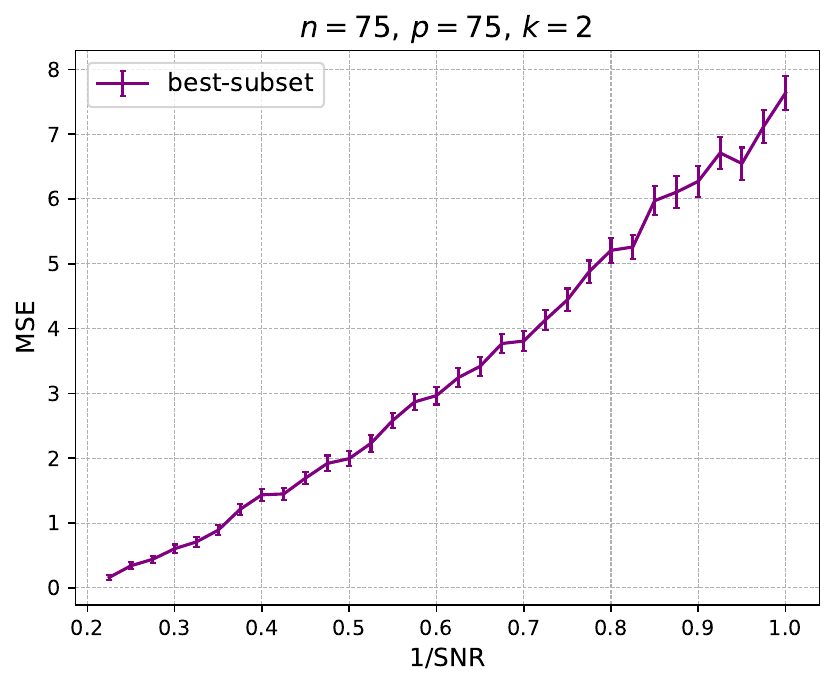}
    \end{subfigure}
    \hfill
    \begin{subfigure}[b]{0.495\linewidth}
        \centering
        \includegraphics[width=\linewidth]{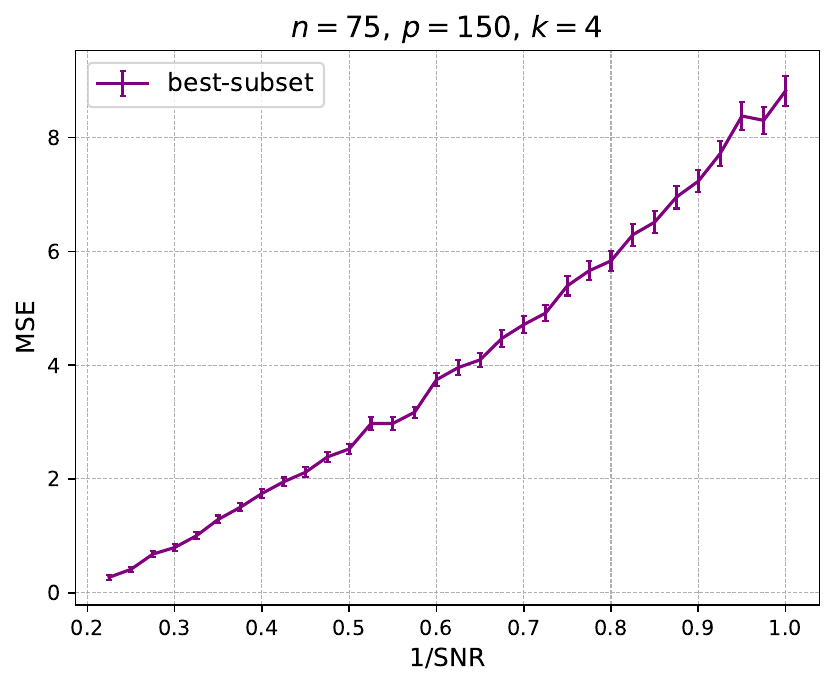}
    \end{subfigure}
    \captionsetup{justification=centering}
    \caption{Mean-squared error comparison at different SNR values. }
    \label{fig:snr-aware-minimax-full-bss}
\end{figure}
  

\section{Discussions} \label{discussion:sec}

\subsection{Related works}

There are some recent works on the significant role of signal-to-noise ratio (SNR) in the context of sparse linear regression. The extensive simulations conducted in \cite{hastie2020best} revealed, among other interesting results, that best subset selection generally performs better than Lasso in very high SNR regimes, while Lasso is better in low SNR regimes. By leveraging additional convex regularization, \cite{hazimeh2020fast, mazumder2023subset} developed new variants of best subset selection that perform consistently well in various SNR regimes. \cite{le2017does, weng2018overcoming, wang2020bridge} adopted the linear asymptotic framework where both the sparsity $k$ and sample size $n$ scale linearly with the dimension $p$, to establish constant-sharp theoretical characterizations of bridge regression (i.e. $\ell_q$-regularization) under varying SNR regimes. Interestingly, their theoretical findings discovered that among the family of $\ell_q$-regularization with $q \in [0,2]$, as SNR decreases from high to low levels, the optimal value of $q$ for parameter estimation and variable selection will move from 0 towards 2. \cite{wang2022does} employed a similar theoretical framework to compare the sorted-$\ell_1$ penalized estimator (SLOPE) \cite{bogdan2015slope} with bridge regression, and found that neither SLOPE nor bridge regression uniformly dominate the other across different SNR levels. All the aforementioned works studied the impact of SNR on several or a family of popular estimators, hence their comparison conclusions only apply to a restricted set of estimators. In contrast, our work focused on minimax analysis that led to stronger optimality-type conclusions.

As discussed in Section \ref{intro:label}, existing minimax results under sparse linear regression \cite{buhlmann2011statistics, hastie2015statistical, wainwright2019high, fan2020statistical} fall short of capturing the SNR effect. From this perspective, the most relevant work to the current paper is \cite{guo2022signaltonoise}. The authors of \cite{guo2022signaltonoise} adopted a similar SNR-aware minimax scheme for sparse Gaussian sequence model, and discovered three SNR regimes associated with different minimaxity. However, the theoretical analysis in our work is much more challenging than that in \cite{guo2022signaltonoise} for the following reasons: (1) Different from our work, all the estimators considered in \cite{guo2022signaltonoise} admit coordinate-separable and closed-form expressions, which largely facilitates the derivation of minimax upper bounds. (2) Due to the presence of random design matrix, the Bayes risks of independent block prior are much harder to compute to obtain sharp minimax lower bounds in our model. (3) The suboptimality results, regarding best subset selection in Proposition \ref{sub:bss} and ridge estimator in Proposition \ref{prop::ridge-suboptimality-second-regime}, require more delicate analysis since the estimators either don't have a closed form or don't enjoy coordinate-wise separability. Given that our results are (mostly) constant-sharp up to second order, it takes substantial efforts to overcome the aforementioned challenges.

\subsection{Conclusion}
In this paper, we revisit the classical problem of high-dimensional sparse linear regression. We present a significantly more informative minimax analysis than existing approaches, achieved through the implementation of two key ideas: (1) signal-to-noise ratio aware minimax framework, and (2) second-order approximation of the minimax risk. The former enables us to evaluate the impact of signal-to-noise ratio on minimax optimality, while the latter aids in obtaining highly accurate approximations of the minimax risk. The theoretical findings of the paper have offered insights inaccessible through classical minimax analysis. In particular, in scenarios of low or moderate signal-to-noise ratio, $\ell_2$-regularization plays a key role in optimal minimax estimation of sparse signals. This is intriguing, considering $\ell_2$-regularization does not promote sparsity.    

\printbibliography

\appendix


\include{preliminaries}

\include{proof_thm2_lower_bound}
\include{proof_thm3}
\include{proof_prop1}
\include{proof_thm4}

\include{proof_thm5}
\include{proof_prop_ridge_reg2}


\end{document}

%% file: premeable.tex
\pdfoutput=1

\usepackage{amsmath} 
\usepackage{parskip}
\usepackage{amsthm,amsfonts,amssymb,commath,bm,bbm, upgreek}
\usepackage{enumitem}
\usepackage{subcaption}
\usepackage{array}
\usepackage{mathtools}
\usepackage[affil-it]{authblk}
\usepackage[top=1 in,bottom=1in, left=1 in, right=1 in]{geometry}
\usepackage{hyperref}
\usepackage{cleveref}
\usepackage[backend=biber]{biblatex}
\usepackage[dvipsnames]{xcolor}

\allowdisplaybreaks

\makeatletter
\newcommand*{\rom}[1]{\expandafter\@slowromancap\romannumeral #1@}
\makeatother

\newcounter{relctr} 
\everydisplay\expandafter{\the\everydisplay\setcounter{relctr}{0}} 

\AtBeginDocument{} 

\newtheorem{remark}{Remark}
\newtheorem{theorem}{Theorem}
\newtheorem{lemma}{Lemma}
\newtheorem{proposition}{Proposition}

\newcommand\E{\mathbb{E}}
\newcommand{\betaR}{\hat{\beta}^{R}}

\newcommand{\betaTE}{\Tilde{\beta}^{E}}

\newcommand{\betahy}{\hat{\beta}^{E}}

\newcommand{\calN}{\mathcal{N}}
\newcommand{\R}{\mathbb{R}}
\newcommand{\Var}{{\rm Var}}

\newcounter{cnstcnt}

\newcounter{deltacnt}

\newcommand\simiid{\overset{i.i.d}{\sim}}

\newcommand\supp{\operatorname{supp}}

\DeclareMathOperator*{\argmin}{arg\,min}

\DeclareMathOperator*{\diag}{diag}

\usepackage{stackengine}

\def\square{\ifmmode\sqr\else{$\sqr$}\fi}
\def\sqr{\vcenter{
         \hrule height.1mm
         \hbox{\vrule width.1mm height2.2mm\kern2.18mm
\vrule width.1mm}
         \hrule height.1mm}}

\usepackage{color}



%% file: preliminaries.tex
\begin{center}
{\LARGE\bf Proofs of technical results}
\end{center}

\section{Preliminaries}
\subsection{Scale invariance} \label{sec:scaling}

The minimax risk defined in \eqref{eq::snr_minimax} of the main text has the following scale invariance property,
\begin{align*}
    R(\Theta(k,\tau),\sigma) = \sigma^{2} \cdot R(\Theta(k,\mu), 1),
\end{align*}
where we recall that $\mu=\tau/\sigma$. This can be easily verified by rescaling the linear regression model to have unit variance.

\subsection{Preliminary probability results}

\begin{lemma}[Exercise 8.1 in \cite{johnstone19}]\label{lem::gaussian-tail-mills-ratio}
Define
\begin{equation*}
    \Tilde{\Phi}_{l}(\lambda) := \lambda^{-1} \phi(\lambda) \sum_{k=0}^{l}\frac{(-1)^{k}}{k!} \frac{\Gamma(2k+1)}{2^{k}\lambda^{2k}},
\end{equation*}
where $\Gamma(\cdot)$ is the gamma function. Then, for each $k\geq 0$ and all $\lambda>0$:
\begin{equation*}
    \Tilde{\Phi}_{2k+1}(\lambda) \leq 1-\Phi(\lambda) \leq \Tilde{\Phi}_{2k}(\lambda).
\end{equation*}
\end{lemma}

\begin{lemma}\label{lem:incomp_quadexp_Gaussian}
Let $Z \sim \calN(0,1)$. Let $\beta> -\frac{1}{2}$ and $\alpha, \gamma\in \mathbb{R}$ be three fixed numbers. We then have
\[
\E \left({\rm e}^{\alpha Z -\beta Z^2 } \mathbbm{1}_{Z \leq \gamma} \right) =\frac{1}{\sqrt{2 \beta+1}} {\rm e}^{\frac{\alpha^2}{2 (2 \beta+1)}} \mathbb{P}\left(Z \leq \sqrt{2\beta+1} \Big(\gamma -\frac{\alpha}{2 \beta+1}\Big)\right). 
\]
\end{lemma}
\begin{proof}
We have
\begin{align*}
\E &\left({\rm e}^{\alpha Z -\beta Z^2 } \mathbbm{1}_{Z \leq \gamma} \right) = \frac{1}{\sqrt{2 \pi}} \int {\rm e}^{\alpha z - \beta z^2 - \frac{z^2}{2}} \mathbbm{1}_{z \leq \gamma} dz \nonumber \\
&= \frac{{\rm e}^{\frac{\alpha^2}{2 (2\beta+1)}}}{\sqrt{2 \pi}} \int {\rm e}^{-\frac{1}{2} (2 \beta+1) \Big(z - \frac{\alpha}{ 2 \beta+1}\Big)^2} \mathbbm{1}_{z \leq \gamma} dz \nonumber \\
&\overset{(a)}{=} \frac{{\rm e}^{\frac{\alpha^2}{2 (2\beta+1)}}}{\sqrt{2 \pi}} \int {\rm e}^{-\frac{1}{2} (2 \beta+1) x^2} \mathbbm{1}_{x \leq \gamma - \frac{\alpha}{2 \beta+1}} dx \nonumber \\
&\overset{(b)}{=} \frac{{\rm e}^{\frac{\alpha^2}{2 (2\beta+1)}}}{\sqrt{2 \pi} \sqrt{2 \beta +1}} \int {\rm e}^{-\frac{t^2}{2}} \mathbbm{1}_{t \leq \sqrt{2 \beta+1} (\gamma - \frac{\alpha}{2 \beta +1 })}dt \nonumber\\
&=\frac{1}{\sqrt{2 \beta+1}} {\rm e}^{\frac{\alpha^2}{2 (2 \beta+1)}} \mathbb{P}\left(Z \leq \sqrt{2\beta+1} \Big(\gamma - \frac{\alpha}{2 \beta+1}\Big)\right),
\end{align*}
where to obtain (a) we have chnaged the variable of integration to $x = z - \frac{\alpha}{2 \beta +1}$ and to obtain (b) we have chnaged the variable of integration to $t = \sqrt{2 \beta+1}x$. 
\end{proof}

The following lemma states a simple concentration for $\ell_2$-norm of standard multivariate Gaussian. The proof follows from the concentration of the Lipschitz function of Gaussians (Theorem 2.26 in \cite{wainwright2019high}) and that the $\ell_{2}$-norm is $1$-Lipschitz function.
\begin{lemma}\label{lem::chi-concentration}
    Let $z \sim \calN(0,I_n)$, then for every $t\geq 0$,
    \begin{equation*}
        \mathbb{P}\Big(\|z\|_2\leq(1+t)\sqrt{n} \Big) \geq 1-e^{-\frac{nt^{2}}{2}}.
    \end{equation*}
\end{lemma}

\begin{lemma}[Lemma 2 of \cite{6283602}] \label{lem::chi-square-concentration}
    Fix $\tau>0$, and let $Z_{i}\sim \calN(0,1)$, $i=1,\ldots, d$. Then,
    \begin{equation*}
        \mathbb{P}\Big( \sum_{i=1}^{d} Z_{i}^{2} < d(1-\tau) \Big) \leq e^{\frac{d}{2} \big(\tau + \log(1-\tau)\big)},
    \end{equation*}
    and 
    \begin{equation*}
        \mathbb{P} \Big( \sum_{i=1}^{d} Z_{i}^{2} > d(1+\tau) \Big) \leq e^{-\frac{d}{2} \big( \tau - \log(1+\tau) \big)}.
    \end{equation*}
\end{lemma}

\begin{lemma}[Lemma 1.2.1 in \cite{vershynin2018high}]\label{lem:meanFromCDF}
Let $X$ denote a non-negative random variable. Then, we have
\[
\E(X) = \int_{0}^\infty \mathbb{P}(X>t) dt.
\]
\end{lemma}

\begin{lemma}[Theorems 3 \& 4 in \cite{ghosh2021exponential}] \label{lem::concentration-non-central-chisquare}
    Suppose $X$ follows a noncentral chi-squared distribution with degrees of freedom $p$ and the noncentrality parameter $\lambda$, i.e., $X\sim \chi_{p}^{2}(\lambda)$. Then,
    \begin{enumerate}[label=(\roman*)]
        \item for $c>0$, $\mathbb{P}(X>p+\lambda+c) \leq \exp \left[ - \frac{pc^{2}}{4(p+2\lambda)(p+2\lambda+c)} \right]$,
        \item for $0<c<p+\lambda$, $\mathbb{P}(X<p+\lambda-c) \leq \exp\left[ 
        - \frac{pc^{2}}{4(p+2\lambda)^{2}} \right]$.
    \end{enumerate}
\end{lemma}

\begin{lemma}[Corollary 5.35 in \cite{vershynin2010introduction}]\label{lem::eval-conc}
    Let the elements of an \(m \times n ~(m  < n)\) matrix \(A\) be drawn independently from \(\calN(0,1)\). Then for any \(t>0\), 
    \begin{equation*}
        \mathbb{P}(\sqrt{n} -\sqrt{m} - t \le \sigma_{\min}(A) \le \sigma_{\max} (A) \le \sqrt{n} + \sqrt{m} + t) \ge 1 - 2e^{-\frac{t^2}{2}},
    \end{equation*}
    where $\sigma_{\max}(A),\sigma_{\min}(A)$ denote the maximum and minimum singular value of $A$ respectively.
\end{lemma}

\subsection{Risk of soft thresholding}

Consider the one-dimensional soft thresholding function: for a given $\chi\geq 0$,
\begin{align}
\eta(u,\chi)&=\argmin_{t\in \mathbb{R}}\frac{1}{2}(t-u)^2+\chi|t| \nonumber \\
&={\rm sign}(u)\cdot (|u|-\chi)_+, \quad \forall u \in \mathbb{R}. \label{soft:thd:f}
\end{align}
For a given pair $(\chi_1,\chi_2)$, define
\begin{align}
\label{1d:risk:f}
r(u;\chi_1,\chi_2) = \E \Big(\frac{1}{1+\chi_2} \eta (u + e, \chi_1 ) - u\Big)^{2},\quad e \sim \mathcal{N}(0,1).
\end{align}
The quantity $r(u;\chi_1,\chi_2)$, as a function of $u$, is the risk of $\frac{1}{1+\chi_2} \eta (u + e, \chi_1 )$ for estimating $u$ based on the noisy observation $u+e$.

\begin{lemma}[Lemma 6 in \cite{guo2022signaltonoise}]
\label{lem::elastic-net-risk}
For any given pair $(\chi_1,\chi_2)$ satisfying $\chi_1>0,\chi_2\geq 0$, it holds that 
\begin{enumerate}[label=(\roman*)]
    \item $r(u;\chi_1,\chi_2)$, as a function of $u$, is symmetric, and increasing over $u \in [0,+\infty)$.
    \item $\max_{(x,y):x^{2}+y^{2} = c^{2}} [r(x;\chi_1,\chi_2) + r(y;\chi_1,\chi_2)] = 2r(c/\sqrt{2};\chi_1,\chi_2)$, \quad $\forall c>0.$
\end{enumerate}
\end{lemma}

%% file: proof_thm3.tex
\section{Proof of Theorem \ref{thm::first-order-snr-minimax}}\label{sec::first-order-snr-minimax}
 The minimax risk in \eqref{eq::snr_minimax} has the scale-invariance property as shown in Section \ref{sec:scaling}. Hence, without loss of generality, we present the proof of Theorem  \ref{thm::first-order-snr-minimax} for $\sigma=1$ in model \eqref{model::gaussian-model}. 

\subsection{Upper bound}

The upper bounds for Regime \rom{1} and \rom{2} can be simply established from the risk of the zero estimator, 
\begin{equation*}
    R(\Theta(k,\mu),1) = \inf_{\hat{\beta}}\sup_{\beta\in\Theta(k,\mu)}\E_{\beta}\|\hat{\beta}-\beta\|_2^{2} \leq \sup_{\beta\in\Theta(k,\mu)} \E_{\beta}\|\mathbf{0}-\beta\|_2^{2} \leq k\mu^{2},
\end{equation*}
where the last inequality follows naturally from the SNR constraint in $\Theta(k,\mu)$. 

The upper bound for Regime \rom{3} follows the upper bound in Theorem \ref{thm::first-order-sparse-minimax}. This is obvious because
\begin{equation*}
    \Theta(k,\mu) \subseteq \Theta(k) \Rightarrow R(\Theta(k,\mu),1) \leq R(\Theta(k),1)=2k\log(p/k)\Big(1+o(1)\Big).
\end{equation*}

\subsection{Lower bound}\label{ssec:lb:firstorder}

Throughout the proof, we drop $\beta$ in $\E_{\beta}(\cdot)$ to simplify the notation. Suppose that we have a prior distribution $\pi$ on the regression coefficients $\beta$ whose support is contained in $\Theta(k,\mu)$. For any estimator $\hat{\beta}$, it is straightforward to see that
\begin{equation}\label{eq:lb1:basic}
\E_\pi \|\hat{\beta} - \beta\|_2^{2} \leq \sup_{\beta \in \Theta(k,\mu)} \E \|\hat{\beta} - \beta\|_2^{2},
\end{equation}
where the expectation on the left is with respect to the randomness in $(X, z, \beta)$, while the expectation on the right is with respect to $(X, z)$ only. Let $B(\pi)$ denote the Bayes risk for prior $\pi$, i.e.,
\[
B(\pi) = \inf_{\hat{\beta}} \E_\pi \|\hat{\beta} - \beta\|_2^{2}. 
\]
By taking an infimum from both sides of \eqref{eq:lb1:basic}, we have
\begin{equation}
\label{baye:risk:lower:bound}
B(\pi) \leq \inf_{\hat{\beta}} \sup_{\beta \in \Theta(k, \mu)} \E \|\hat{\beta} - \beta\|_2^2 = R(\Theta(k,\mu), 1). 
\end{equation}

Therefore, the lower bound of the minimax risk can be provided by the Bayes risk of a prior $\pi$ whose support is contained in the parameter space $\Theta(k,\mu)$. 

We use the independent block prior \cite{johnstone19, donoho1997universal, su2016slope} to obtain a sharp lower bound. The independent block prior, denoted by $\pi_{IB}(\lambda; p,k)$, is constructed in the following way: divide $\beta\in \mathbb{R}^p$ into $k$ disjoint blocks of size $m=p/k$\footnote{For simplicity, $p/k$ is assumed to be an integer. Otherwise, we can slightly adjust the block size to obtain the same lower bound.}: $\beta=(\beta^{(1)},\ldots, \beta^{(k)})$; For each block $1\leq j\leq k$, randomly select an index $I \in [m]$ and set $\beta^{(j)}= \lambda e_I$; The selection between different blocks are independent. The spike choice $\lambda$ can depend on the tuple $(n,p,k)$. But we drop such a dependence throughout the proof for notational simplicity.

 From the construction steps, it already implies that $\pi_{IB}(\lambda;p,k)$ is supported on $\Theta(k)$, i.e. satisfying the sparsity constraint in $\Theta(k,\mu)$. Furthermore, if $ |\lambda|\leq\mu$, then
\begin{equation}
     \supp(\pi_{IB}(\lambda;p,k)) \subseteq \Theta(k,\mu). \label{eq::prior-supported-on-SNR-constraint}
\end{equation}
Thus, the independent block prior with $0<\lambda\leq\mu$ can provide a lower bound for the minimax risk over $\Theta(k,\mu)$. The following proposition obtains a lower bound for the Bayes risk of the independent block prior. We will later use this proposition to obtain the lower bounds required in Theorem \ref{thm::first-order-snr-minimax}.

\begin{proposition} \label{prop::first-order-block-prior}
Assume model \eqref{model::gaussian-model} with $\sigma=1$. Let $\pi := \pi_{IB}(\lambda; p, k)$ be the independent block prior of $\beta$, and $\hat{\beta}_{\pi}$ be the Bayesian estimator (posterior mean) under $\pi$. Suppose $(\log (p/k))/n\rightarrow 0$ and $p/k\rightarrow \infty$. If $\lambda>0$ and $\lambda^2\leq (2-\delta)\log(p/k)$ for a fixed constant $\delta \in (0,1)$, we have
\begin{equation*}
    \E_{\pi} \|\hat{\beta}_{\pi} - \beta\|_2^{2} \geq  k \lambda^{2} \left(1+o(1)\right). 
\end{equation*}
\end{proposition}

\begin{proof}
Given that we have used the independent block prior, the Bayes risk satisfies the following property:
\begin{equation}
    \E_{\pi} \|\hat{\beta}_{\pi} - \beta\|_2^{2} = k \E_{\pi} \|\hat{\beta}_{\pi}^{(1)} - \beta^{(1)}\|_2^{2}, \label{eq::divide-into-k-blocks}
\end{equation}
where $\beta^{(1)}$ (resp. $\hat{\beta}_{\pi}^{(1)}$) denotes the first block of $\beta$ (resp. $\hat{\beta}_{\pi}$). In the rest of the proof, we will use the notation $\beta^{(-1)}$ to denote the remaining blocks of $\beta$. As a result, we have $\beta=(\beta^{(1)}, \beta^{(-1)})$. Accordingly, we write $X=(X^{(1)},X^{(-1)})$ and 
\begin{equation}
\label{single:spike:reg}
    \Tilde{y}:= y - X^{(-1)} \beta^{(-1)} = X^{(1)} \beta^{(1)} + z.
\end{equation}
As is clear from \eqref{eq::divide-into-k-blocks}, to obtain a lower bound for the Bayes risk, we need to find a lower bound for $\E_{\pi} \|\hat{\beta}_{\pi}^{(1)} - \beta^{(1)}\|_2^{2}$. We have
\begin{align}\label{eq:lb:complete:classicalminimax}
     &\quad ~~\E_{\pi} \|\E_{\pi} \big( \beta^{(1)} \vert y, X\big) - \beta^{(1)}\|_2^{2} \nonumber \\
    &\overset{(a)}{\geq} \E_{\pi} \|\E_{\pi} \big( \beta^{(1)}\vert y, X, \beta^{(-1)} \big) - \beta^{(1)}\|_2^{2} \nonumber \\
    &\overset{(b)}{=} \E_{\pi} \|\E_{\pi} \big( \beta^{(1)}\vert \Tilde{y}, X^{(1)}, X^{(-1)}, \beta^{(-1)} \big) - \beta^{(1)}\|_2^{2} \nonumber \\
    &\overset{(c)}{=} \E_{\pi} \| \E_{\pi} \big(\beta^{(1)} \vert \Tilde{y}, X^{(1)} \big) - \beta^{(1)} \|_2^{2}.
\end{align}
Inequality $(a)$ holds because further conditioning reduces mean squared error. Equality $(b)$ is due to the fact that $\{y, X, \beta^{(-1)}\}$ can be recovered from $\{\tilde{y}, X^{(1)}, X^{(-1)}, \beta^{(-1)}\}$ and vice versa. Equality $(c)$ follows by the result that $\{X^{(-1)}, \beta^{(-1)}\}$ are independent of $\{\Tilde{y}, X^{(1)}, \beta^{(1)}\}$. 

Note that the lower bound \eqref{eq:lb:complete:classicalminimax} is the Bayes risk under the regression model \eqref{single:spike:reg} with a single spike prior $\pi_{S}(\lambda;m)$ for $\beta^{(1)}\in \mathbb{R}^m$, defined as: select an index $I\in [m]$ uniformly at random and set $\beta^{(1)}=\lambda e_I$. The following lemma provides a lower bound for such Bayes risk.

\begin{lemma}\label{lem::first_order-lower_bound-spike_prior}
    Consider model \eqref{model::gaussian-model} with $\sigma=1$ and $\beta\in \R^{m}$. Suppose $m \rightarrow \infty$ and $(\log m)/ n \rightarrow 0$. Let $\pi:=\pi_{S}(\lambda;m)$ be the single spike prior of $\beta$. Denote $\hat{\beta}_{\pi}$ as the Bayesian estimator under $\pi$. If $\lambda>0$ and $\lambda^2\leq (2-\delta)\log m$ for a fixed constant $\delta \in (0,1)$, we have
\begin{equation*}
    \E_{\pi}\|\hat{\beta}_{\pi} - \beta\|_2^{2} \geq \lambda^{2} (1+o(1)).
\end{equation*}
\end{lemma}
We will prove this lemma in the next section. But before proving the lemma, let us show how this lemma helps us finish the proof of Proposition \ref{prop::first-order-block-prior}, and how Proposition \ref{prop::first-order-block-prior} enables us to finish the  proof of Theorem \ref{thm::first-order-snr-minimax}. 

Recalling $m=p/k$, from  Lemma \ref{lem::first_order-lower_bound-spike_prior} we conclude that
\begin{equation*}
    \E_{\pi} \| \E_{\pi} \big(\beta^{(1)} \vert \Tilde{y}, X^{(1)} \big) - \beta^{(1)} \|_2^{2} \geq \lambda^{2} (1+o(1)).
\end{equation*}
This combined with \eqref{eq::divide-into-k-blocks} and \eqref{eq:lb:complete:classicalminimax} completes the proof of Proposition \ref{prop::first-order-block-prior}.
\end{proof}

By choosing proper values for $\lambda$, we can use Proposition \ref{prop::first-order-block-prior} to obtain the lower bounds required in Theorem \ref{thm::first-order-snr-minimax}. These choices are clarified below.

\begin{itemize}
    \item For Regime \rom{1}, let $\lambda=\mu\rightarrow 0$, then $\lambda^2=o\big((2-\delta)\log (p/k)\big)$ and hence 
    \[
    R(\Theta(k,\mu), 1)\geq B(\pi_{IB}(\lambda;p,k)) \geq k\mu^{2}\big(1+o(1)\big).
    \]
    \item For Regime \rom{2}, let $\lambda=\mu$, then $\lambda^2=o\big((2-\delta)\log (p/k)\big)$ and hence
    \[
    R(\Theta(k,\mu), 1)\geq B(\pi_{IB}(\lambda;p,k)) \geq k\mu^{2}\big(1+o(1)\big).
    \]
    \item For Regime \rom{3}, let $\lambda=\sqrt{(2-\delta)\log(p/k)}$, then  
    \[
    R(\Theta(k,\mu), 1)\geq B(\pi_{IB}(\lambda;p,k)) \geq (2-\delta)k\log(p/k)\big(1+o(1)\big).
    \]
    Hence, 
    \[
    \liminf_{n\rightarrow \infty}\frac{R(\Theta(k,\mu), 1)}{k\log(p/k)}\geq 2-\delta, \quad \forall \delta \in (0,1).
    \]
  Letting $\delta\rightarrow 0+$ in the above yields $\liminf_{n\rightarrow \infty}\frac{R(\Theta(k,\mu), 1)}{k\log(p/k)}\geq 2$. This lower bound serves the purpose.
\end{itemize}



\subsection{Proof of Lemma \ref{lem::first_order-lower_bound-spike_prior}}
Under the  single spike prior, let $I$ denote the index of the spike coordinate. Throughout this section, we will use the notation $p_{i}:=\mathbb{P}(I=i~|~y, X)$, $i\in[m]$. The notation $P_{\lambda e_i}$ denotes the joint probability of $\{y, X\}$ under the model $y=X\beta+z$ with $\beta=\lambda e_i$, and $\E_{\lambda e_{i}}$ denotes the expectation taken under $P_{\lambda e_i}$. We have 
\begin{eqnarray}\label{eq:lb:lemma6}
    \E_{\pi}\|\hat{\beta}_{\pi} - \beta\|_{2}^{2} &=& \lambda^{2} \E_{\lambda e_{1}} (p_{1} - 1)^{2} + \lambda^{2}(m-1)\E_{\lambda e_{2}}p_{1}^{2} \nonumber \\
    &\geq& \lambda^{2} \E_{\lambda e_{1}} (p_{1} - 1)^{2}.
\end{eqnarray}
It is also straightforward to calculate the posterior probability $p_1$,
\begin{equation}
\label{first:form}
    p_{1}=\mathbb{P}(I=1~|~ y, X) = \frac{\exp(\lambda X_{1}^{T}y - \lambda^{2}\|X_{1}\|_2^{2}/2)}{\exp(\lambda X_{1}^{T}y - \lambda^{2}\|X_{1}\|_2^{2}/2) + \sum_{i=2}^{m} \exp (\lambda X_{i}^{T}y - \lambda^{2}\|X_{i}\|_2^{2}/2)},
\end{equation}
where $X_i$ represents the $i$th column of $X$. Since $0\leq p_{1}\leq 1$, if we can show that
\begin{equation}\label{eq:p1goestozero}
    p_{1} \rightarrow 0 \quad \text{in }P_{\lambda e_{1}}\text{- probability},
\end{equation}
then by combining the continuous mapping and the dominated convergence theorems with \eqref{eq:lb:lemma6}, we will conclude
\begin{equation*}
    \E_{\pi} \| \hat{\beta}_{\pi} - \beta \|_2^{2} \geq \lambda^{2} (1+o(1)).
\end{equation*}
To prove \eqref{eq:p1goestozero}, let us first write out $p_1$ in \eqref{first:form} when $\beta=\lambda e_1$, as a function of $\{X, z\}$. We use $X_{i,1}$ and $X_{i, -1}$ to denote the first coordinate and the remaining coordinates of $X_i$, respectively. We plug $y=\lambda X_1+z$ into \eqref{first:form} to obtain
\begin{eqnarray*}
    p_{1} &=& \Bigg[ 1 + \frac{\sum_{i=2}^{m} \exp\Big( \lambda X_{i}^{T} (\lambda X_{1}+ z) - \lambda^{2}\|X_{i}\|_2^{2}/2 \Big)}{\exp\Big( \lambda^{2} \|X_{1}\|_2^{2}/2 + \lambda X_{1}^{T}z \Big)} \Bigg]^{-1} \\
    &\overset{d}{=}& \Bigg[ 1 + \frac{\sum_{i=2}^{m}\exp\Big( \|\lambda X_{1}+ z\|_2\lambda X_{i,1} - \lambda^{2} X_{i,1}^{2}/2 - \lambda^{2}\|X_{i,-1}\|_2^{2}/2\Big)}{\exp\Big( \lambda^{2}\|X_{1}\|_2^{2}/2 + \lambda X_{1}^{T}z\Big)} \Bigg]^{-1},
\end{eqnarray*}
where in the last step, we have used the fact that $\{AX_2,\ldots, AX_m |X_1, z\} \overset{d}{=}\{X_2,\ldots, X_m |X_1, z\}$, for any orthogonal matrix $A$ whose first row equals $(\lambda X_1+z)/\|\lambda X_1+z\|_2$. We now equivalently write $p_1$ as 
\begin{equation*}
    p_{1} = \left(1+\mathcal{A}_{n,m}\mathcal{B}_{n,m}\right)^{-1},
\end{equation*}
where
\begin{align}
    & \mathcal{A}_{n,m} = \frac{\sum_{i=2}^{m} \exp\left( \|\lambda X_{1} + z\|_2 \lambda X_{i,1} - \frac{\lambda^{2}}{2} X_{i,1}^{2} - \frac{\lambda^{2}}{2} \|X_{i,-1}\|_2^{2} \right)}{(m-1) (1+\lambda^2/n)^{-\frac{n}{2}} \exp\left( \frac{\lambda^2}{2(n+\lambda^{2})} \|z+\lambda X_{1}\|_2^{2} \right) }, \label{term::A-first-order-spike} \\
    & \mathcal{B}_{n,m} = \frac{(m-1) (1+\lambda^{2}/n)^{-\frac{n}{2}} \exp\left( \frac{\lambda^2}{2(n+\lambda^{2})} \|z+\lambda X_{1}\|_2^{2} \right)}{\exp\left( \frac{\lambda^{2}}{2}\|X_{1}\|_2^{2} + \lambda X_{1}^{T}z \right)}. \label{term::B-first-order-spike}
\end{align}
In order to show \eqref{eq:p1goestozero}, our goal is to show that $\mathcal{A}_{n,m}\overset{p}{\rightarrow}1$ and $\mathcal{B}_{n,m} \overset{p}{\rightarrow}+\infty$. This will be done in the next two lemmas.

\begin{lemma}
\label{Bnm:lem:snr}
Assume $\sqrt{n}X_1,\ldots, \sqrt{n}X_m, z\overset{i.i.d.}{\sim}\mathcal{N}(0, I_n)$. Suppose $m\rightarrow \infty$ and $(\log m)/n\rightarrow 0$. Consider the random variable $\mathcal{B}_{n,m}$ defined in \eqref{term::B-first-order-spike}. If $\lambda>0$ and $\lambda^2\leq (2-\delta)\log m$ for a fixed constant $\delta \in (0,1)$, then
\begin{equation*}
    \mathcal{B}_{n,m} \overset{p}{\rightarrow} \infty.
\end{equation*}
\end{lemma}
\begin{proof}
We can rewrite $\mathcal{B}_{n,m}$ in the following form: 
\begin{equation*}
    \mathcal{B}_{n,m} = (m-1) \left(1+\lambda^{2}/n\right)^{-\frac{n}{2}} \exp \left( \frac{\lambda^2}{2(n+\lambda^{2})}\Big(\|z\|_2^{2}-\|\sqrt{n}X_1\|_2^2\Big)  - \frac{\lambda}{1+\lambda^{2}/n} X_{1}^{T}z  \right).
\end{equation*}
Using the central limit theorem and Taylor's theorem, we have the following results: 
\begin{align*}
& \|z\|_2^{2}-\|\sqrt{n}X_1\|_2^{2} = O_{p}(\sqrt{n}), ~~X_{1}^{T} z=O_{p} (1), \\
& \big(1+\lambda^{2}/n\big)^{-\frac{n}{2}}=\exp\Big(-\frac{n}{2}\log(1+\lambda^2/n)\Big)=\exp\Big(-\frac{n}{2}(\lambda^2/n-(0.5+o(1))\lambda^4/n^2)\Big),
\end{align*}
where we have used the fact $\lambda=o(\sqrt{n})$ that is implied by the conditions $(\log m)/n\rightarrow 0, \lambda^2\leq (2-\delta)\log m$. Plugging the above results into $\mathcal{B}_{n,m}$ we obtain
\begin{eqnarray*}
    \mathcal{B}_{n,m} &=& \exp \bigg[ \log (m-1) - \frac{\lambda^{2}}{2} +\frac{n^{-1}\lambda^{4}}{4+o(1)}+ \frac{\lambda^2\sqrt{n}O_{p}(1)}{2(n+\lambda^{2})} - \frac{\lambda O_p(1)}{1+\lambda^{2}/n}  \bigg] \\
    &=& \exp\bigg[(1+o(1))\cdot \log m - \frac{\lambda^{2}}{2} +\frac{n^{-1}\lambda^{4}}{4+o(1)}+ \lambda O_{p}(1) \bigg].
\end{eqnarray*}
Then to show $\mathcal{B}_{n,m} \overset{p}{\rightarrow} +\infty$, using the above expression and the continuous mapping theorem, it's sufficient to see that under the conditions $m\rightarrow \infty, \lambda^2\leq (2-\delta)\log m$, 
\begin{align*}
   & ~(1+o(1))\cdot \log m - \frac{\lambda^{2}}{2}+\lambda O_{p}(1)\\
   \geq & ~ \frac{\delta+o(1)}{2}\cdot \log m + O_{p}(\sqrt{\log m}) \overset{p}{\rightarrow} +\infty.
\end{align*}
\end{proof}


\begin{lemma}\label{lem::A-term-convergest-to-one}
     Assume the same conditions of Lemma \ref{Bnm:lem:snr}. Consider the random variable $\mathcal{A}_{n,m}$ defined in \eqref{term::A-first-order-spike}. Then, 
    \begin{equation*}
        \mathcal{A}_{n,m} \overset{p}{\rightarrow} 1.
    \end{equation*}
\end{lemma}

\begin{proof}
We first introduce a few notations:
\begin{align*}
&b_{n,m}:=(m-1) (1+\lambda^2/n)^{-\frac{n}{2}} \exp\Big( \frac{\lambda^2}{2(n+\lambda^{2})} \|\lambda X_{1}+z\|_2^{2} \Big) \\ 
& Y_{m,i} := \exp\Big[ \|\lambda X_{1} + z\|_2\lambda X_{i,1} - \frac{\lambda^{2}}{2}X_{i,1}^{2} - \frac{\lambda^{2}}{2}\|X_{i,-1}\|_2^{2} \Big], \quad \bar{Y}_{m,i}:=Y_{m,i}\mathbbm{1}_{(Y_{m,i}\leq b_{n,m})}\\
&S_{n,m}:=\sum_{i=2}^mY_{m,i}, ~~\bar{S}_{n,m}:=\sum_{i=2}^m\bar{Y}_{m,i}, ~~a_{n,m}=\E\big(\bar{S}_{n,m}~|~\|\lambda X_{1} + z\|_2\big).
\end{align*}
Thus, $\mathcal{A}_{n,m}=\frac{S_{n,m}}{b_{n,m}}$. Note that if we condition on $\|\lambda X_1 +z\|_2$, then $\{Y_{m,i}: i=2,\ldots, m\}$ are independent and identically distributed random variables. It is straightforward to verify that $\E\mathcal{A}_{n,m}=1$. However, the variance ${\rm Var}(\mathcal{A}_{n,m})$ is so large that $\mathcal{A}_{n,m} \overset{p}{\rightarrow} 1$ cannot be concluded by Chebychev's inequality. We resolve the issue via a truncation method similar to Theorem 2.2.11 of \cite{durrett2019probability}. We will prove Lemma \ref{lem::A-term-convergest-to-one} by establishing 
\begin{align}
\label{lem14:map}
\frac{S_{n,m}-a_{n,m}}{b_{n,m}}\overset{p}{\rightarrow}0, \quad  \quad \frac{a_{n,m}}{b_{n,m}} \overset{p}{\rightarrow} 1.
\end{align}

To prove the first result in \eqref{lem14:map}, we have $\forall \epsilon >0$,
\begin{align}
&~\mathbb{P}\Big(\Big|\frac{S_{n,m}-a_{n,m}}{b_{n,m}}\Big|>\epsilon \Big) \nonumber \\
\leq &~\mathbb{P}(S_{n,m}\neq \bar{S}_{n,m})+\mathbb{P}\Big(\Big|\frac{\bar{S}_{n,m}-a_{n,m}}{b_{n,m}}\Big|>\epsilon \Big) \nonumber \\
\leq &~ \sum_{i=2}^m \mathbb{P}(Y_{m,i}>b_{n,m})+ \mathbb{P}\Big(\Big|\frac{\bar{S}_{n,m}-a_{n,m}}{b_{n,m}}\Big|>\epsilon \Big), \label{first:result:decomp}
\end{align}
and we aim to show that both terms converge to zero. We start by applying Lemma \ref{lem::wlln-conditions-lemma} (i) to obtain
\begin{align}
\label{term1:step1}
\sum_{i=2}^m \mathbb{P}\Big(Y_{m,i}>b_{n,m}~|~\|\lambda X_1+z\|_2\Big) = (m-1)\cdot \mathbb{P}\Big(\chi^2_n(\gamma_1) < c_1 ~\Big |~ \|\lambda X_{1}+z\|_2\Big),
\end{align}
where $c_1=n\lambda^{-2}( \|\lambda X_{1}+z\|_2^{2} -2\log b_{n,m})$, and $\chi^2_n(\gamma_1)$ is a noncentral chi-squared distribution with degrees of freedom $n$ and non-centrality parameter $\gamma_1=n\lambda^{-2}\|\lambda X_1+z\|_2^2$. It turns out that the usual concentration inequality (e.g. Lemma \ref{lem::concentration-non-central-chisquare}) does not yield a sharp enough left-tail bound for the above. Instead, we resort to the Cram\'{e}r–Chernoff bounding method with a careful analysis of the moment-generating function. Specifically, $\forall \alpha>0$,
\begin{align}
\label{term1:step2}
&~(m-1)\cdot \mathbb{P}\Big(\chi^2_n(\gamma_1) < c_1 ~\Big |~ \|\lambda X_{1}+z\|_2\Big) \nonumber \\
= &~ (m-1)\cdot \mathbb{P}\Big(e^{-\alpha \chi^2_n(\gamma_1)} > e^{-\alpha c_1} ~\Big |~ \|\lambda X_{1}+z\|_2\Big) \nonumber \\
\leq &~ e^{\alpha c_1+\log(m-1)} \cdot \E e^{-\alpha \chi^2_n(\gamma_1) }=e^{\alpha c_1+\log(m-1)} \cdot e^{\frac{-\alpha \gamma_1}{1+2\alpha}}(1+2\alpha)^{-n/2} \nonumber \\
= & ~\exp\Bigg \{ \frac{\alpha n^2}{\lambda^2}\log \Big (1 + \frac{\lambda^2}{n} \Big ) - \frac{n}{2}\log (1+2\alpha) + \nonumber \\
    &\quad \quad \quad \|\lambda X_1 + z\|_2^2\left (\frac{2\alpha^2 n}{\lambda^2(1+2\alpha)} -\frac{\alpha n}{n + \lambda^2} \right )  + \left (1 - \frac{2\alpha n}{\lambda^2} \right )\log (m-1) \Bigg \}:=e^{f_{n,m}(\alpha)}. 
\end{align}
Here, the first inequality applies Markov’s inequality; the second equality uses the moment-generating function of noncentral chi-squared distribution; and the third equality is obtained by plugging in the values of $\gamma_1,c_1$ and $b_{n,m}$. Combining \eqref{term1:step1} and \eqref{term1:step2} gives
\begin{align}
\label{term1:step3}
\sum_{i=2}^m \mathbb{P}(Y_{m,i}>b_{n,m}) &\leq \E e^{f_{n,m}(\alpha)} \nonumber \\
&= \exp\Bigg \{ \frac{\alpha n^2}{\lambda^2}\log \Big (1 + \frac{\lambda^2}{n} \Big ) - \frac{n}{2}\log (1+2\alpha) + \left (1 - \frac{2\alpha n}{\lambda^2} \right )\log (m-1)+ \nonumber \\
    &\quad \quad \quad -\frac{n}{2}\log\Bigg(1-\frac{2\alpha(2n\alpha-\lambda^2)}{\lambda^2(1+2\alpha)}\Bigg)  \Bigg \}:=e^{g_{n,m}(\alpha)},
\end{align}
where the first equality holds because $\|\lambda X_1+z\|_2^2 \sim \frac{\lambda^2+n}{n}\cdot \chi^2_n$, hence we can use the moment-generating function of chi-squared distribution to compute $\E e^{f_{n,m}(\alpha)}$. Based on \eqref{term1:step3}, to prove $\sum_{i=2}^m \mathbb{P}(Y_{m,i}>b_{n,m})\rightarrow 0$, it is sufficient to show $g_{n,m}(\alpha)\rightarrow -\infty$. We set
\begin{align*}
\label{alpha:choice}
\alpha=\frac{4-\delta}{8-4\delta}\cdot \frac{\lambda^2}{n}.
\end{align*}
Under the conditions $m\rightarrow \infty, \frac{\log m}{n}\rightarrow 0, \lambda^2 \leq (2-\delta)\log m$, using the Taylor's expansion $\log(1+t)=t-\frac{1+o(1)}{2}t^2$ for $t\rightarrow 0$, we can easily obtain
\begin{align*}
\frac{\alpha n^2}{\lambda^2}\log \Big (1 + \frac{\lambda^2}{n} \Big )&= \frac{4-\delta+o(1)}{8-4\delta}\cdot \lambda^2\\
\frac{n}{2}\log (1+2\alpha)&=\frac{4-\delta+o(1)}{8-4\delta}\cdot \lambda^2 \\
\frac{n}{2}\log\Bigg(1-\frac{2\alpha(2n\alpha-\lambda^2)}{\lambda^2(1+2\alpha)}\Bigg) & \geq \frac{-(4+o(1)-\delta)\delta}{8(2-\delta)}\cdot \log m \\
\Big(1 - \frac{2\alpha n}{\lambda^2} \Big )\log (m-1) &=\frac{(-1+o(1))\delta}{4-2\delta} \cdot \log m
\end{align*}
Combining the above results with \eqref{term1:step3}, we can conclude that
\[
g_{n,m}(\alpha)\leq -\frac{(1+o(1))\delta^2}{16-8\delta} \cdot \log m+ o(\lambda^2)\rightarrow -\infty.
\]

Next, we show the second term $\mathbb{P}\Big(\Big|\frac{\bar{S}_{n,m}-a_{n,m}}{b_{n,m}}\Big|>\epsilon \Big)$ in \eqref{first:result:decomp} converges to zero. We first have
\begin{align}
\label{term2:step1}
&~\mathbb{P}\Big(\Big|\frac{\bar{S}_{n,m}-a_{n,m}}{b_{n,m}}\Big|>\epsilon ~\big |~\|\lambda X_1+z\|_2 \Big) \nonumber \\ \leq &~\epsilon^{-2}b_{n,m}^{-2}\cdot {\rm Var}\big(\bar{S}_{n,m}~|~\|\lambda X_1+z\|_2\big) \nonumber \\
\leq &~\epsilon^{-2}b_{n,m}^{-2}\cdot \sum_{i=2}^{m} \E \Big(Y_{m,i}^{2} \mathbbm{1}_{(Y_{m,i}\leq b_{n,m})} ~ \big | ~\|~\lambda X_{1} + z\|_2\Big),
\end{align}
where the first inequality applies Chebychev's inequality, and the second inequality holds since $\{\bar{Y}_{m,i}, i=2,\ldots, m\}$ are independent conditioning on $\|\lambda X_1+z\|_2$ and ${\rm Var}\big(\bar{Y}_{m,i}~|~\|\lambda X_1+z\|_2\big)\leq \E \big(Y_{m,i}^{2} \mathbbm{1}_{(Y_{m,i}\leq b_{n,m})} ~ \big | ~\|~\lambda X_{1} + z\|_2\big)$. Using Lemma \ref{lem::wlln-conditions-lemma} (ii), we can continue from \eqref{term2:step1} to obtain 
\begin{align}
\label{term2:key:bound}
    &~ b_{n,m}^{-2}\sum_{i=2}^{m} \E \Big(Y_{m,i}^{2} \mathbbm{1}_{(Y_{m,i}\leq b_{n,m})} ~ \big | ~\|~\lambda X_{1} + z\|_2\Big) \nonumber \\
    =&~ b_{n,m}^{-2} (m-1)\Big(1+\frac{2\lambda^{2}}{n}\Big)^{-\frac{n}{2}} \exp \left( \frac{2\| \lambda X_{1} + z\|_2^{2}}{2+n/\lambda^{2}} \right) \cdot  \mathbb{P}\bigg[ \chi_n^2(\gamma_2) \geq c_2~ \Big | ~\|\lambda X_{1} + z\|_2\bigg] \nonumber \\
    \leq & ~b_{n,m}^{-2} (m-1)\Big(1+\frac{2\lambda^{2}}{n}\Big)^{-\frac{n}{2}} \exp \left( \frac{2\| \lambda X_{1} + z\|_2^{2}}{2+n/\lambda^{2}} \right)\cdot e^{-\alpha c_2}\E e^{\alpha \chi^2_n(\gamma_2)} \nonumber \\
    = &~\exp\Bigg(\frac{2\alpha n+(4\alpha-1)\lambda^2}{\lambda^2}\cdot \log(m-1)-\frac{n}{2}\log\Big(1+\frac{2\lambda^2}{n}\Big)-\frac{n}{2}\log(1-2\alpha)+ \nonumber \\
    &~\frac{n(\lambda^2-\alpha(n+2\lambda^2))}{\lambda^2}\log\Big(1+\frac{\lambda^2}{n}\Big)+\frac{\alpha(n^2-(1-2\alpha)(2\lambda^2+n)^2)}{(1-2\alpha)\lambda^2(2\lambda^2+n)}\|\lambda X_1+z\|^2_2+  \nonumber \\
    &~\quad \quad \Big(\frac{2\lambda^2}{2\lambda^2+n}-\frac{\lambda^2-\alpha(n+2\lambda^2)}{n+\lambda^2}\Big)\|\lambda X_1+z\|^2_2\Bigg), \quad \quad \forall \alpha >0,
\end{align}
where $\gamma_2=\frac{n^2}{2\lambda^4+n\lambda^2}\|\lambda X_1+z\|_2^2, c_2=\frac{n+2\lambda^2}{\lambda^{2}}\big(\|\lambda X_{1} + z\|_2^{2} - 2\log b_{n,m}\big)$. In the above, the last two steps follow similarly as in \eqref{term1:step2} by using Markov’s inequality and the moment-generating function of noncentral chi-squared distribution. We set $\alpha=\frac{4-3\delta}{8-4\delta}\cdot \frac{\lambda^2}{n}$. Under the conditions $m\rightarrow \infty, \frac{\log m}{n}\rightarrow 0, \lambda^2 \leq (2-\delta)\log m$, as in bounding $g_{n,m}(\alpha)$, we can have
\begin{align*}
-\frac{n}{2}\log\Big(1+\frac{2\lambda^2}{n}\Big)&=-\lambda^2+o(\lambda^2) \\
-\frac{n}{2}\log(1-2\alpha)&=n\alpha+o(n\alpha)\\
\frac{n(\lambda^2-\alpha(n+2\lambda^2))}{\lambda^2}\log\Big(1+\frac{\lambda^2}{n}\Big)&=\lambda^2-n\alpha+o(\lambda^2+n\alpha) \\
\frac{2\alpha n+(4\alpha-1)\lambda^2}{\lambda^2}\cdot \log(m-1)&=\frac{-\delta+o(1)}{4-2\delta}\cdot \log m \\
\frac{\alpha(n^2-(1-2\alpha)(2\lambda^2+n)^2)}{(1-2\alpha)\lambda^2(2\lambda^2+n)}&=\frac{(4-3\delta+o(1))(5\delta-12)}{2(4-2\delta)^2}\cdot \frac{\lambda^2}{n} \\
\frac{2\lambda^2}{2\lambda^2+n}-\frac{\lambda^2-\alpha(n+2\lambda^2)}{n+\lambda^2}&=\frac{12-7\delta+o(1)}{8-4\delta}\cdot \frac{\lambda^2}{n} \\
\|\lambda X_1+z\|^2_2&=n+o_p(n)
\end{align*}
Putting together the above results with \eqref{term2:step1} and \eqref{term2:key:bound} yields
\begin{align*}
\mathbb{P}\Big(\Big|\frac{\bar{S}_{n,m}-a_{n,m}}{b_{n,m}}\Big|>\epsilon ~\big |~\|\lambda X_1+z\|_2 \Big)\leq \epsilon^{-2} e^{-\frac{\delta^2+o_p(1)}{16-8\delta}\cdot \log m} \overset{p}{\rightarrow }0.
\end{align*}
As a result, the dominated (Vitali's) convergence theorem implies 
\[
\E\Big(\Big|\frac{\bar{S}_{n,m}-a_{n,m}}{b_{n,m}}\Big|>\epsilon\Big)=\E\bigg[\mathbb{P}\Big(\Big|\frac{\bar{S}_{n,m}-a_{n,m}}{b_{n,m}}\Big|>\epsilon ~\big |~\|\lambda X_1+z\|_2 \Big)\bigg]\rightarrow 0.
\]

So far we have obtained the first result in \eqref{lem14:map}. It remains to prove the second one in \eqref{lem14:map}. From Lemma \ref{lem::wlln-conditions-lemma} (iii), we have
\begin{equation*}
    a_{n,m} = (m-1) \Big(1 +\lambda^{2}/n \Big)^{-n/2} \exp\bigg( \frac{\|\lambda X_{1} + z\|_2^{2}}{2(1+n/\lambda^{2})} \bigg) \bigg( 1- \mathbb{P}\Big( \chi_{n}^{2} (\gamma_{3}) \leq c_3 ~\Big |~ \|\lambda X_{1} + z \|_2 \Big) \bigg), \label{eq::a_n_p_definition}
\end{equation*}
where $\gamma_{3} = \frac{n^{2}}{\lambda^{4}+n\lambda^2}\|\lambda X_{1} + z\|_2^{2}$ and $c_{3}= \frac{n+\lambda^2}{\lambda^2}(\|\lambda X_{1} + z \|_2^2-2\log b_{n,m})$. Hence, to show $\frac{a_{n,m}}{b_{n,m}} \overset{p}{\rightarrow} 1$, it is equivalent to show
\begin{equation*}
    \mathbb{P}\Big( \chi_{n}^{2} (\gamma_{3}) \leq c_3 ~\Big |~ \|\lambda X_{1} + z \|_2 \Big) =o_p(1).
\end{equation*}
Similar to the calculations of \eqref{term1:step2}, we use the Cram\'{e}r–Chernoff bounding method to obtain a sharp tail bound for the above,
\begin{align}
\label{final:one}
&~\mathbb{P}\Big( \chi_{n}^{2} (\gamma_{3}) \leq c_3 ~\Big |~ \|\lambda X_{1} + z \|_2 \Big) \nonumber \\
\leq &~e^{\alpha c_3} \cdot \E e^{-\alpha \chi^2_n(\gamma_3) }=e^{\alpha c_3} \cdot e^{\frac{-\alpha \gamma_3}{1+2\alpha}}(1+2\alpha)^{-n/2} \nonumber\\
=&~\exp\Bigg(-\frac{n}{2}\log(1+2\alpha)+\frac{\alpha n(n+\lambda^2)}{\lambda^2}\log\Big(1+\frac{\lambda^2}{n}\Big)-\frac{2\alpha(n+\lambda^2)}{\lambda^2}\log(m-1)+ \nonumber \\
&\quad \quad \quad \frac{n\alpha(\lambda^2+2\alpha n+2\alpha \lambda^2)}{(1+2\alpha)\lambda^2(n+\lambda^2)}\|\lambda X_{1} + z \|_2^2\Bigg).
\end{align}
We choose $\alpha=\frac{\delta}{8-4\delta}\cdot \frac{\lambda^2}{n}$, and verify that
\begin{align*}
-\frac{n}{2}\log(1+2\alpha)&=-n\alpha +o(\lambda^2) \\
\frac{\alpha n(n+\lambda^2)}{\lambda^2}\log\Big(1+\frac{\lambda^2}{n}\Big)&=n\alpha +o(\lambda^2) \\
-\frac{2\alpha(n+\lambda^2)}{\lambda^2}\log(m-1)&=\frac{\delta+o(1)}{2\delta-4}\cdot \log m \\
\frac{n\alpha(\lambda^2+2\alpha n+2\alpha \lambda^2)}{(1+2\alpha)\lambda^2(n+\lambda^2)}\|\lambda X_{1} + z \|_2^2 &\leq \frac{\delta(4-\delta)+o_p(1)}{8(2-\delta)}\cdot \log m
\end{align*}
Plugging the above results into \eqref{final:one} gives
\[
\mathbb{P}\Big( \chi_{n}^{2} (\gamma_{3}) \leq c_3 ~\Big |~ \|\lambda X_{1} + z \|_2 \Big) \leq e^{-\frac{-\delta^2+o(1)}{8(2-\delta)}\cdot \log m} \overset{p}{\rightarrow} 0.
\]
\end{proof}

\begin{lemma}\label{lem::wlln-conditions-lemma}
Consider the random vectors $X_{1}, \ldots, X_{m} \overset{i.i.d.}{\sim} \calN(0, \frac{1}{n}I_{n})$, independent of $z\sim \calN(0, I_{n})$. Moreover, let $Z\sim \mathcal{N}(0,1)$ and $V \sim \chi_{n-1}^{2}$ be two independent random variables which are also independent of $(X_1, z)$. Denote
\begin{equation*}
   Y_{m,i} := \exp\Big( \|\lambda X_{1} + z\|_2\lambda X_{i,1} - \frac{\lambda^{2}}{2}X_{i,1}^{2} - \frac{\lambda^{2}}{2}\|X_{i,-1}\|_2^{2} \Big), \quad i=2,\ldots, m,
\end{equation*}
where $X_{i,1}\in \mathbb{R}$ and $X_{i, -1}\in \mathbb{R}^{n-1}$ denote the first coordinate and the remaining coordinates of $X_i$, respectively. For any constants $\lambda >0, b>0$, it holds that 
\begin{enumerate}[label=(\roman*)]
    \item 
    $\begin{aligned}[t]
    & \mathbb{P}\Big(Y_{m,i}>b~|~\|\lambda X_1+z\|_2\Big) \nonumber \\
=&~ \mathbb{P}\bigg[ \Big( Z - \frac{\sqrt{n}}{\lambda} \|\lambda X_{1}+ z\|_2 \Big)^{2} + V < n\lambda^{-2} \Big( \|\lambda X_{1}+z\|_2^{2} -2\log b \Big) ~\Big |~ \|\lambda X_{1}+z\|_2 \bigg] 
    \end{aligned}$
    \item 
    $\begin{aligned}[t]
        & \E\Big(Y_{m,i}^{2}\mathbbm{1}_{(Y_{m,i}\leq b)} ~|~ \|\lambda X_{1} + z\|_2\Big) \\
        =& \Big(1+\frac{2\lambda^{2}}{n}\Big)^{-\frac{n}{2}} \exp \left( \frac{2\| \lambda X_{1} + z\|_2^{2}}{2+n/\lambda^{2}} \right) \cdot \nonumber\\
        & \mathbb{P}\bigg[ \Big(Z - \frac{n/\lambda^{2}}{\sqrt{2+n/\lambda^{2}}} \|\lambda X_{1}+z\|_2 \Big)^{2} + V \geq \frac{n+2\lambda^2}{\lambda^{2}}\big(\|\lambda X_{1} + z\|_2^{2} - 2\log b\big)~ \Big | ~\|\lambda X_{1} + z\|_2\bigg] \nonumber
    \end{aligned}$
    \item 
    $\begin{aligned}[t]
        & \E\Big(Y_{m,i}\mathbbm{1}_{(Y_{m,i}\leq b)} ~|~ \|\lambda X_{1} + z\|_2\Big) \\
        =& \Big(1+\frac{\lambda^{2}}{n}\Big)^{-\frac{n}{2}} \exp \left( \frac{\| \lambda X_{1} + z\|_2^{2}}{2(1+n/\lambda^{2})} \right) \cdot \nonumber\\
        & \mathbb{P}\bigg[ \Big(Z - \frac{n/\lambda^{2}}{\sqrt{1+n/\lambda^{2}}} \|\lambda X_{1}+z\|_2 \Big)^{2} + V \geq \frac{n+\lambda^2}{\lambda^{2}}\big(\|\lambda X_{1} + z\|_2^{2} - 2\log b\big)~ \Big | ~\|\lambda X_{1} + z\|_2\bigg] \nonumber
    \end{aligned}$
\end{enumerate}
\end{lemma}
\begin{proof}
We first prove two useful preliminary results. For any constants $t_1<0, t_2,t_3\in \mathbb{R}$, we have
\begin{align}
\label{single:norm}
\E \big(e^{t_1(Z-t_2)^2}\mathbbm{1}_{\{(Z-t_2)^2\geq t_3\}}\big)=&~\int_{-\infty}^{+\infty}\frac{1}{\sqrt{2\pi}}e^{t_1(Z-t_2)^2-\frac{1}{2}Z^2}\mathbbm{1}_{\{(Z-t_2)^2\geq t_3\}} dZ \nonumber \\
=&~\frac{1}{\sqrt{2\pi}}e^{\frac{t_1t_2^2}{1-2t_1}}\cdot \int_{-\infty}^{+\infty}e^{-\frac{1-2t_1}{2}(Z-\frac{2t_1t_2}{2t_1-1})^2}\mathbbm{1}_{\{(Z-t_2)^2\geq t_3\}} dZ \nonumber \\
=&~(1-2t_1)^{-1/2}e^{\frac{t_1t_2^2}{1-2t_1}} \cdot \int_{-\infty}^{+\infty}\frac{1}{\sqrt{2\pi}}e^{-\frac{1}{2}\tilde{Z}^2}\mathbbm{1}_{\big\{\big(\tilde{Z}-\frac{t_2}{\sqrt{1-2t_1}}\big)^2\geq (1-2t_1)t_3\big\}} d\tilde{Z} \nonumber \\
=&~(1-2t_1)^{-1/2}e^{\frac{t_1t_2^2}{1-2t_1}} \cdot \E\mathbbm{1}_{\big\{\big(Z-\frac{t_2}{\sqrt{1-2t_1}}\big)^2\geq (1-2t_1)t_3\big\}},
\end{align}
where in the third equality we have applied the change of variable $\tilde{Z}=\sqrt{1-2t_1}\big(Z-\frac{2t_1t_2}{2t_1-1}\big)$. Moreover, for any $t_1<\frac{1}{2}, t_2\in \mathbb{R}$, 
\begin{align}
\label{chi:square:form}
\E\big(e^{t_1 V}\mathbbm{1}_{(V\geq t_2)}\big)&=\int_0^{\infty} e^{t_1 V}\frac{1}{2^{(n-1)/2}\Gamma((n-1)/2)}V^{\frac{n-1}{2}-1}e^{-\frac{1}{2}V}\mathbbm{1}_{(V\geq t_2)}dV \nonumber \\
&=(1-2t_1)^{-(n-1)/2} \cdot \int_0^{\infty}\frac{1}{2^{(n-1)/2}\Gamma((n-1)/2)}\tilde{V}^{\frac{n-1}{2}-1}e^{-\frac{1}{2}\tilde{V}}\mathbbm{1}_{\{\tilde{V}\geq (1-2t_1)t_2\}}d\tilde{V} \nonumber \\
&=(1-2t_1)^{-(n-1)/2}\cdot \E\mathbbm{1}_{\{V\geq (1-2t_1)t_2\}},
\end{align}
where in the second equality we have applied the change of variable $\tilde{V}=(1-2t_1)V$.

Now, we rewrite $Y_{m,i}$ in the form 
\begin{align}
Y_{m,i}&=\exp\Bigg(-\frac{\lambda^2}{2n}\Big(\big(\sqrt{n}X_{i,1}-\frac{\sqrt{n}}{\lambda}\|\lambda X_1+z\|_2\big)^2+\|\sqrt{n}X_{i,-1}\|_2^2\Big)+\frac{1}{2}\|\lambda X_1+z\|_2^2\Bigg), \nonumber \\
&\overset{d}{=}\exp\Bigg(-\frac{\lambda^2}{2n}\Big(\big(Z-\frac{\sqrt{n}}{\lambda}\|\lambda X_1+z\|_2\big)^2+V\Big)+\frac{1}{2}\|\lambda X_1+z\|_2^2\Bigg). \label{new:formy}
\end{align}

Equality (i) directly follows from \eqref{new:formy}.

To show Equality (ii), we use \eqref{new:formy} to write out
\begin{align*}
&~\E\Big(Y_{m,i}^{2}\mathbbm{1}_{(Y_{m,i}\leq b)} ~|~ \|\lambda X_{1} + z\|_2\Big) \\
=&~\E\Big(e^{-\frac{\lambda^2}{n}V+\|\lambda X_1+z\|_2^2} \cdot \big(e^{t_1(Z-t_2)^2}\mathbbm{1}_{\{(Z-t_2)^2\geq t_3\}}\big) ~|~ \|\lambda X_{1} + z\|_2 \Big)
\end{align*}
with $t_1=-\frac{\lambda^2}{n}, t_2=\frac{\sqrt{n}}{\lambda}\|\lambda X_1+z\|_2, t_3=n\lambda^{-2} \Big( \|\lambda X_{1}+z\|_2^{2} -2\log b \Big)-V$. Due to the independence between $\{Z,V, \|\lambda X_1+z\|_2\}$, we first use \eqref{single:norm} to compute the expectation with respect to $Z$ to obtain
\begin{align*}
&~\E\Big(Y_{m,i}^{2}\mathbbm{1}_{(Y_{m,i}\leq b)} ~|~ \|\lambda X_{1} + z\|_2\Big) \\
=&~\Big(1+\frac{2\lambda^2}{n}\Big)^{-1/2}e^{\frac{2\lambda^2}{n+2\lambda^2}\|\lambda X_1+z\|_2^2}\cdot \E\big(e^{t_1 V}\mathbbm{1}_{(V\geq t_2)} ~|~ \|\lambda X_{1} + z\|_2\big ),
\end{align*}
where $t_1=-\frac{\lambda^2}{n}, t_2=n\lambda^{-2} \Big( \|\lambda X_{1}+z\|_2^{2} -2\log b \Big)-\frac{n}{n+2\lambda^2}\Big(Z - \frac{n/\lambda^{2}}{\sqrt{2+n/\lambda^{2}}} \|\lambda X_{1}+z\|_2 \Big)^{2}$. Further applying \eqref{chi:square:form} to compute the expectation with respect to $V$ proves Equality (ii). Equation (iii) can be derived in a similar way. We thus do not repeat the arguments. 
\end{proof}

%% file: proof_prop1.tex
\section{Proof of Proposition \ref{sub:bss}}
\label{proof:sub:bss}

\begin{proof}

Given that $\hat{\beta}^{BS}$ in \eqref{eq:MLE_1} is the $k$-sparse vector that minimizes the sum of the squares of the residuals, we can reformulate it as $\hat{\beta}^{BS}_{\hat{Q}}=(X_{\hat{Q}}^TX_{\hat{Q}})^{-1}X^T_{\hat{Q}}y$ and $\hat{\beta}^{BS}_{\hat{Q}^c}=0$, where
\begin{align*}
\hat{Q}\in \argmin_{Q\subseteq [p]: |Q|\leq k} \|y-X_Q(X_Q^TX_Q)^{-1}X^T_Qy\|_2^2.
\end{align*}
Note that $(X_Q^TX_Q)^{-1}$ above is well defined since $X_Q$ is of full column rank with probability one when $n\geq k$. This reformulation leads to the following key characterization,
\begin{align}
\label{key:formulation}
\|X\hat{\beta}^{BS}\|_2^2=\max_{Q\subseteq [p]: |Q|\leq k} \|X_Q(X_Q^TX_Q)^{-1}X^T_Qy\|_2^2.
\end{align}
Define the maximum $k$-sparse eigenvalue as
\[
\bar{\theta}_k(X)=\max_{\Delta\in S_k}\|X\Delta\|_2^2, \quad S_k:=\big\{\Delta\in \mathbb{R}^p: \|\Delta\|_0\leq k, \|\Delta\|_2=1\big\}.
\]
In the rest of the proof, we let $\E_{\beta} (\cdot)$ denote the expected value of a quantity when the regression coefficients of the true model is $\beta$. Then, we can bound the supremum risk of $\hat{\beta}^{BS}$ in the following way:
\begin{align}
\label{first:chain:form}
\sup_{\beta\in \Theta(k,\tau)}\mathbb{E}_{\beta} \|\hat{\beta}^{BS}-\beta \|_2^2 &\geq \mathbb{E}_{0} \|\hat{\beta}^{BS}\|_2^2 \geq \mathbb{E}_{0}\Big(\frac{1}{\bar{\theta}_k(X)}\|X\hat{\beta}^{BS}\|_2^2\Big) \nonumber \\
&=\sigma^2 \cdot \mathbb{E}\Big(\frac{1}{\bar{\theta}_k(X)}\cdot \max_{Q\subseteq [p]: |Q|\leq k} z^TX_Q(X_Q^TX_Q)^{-1}X^T_Qz\Big) \nonumber\\
&\geq \sigma^2 \cdot \mathbb{E}\Big(\frac{1}{\bar{\theta}^2_k(X)}\cdot \max_{Q\subseteq [p]: |Q|\leq k} \|X^T_Qz\|_2^2\Big) \nonumber \\
&=\sigma^2 \cdot \mathbb{E}\Big(\|z\|_2^2/n\Big)\cdot \mathbb{E}\Big(\frac{1}{\bar{\theta}^2_k(X)}\cdot \max_{Q\subseteq [p]: |Q|\leq k} \|X^T_Q(\sqrt{n}z/\|z\|_2)\|_2^2\Big),
\end{align}
where in the first equality we have used \eqref{key:formulation} and the fact $y=\sigma z$ when the true signal $\beta=0$, and the second equality holds since $(X, z/\|z\|_2, \|z\|_2)$ are mutually independent. Now denote $g:=\sqrt{n}X^Tz/\|z\|_2$. It is clear that $g\sim \mathcal{N}(0,I_p)$, and $\max_{Q\subseteq [p]: |Q|\leq k} \|X^T_Q(\sqrt{n}z/\|z\|_2)\|_2=\max_{u\in S_k} u^Tg$. We can thus continue from \eqref{first:chain:form} to obtain
\begin{align}
\label{second:chain:form}
\sup_{\beta\in \Theta(k,\tau)}\mathbb{E}_{\beta} \|\hat{\beta}^{BS}-\beta \|_2^2  &\geq \sigma^2 \cdot \mathbb{E}\Big(\frac{1}{\bar{\theta}^2_k(X)} \cdot \big(\max_{u\in S_k} u^Tg\big)^2\Big)  \nonumber \\
& \geq \frac{4\sigma^2}{9}\cdot \mathbb{E}\Big(\mathbbm{1}_{\bar{\theta}_k(X)\leq 3/2}\cdot \big(\max_{u\in S_k} u^Tg\big)^2\Big) \nonumber \\
& \geq \frac{4\sigma^2}{9}\cdot \Big(\mathbb{E}\big(\max_{u\in S_k} u^Tg\big)^2- \sqrt{\mathbb{E}\big(\max_{u\in S_k} u^Tg\big)^4} \cdot \sqrt{\mathbb{P}(\bar{\theta}_k(X)> 3/2)}\Big),
\end{align}
where we have used Cauchy–Schwarz inequality in the last inequality. Note that $\mathbb{E}\max_{u\in S_k} u^Tg$ is the Gaussian width of the sparse-vector set $S_k$, and it is known that $\mathbb{E}\max_{u\in S_k} u^Tg$ is of order $\sqrt{k\log(p/k)}$. Specifically, we apply Exercise 10.3.9 in \cite{vershynin2018high} to obtain a lower bound and Lemma 15 in \cite{guo2024minimaxlr} to obtain an upper bound:
\begin{align*}
&\mathbb{E}\big(\max_{u\in S_k} u^Tg\big)^2\geq \big(\mathbb{E}\max_{u\in S_k} u^Tg\big)^2\geq C_1k\log(2p/k), \\
&\mathbb{E}\big(\max_{u\in S_k} u^Tg\big)^4 \leq C_2 (k\log(ep/k))^2,
\end{align*}
where $C_1,C_2>0$ are absolute constants. Combining these results with \eqref{second:chain:form}, the proof will be completed if we can further show $\mathbb{P}(\bar{\theta}_k(X)> 3/2)=o(1)$. Indeed, the maximum $k$-sparse eigenvalue concentrates around one. For instance, we can apply Lemma 1 in \cite{guo2024minimaxlr} to obtain 
\[
\mathbb{P}(\bar{\theta}_k(X)> 3/2)\leq 2e^{-C_3k\log p}\rightarrow 0
\]
under the scaling conditions $k/p\rightarrow 0, (k\log p)/n\rightarrow 0$.

\end{proof}

%% file: proof_thm4.tex
\section{Proof of Theorem \ref{thm::second-order-low-snr-minimax}}
\label{proof:reg1:ridge}
Using the scale invariance property discussed in Section \ref{sec:scaling}, without loss of generality, we assume that $\sigma=1$. It is sufficient to prove the following upper and lower bounds:
\begin{align}
\sup_{\beta \in \Theta(k,\mu)}\E_{\beta} \|\betaR(\lambda) - \beta\|_2^{2} &\leq k \mu^{2} \Big( 1 - \frac{k\mu^{2}}{p} \big(1+o(1)\big) \Big), \label{thm3:upper:ridge} \\
R(\Theta(k,\mu),1) &\geq k \mu^{2} \Big( 1 - \frac{k\mu^{2}}{p} \big(1+o(1)\big) \Big). \label{thm3:lower:minimax}
\end{align}

\subsection{Upper bound}
The ridge estimator $\betaR(\lambda)$ is given by
\begin{align*}
    \betaR(\lambda) & \in \argmin_{b\in \mathbb{R}^p}~\|y-Xb\|_2^{2}+\lambda \|b\|_2^{2} \\
    & = (X^{T}X+\lambda I)^{-1}X^{T}y
\end{align*}
for $\lambda>0$. The next lemma obtains the upper bound \eqref{thm3:upper:ridge} with a proper choice of $\lambda$. 

\begin{lemma}
\label{ridge:upper:lemma}
    Assume model \eqref{model::gaussian-model} with $\sigma=1$. Suppose $k/p \rightarrow 0$ and $k /n \rightarrow 0$. As $\mu\rightarrow 0$, the ridge estimator with $\lambda=p(k\mu^{2})^{-1}$ has supremum risk
    \begin{equation*}
        \sup_{\beta \in \Theta(k,\mu)}\E_{\beta} \|\betaR(\lambda) - \beta\|_2^{2} \leq k\mu^{2} \Big(1- \frac{k\mu^{2}}{p} + o\Big( \frac{k\mu^{2}}{p} \Big)\Big).
    \end{equation*}
\end{lemma}
\begin{proof}
Throughout the proof, we write $\betaR(\lambda)$ as $\betaR$ and $\E_{\beta}(\cdot)$ as $\E(\cdot)$, for notational convenience. The risk of the ridge estimator is
\begin{eqnarray}
    && \E \| \betaR - \beta \|_2^{2} = \E \| (X^{T}X+\lambda I)^{-1}X^{T}(X\beta+z) -\beta \|_2^{2} \nonumber\\
    &=& \E \| (X^{T}X+\lambda I)^{-1}(X^{T}X+\lambda I) \beta - (X^{T}X+\lambda I)^{-1}\lambda \beta + (X^{T}X+\lambda I)^{-1}X^{T}z-\beta \|_2^{2} \nonumber\\
    &=& \E \| - (X^{T}X+\lambda I)^{-1}\lambda \beta + (X^{T}X+\lambda I)^{-1}X^{T}z \|_2^{2} \nonumber\\
    &=& \E \| (X^{T}X+\lambda I)^{-1}\lambda \beta \|_2^{2}+\E\|(X^{T}X+\lambda I)^{-1}X^{T}z\|_2^{2}, \label{eq::ridge-estimator-risk-decomposition}
\end{eqnarray}
where the last step used $\E \beta^{T} (X^{T}X+\lambda I)^{-2}X^{T} z = 0$. To deal with the first term in \eqref{eq::ridge-estimator-risk-decomposition}, we denote the spectral decomposition of $X^TX$ by $X^{T}X= Q\Lambda Q^T$, where $Q\in \R^{p\times p}$ is the eigenvector matrix and $\Lambda = \diag(\sigma_{1}, \ldots, \sigma_{p})$ has the eigenvalues of $X^{T}X$ on the diagonal. Define the function $f(x):= \frac{1}{(1+x)^{2}} - (1-2x+3x^{2})$. Using Taylor's theorem, it is direct to verify that $f(x)\leq 0, \forall x\geq 0$. We thus have
\begin{eqnarray*}
    && \bigg(\frac{1}{\lambda}X^{T}X+I\bigg)^{-2} - \bigg( I - \frac{2}{\lambda}X^{T}X + \frac{3}{\lambda^{2}}(X^{T}X)^{2} \bigg) \\
    &=& Q\bigg[ \bigg( \frac{1}{\lambda} \Lambda +I \bigg)^{-2} - \bigg(I - \frac{2}{\lambda}\Lambda + \frac{3}{\lambda^{2}}\Lambda^{2} \bigg) \bigg] Q^T \\
    &=& Q \diag \bigg[ f\Big(\frac{\sigma_{1}}{\lambda}\Big), ~\ldots~, f\Big( \frac{\sigma_{p}}{\lambda} \Big) \bigg]Q^T  \preccurlyeq \mathbf{0}_{p\times p}.
\end{eqnarray*}

Therefore,
\begin{align}
    \E \|(X^{T}X+\lambda I)^{-1} \lambda \beta\|_2^{2} & \leq \E \bigg[ \|\beta\|_2^{2} - \frac{2\beta^TX^TX\beta}{\lambda} + \frac{3\beta^{T}(X^{T}X)^{2}\beta}{\lambda^{2}} \bigg] \nonumber\\
    & = \|\beta\|_2^{2} \cdot \bigg[ 1 - \frac{2}{\lambda} + \frac{3\E \beta^{T}(X^{T}X)^{2}\beta}{\lambda^{2}\|\beta\|_2^{2}} \bigg] \nonumber\\
    & = \|\beta\|_2^{2} \cdot \bigg[ 1 - 2\frac{k \mu^{2}}{p} + 3\Big(\frac{k\mu^{2}}{p}\Big)^{2} \cdot \frac{\E \beta^{T} (X^{T}X)^{2}\beta}{ \|\beta\|_2^{2}} \bigg] \nonumber\\
    &= \|\beta\|_2^{2} \cdot \bigg[ 1 - 2\frac{k \mu^{2}}{p} + 3\Big(\frac{k\mu^{2}}{p}\Big)^{2} \cdot \Big( 1+\frac{p+1}{n} \Big) \bigg] \nonumber\\
    &\leq k\mu^{2}\bigg[ 1 - 2\frac{k \mu^{2}}{p} + o\Big( \frac{k\mu^{2}}{p} \Big)\bigg]. \label{eq::first-term-ridge-risk}
\end{align}
Here, the last inequality is due to the conditions $\|\beta\|_2^2\leq k\mu^2, \forall \beta\in \Theta(k,\mu)$ and $k/p\rightarrow 0,k/n\rightarrow 0,\mu\rightarrow 0$; the first equality uses $\mathbb{E}(X^TX)=I_p$; in the second equality we adopt $\lambda = p(k\mu^{2})^{-1}$; the third equality holds since 
\begin{align*}
\E (X^{T}X)^{2} &=\sum_{i,j}\E(x_ix_i^Tx_jx_j^T)= n(n-1)\mathbb{E}(x_1x_1^Tx_2x_2^T)+n\mathbb{E}(x_1x_1^T\|x_1\|_2^2) \\
&=\frac{n-1}{n}I_p+\frac{1}{n}\cdot \mathbb{E}\Big(\frac{gg^T}{\|g\|^2_2}\Big)\cdot \E\|g\|_2^4 \quad \quad g\sim \mathcal{N}(0,I_p) \\
&=\frac{n-1}{n}I_p+\frac{1}{n}\cdot \frac{1}{p}I_p \cdot (2p+p^2)=\frac{n+p+1}{n}I_p,
\end{align*}
where we have used the fact that in the polar form, $\frac{g}{\|g\|_2}$ is uniformly distributed on the unit sphere and is independent of $\|g\|^2_2\sim \chi^2_p$.

Using the independence between $X$ and $z$, the second term in \eqref{eq::ridge-estimator-risk-decomposition} can be bounded as
\begin{equation}
    \E\|(X^{T}X+\lambda I)^{-1}X^{T}z\|^{2}_2 \leq \frac{1}{\lambda^{2}} \E \|X^{T}z\|_2^{2} = \frac{1}{\lambda^{2}}\cdot \frac{p}{n} \E \|z\|_2^{2} = k\mu^{2} \cdot \frac{k\mu^{2}}{p}. \label{eq::second-term-ridge-risk}
\end{equation}

Combining \eqref{eq::ridge-estimator-risk-decomposition}, \eqref{eq::first-term-ridge-risk} and \eqref{eq::second-term-ridge-risk}, we conclude
\begin{equation*}
    \sup_{\beta \in \Theta(k,\mu)}\E \|\betaR - \beta\|_2^{2} \leq k\mu^{2} \Big(1- \frac{k\mu^{2}}{p} + o\Big( \frac{k\mu^{2}}{p} \Big)\Big).
\end{equation*}
\end{proof}



\subsection{Lower bound}\label{sec::regime1-lower-bound}

We derive the lower bound \eqref{thm3:lower:minimax} based on the independent block prior described in Section \ref{sec::first-order-snr-minimax} except that the signal can now be evenly positive or negative. Specifically, the symmetric independent block prior, denoted by $\pi_{\pm IB}(\lambda;p,k)$, is constructed as follows: divide $\beta\in \mathbb{R}^p$ into $k$ disjoint blocks of size $m=p/k$: $\beta=(\beta^{(1)},\ldots, \beta^{(k)})$; For each block $1\leq j\leq k$, randomly select an index $I \in [m]$, and then set $\beta^{(j)}= \pm \lambda e_I$ each with probability $\frac{1}{2}$; The selection between different blocks are independent. Setting $\lambda=\mu$, from \eqref{baye:risk:lower:bound} we have 
\[
R(\Theta(k,\mu),1)\geq B(\pi_{\pm IB}(\mu; p, k)).
\]
Hence, the lower bound proof is completed by the following proposition. 
\begin{proposition}
\label{thm3:lower:prop4}
    Assume model \eqref{model::gaussian-model} with $\sigma=1$. Suppose $n\rightarrow \infty$ and $p/k\rightarrow \infty$. If $\mu \rightarrow 0$, then the Bayes risk of the symmetric independent block prior satisfies
    \begin{equation*}
        B(\pi_{\pm IB}(\mu; p, k)) \geq k\mu^{2} \Big( 1 - \frac{k\mu^{2}}{p}  + o\Big(\frac{k\mu^{2}}{p}\Big) \Big).
    \end{equation*}
\end{proposition}

\begin{proof}
Recall that $m=p/k$. Define the symmetric spike prior $\pi_{\pm S}(\mu;m)$ for $\beta \in \mathbb{R}^m$ as: select an index $I\in [m]$ uniformly at random and then set $\beta =\pm \mu e_I$ with equal probability. Using exactly the same argument of \eqref{eq::divide-into-k-blocks} and \eqref{eq:lb:complete:classicalminimax} in the proof of Proposition \ref{prop::first-order-block-prior}, we can finish the proof by calculating the Bayes risk for $\pi_{\pm S}(\mu;m)$, as done in the next lemma. 
\end{proof}

\begin{lemma}
    Consider model \eqref{model::gaussian-model} with $\sigma=1$ and $\beta\in \mathbb{R}^{m}$. Suppose $n, m \rightarrow \infty$ and $\mu\rightarrow 0$. Then the Bayes risk satisfies
    \begin{equation*}
        B(\pi_{\pm S}(\mu; m)) \geq \mu^{2} - \frac{\mu^{4}}{m} \Big(1+o(1)\Big).
    \end{equation*}
\end{lemma}

\begin{proof}

For the Bayesian estimator, denoted by $\hat{\beta}:=\E(\beta |y, X)$, it is straightforward to obtain that for $j=1,\ldots, m$,
\[
\hat{\beta}_j=\mu \cdot \frac{\exp(\mu X_{j}^{T}y - \mu^{2}\|X_{j}\|_2^{2}/2) - \exp (-\mu X_{j}^{T}y - \mu^{2} \|X_{j}\|_2^{2}/2)}{\sum_{i=1}^{m}\Big(\exp(\mu X_{i}^{T}y - \mu^{2}\| X_{i}\|_2^{2}/2) + \exp (-\mu X_{i}^{T}y - \mu^{2} \|X_{i}\|_2^{2}/2)\Big)}:=\mu\cdot \mathcal{P}_j,
\]
where $X_i$ denotes the $i$th column of $X$. Let $\E_{\pm \mu e_{i}}$ denote the expectation taken under the model $y=X\beta+z$ with $\beta=\pm \mu e_i$. By the symmetry in the spike prior and the model, the Bayes risk satisfies 
\begin{align}
    B(\pi_{\pm S}(\mu; m)) &= \E_{\mu e_{1}} (\hat{\beta}_{1} - \mu)^{2} + (m-1) \E_{\mu e_{2}} \hat{\beta}_{1}^{2} \nonumber \\
    &\geq \mu^{2} \Big( 1 - 2 \E_{\mu e_{1}} \mathcal{P}_1 + (m-1)\E_{\mu e_{2}} \mathcal{P}_1^{2} \Big). \label{eq::2nd-order-lower-bound-low-SNR-spike-ineq}
\end{align}

Using $y=\mu X_1+z$, we can write 
\begin{align*}
     \E_{\mu e_{1}}\mathcal{P}_1 &= \E\frac{e^{\mu X_{1}^{T}(\mu X_{1}+z) - \mu^{2}\|X_{1}\|_2^{2}/2} - e^{-\mu X_{1}^{T}(\mu X_{1}+z) - \mu^{2} \|X_{1}\|_2^{2}/2}}{\sum_{i=1}^{m}\Big(e^{\mu X_{i}^{T}(\mu X_{1}+z) - \mu^{2}\|X_{i}\|^{2}/2} + e^{-\mu X_{i}^{T}(\mu X_{1}+z) - \mu^{2} \|X_{i}\|^{2}/2}\Big)}\\
     &=  \E\mathcal{P}_1^{(1)}+\E\mathcal{P}_1^{(2)}+\E\mathcal{P}_1^{(3)},
\end{align*}
where 
\begin{align}
    & \mathcal{P}_1^{(1)} := \frac{e^{\mu X_{1}^Tz}\big(e^{\frac{1}{2}\mu^{2}\|X_{1}\|_2^{2}} - e^{-\frac{1}{2}\mu^{2}\|X_{1}\|_2^{2}} \big)}{\sum_{i=1}^{m}e^{- \frac{1}{2}\mu^{2}\|X_{i}\|_2^{2}}\big(e^{\mu X_{i}^{T}(\mu X_{1}+z)} + e^{-\mu X_{i}^{T}(\mu X_{1}+z)}\big)}, \label{eq::p_1^(1)}\\
    & \mathcal{P}_1^{(2)} := \frac{e^{-\mu X_{1}^Tz}\big( e^{-\frac{1}{2}\mu^{2}\|X_{1}\|_2^{2}} - e^{-\frac{3}{2}\mu^{2}\|X_{1}\|_2^{2}} \big)}{\sum_{i=1}^{m}e^{- \frac{1}{2}\mu^{2}\|X_{i}\|_2^{2}}\big(e^{\mu X_{i}^{T}(\mu X_{1}+z)} + e^{-\mu X_{i}^{T}(\mu X_{1}+z)}\big)}, \label{eq::p_1^(2)}\\
    & \mathcal{P}_1^{(3)} := \frac{e^{\mu x_{1}^{T}z - \frac{1}{2}\mu^{2}\|X_{1}\|_2^{2}} - e^{-\mu X_{1}^{T}z - \frac{1}{2}\mu^{2}\|X_{1}\|_2^{2}}}{\sum_{i=1}^{m}e^{- \frac{1}{2}\mu^{2}\|X_{i}\|_2^{2}}\big(e^{\mu X_{i}^{T}(\mu X_{1}+z)} + e^{-\mu X_{i}^{T}(\mu X_{1}+z)}\big)}. \label{eq::p_1^(3)}
\end{align}
Lemmas \ref{lem::2nd-order-lower-bound-low-SNR-p1&2}-\ref{lem::2nd-order-lower-bound-low-SNR-p3} together imply
\begin{equation}
\label{key:part1}
     \E_{\mu e_{1}}\mathcal{P}_1 \leq  \frac{\mu^{2}}{m}\Big(1+o(1)\Big).
\end{equation}
Using $y=\mu X_2+z$, we can write 
  \begin{equation}
  \label{p12:expression}
       \E_{\mu e_{2}} \mathcal{P}_1^{2} = \E \frac{\Big(e^{ \mu X_{1}^{T} (\mu X_{2}+z) - \frac{1}{2}\mu^{2}\|X_{1}\|_2^{2}} - e^{-\mu X_{1}^{T}(\mu X_{2}+z) - \frac{1}{2}\mu^{2}\|X_{1}\|_2^{2}}\Big)^{2}}{\Big(\sum_{i=1}^{m}e^{- \frac{1}{2}\mu^{2}\|X_{i}\|_2^{2}} \big(e^{\mu X_{i}^{T}(\mu X_{2}+z)} + e^{-\mu X_{i}^{T}(\mu X_{2} + z)}\big) \Big)^{2}}.
    \end{equation}
Lemma \ref{lem::2nd-order-lower-bound-low-SNR-p^2} shows that
\begin{equation}
\label{key:part2}
    (m-1)\E_{\mu e_{2}} \mathcal{P}_1^{2} \geq (m-1)\cdot \frac{\mu^{2}}{m^{2}} \Big(1+o(1) \Big) = \frac{\mu^{2}}{m}\Big(1+o(1)\Big).
\end{equation}
Thus, combining results \eqref{eq::2nd-order-lower-bound-low-SNR-spike-ineq}, \eqref{key:part1} and \eqref{key:part2} completes the proof. 
    
\end{proof}

\begin{lemma}\label{lem::2nd-order-lower-bound-low-SNR-p1&2}
    Consider $\sqrt{n}X_1,\ldots, \sqrt{n}X_m, z \overset{i.i.d.}{\sim}\mathcal{N}(0,I_n)$. Suppose $n, m \rightarrow \infty$ and $\mu \rightarrow 0$. Then $\mathcal{P}_1^{(1)}$ and $\mathcal{P}_1^{(2)}$ defined in \eqref{eq::p_1^(1)} and \eqref{eq::p_1^(2)} satisfy
    \begin{equation*}
        \text{(i)}~~\E \mathcal{P}_1^{(1)} \leq \frac{\mu^2}{2m} (1+o(1)), \qquad \text{(ii)}~~\E\mathcal{P}_1^{(2)} \leq \frac{\mu^2}{2m} (1+o(1)).
    \end{equation*}
\end{lemma}
\begin{proof}
    Since the numerators in both $\mathcal{P}_1^{(1)}$ and $\mathcal{P}_1^{(2)}$ are nonnegative, we can first bound their common denominator using the basic inequality $t+\frac{1}{t}\geq 2, \forall t>0$:
    \begin{equation}
        \sum_{i=1}^{m}e^{- \frac{1}{2}\mu^{2}\|X_{i}\|_2^{2}}\big(e^{\mu X_{i}^{T}(\mu X_{1}+z)} + e^{-\mu X_{i}^{T}(\mu X_{1}+z)}\big) \geq 2\sum_{i=1}^{m} e^{-\frac{1}{2}\mu^{2}\|X_{i}\|_2^{2}}. \label{common:den}
    \end{equation}
   Therefore, to show (i), it is sufficient to show
    \begin{equation}
    \E\frac{me^{\mu X_{1}^Tz}\big(e^{\frac{1}{2}\mu^{2}\|X_{1}\|_2^{2}} - e^{-\frac{1}{2}\mu^{2}\|X_{1}\|_2^{2}} \big)}{\mu^2\sum_{i=1}^{m} e^{-\frac{1}{2}\mu^{2}\|X_{i}\|_2^{2}}}:=\E \mathcal{W}_n \rightarrow 1. \label{part1:suf}
    \end{equation}
   We prove the above result through dominated convergence theorem in two steps: (1) we show $\mathcal{W}_n \overset{p}{\rightarrow} 1$; (2) we show $\E\mathcal{W}^2_n=O(1)$ so that $\mathcal{W}_n$ is uniformly integrable. Step 1 follows if 
   \begin{align*}
    (a)~~ \mu^{-2}e^{\mu X_{1}^Tz}\big(e^{\frac{1}{2}\mu^{2}\|X_{1}\|_2^{2}} - e^{-\frac{1}{2}\mu^{2}\|X_{1}\|_2^{2}} \big)  \overset{p}{\rightarrow} 1,  \quad (b)~~\frac{1}{m}\sum_{i=1}^{m} e^{-\frac{1} {2}\mu^{2}\|X_{i}\|_2^{2}}  \overset{p}{\rightarrow} 1.
   \end{align*}
Result (a) holds by applying law of large numbers and continuous mapping theorem to obtain
\[
e^{\mu X_{1}^{T}z} \overset{p}{\rightarrow} 1 {\rm ~~and~~} \mu^{-2}\big(e^{\frac{1}{2}\mu^{2}\|X_{1}\|_2^{2}} - e^{-\frac{1}{2}\mu^{2}\|X_{i}\|_2^{2}}\big) \overset{p}{\rightarrow} 1.
\]
Result (b) is due to the fact that as $n,m\rightarrow \infty, \mu \rightarrow 0$,
    \begin{itemize}
        \item $\E\Big(\frac{1}{m}\sum_{i=1}^{m} e^{-\frac{1} {2}\mu^{2}\|X_{i}\|_2^{2}}\Big)=\Big( 1+ \frac{\mu^{2}}{n} \Big)^{-\frac{n}{2}}\rightarrow 1$
                \item $\begin{aligned}[t]
        \Var\Big(\frac{1}{m}\sum_{i=1}^{m} e^{-\frac{1} {2}\mu^{2}\|X_{i}\|_2^{2}} \Big)& =\frac{1}{m}\Big[ \E e^{-\mu^{2}\|X_{i}\|_2^{2}}- \Big( \E e^{-\frac{1} {2}\mu^{2}\|X_{i}\|_2^{2}} \Big)^{2}\Big] \\
        & = \frac{1}{m}\Big[\Big( 1+ \frac{2\mu^{2}}{n} \Big)^{-\frac{n}{2}} - \Big( 1+ \frac{\mu^{2}}{n} \Big)^{-n} \Big] \rightarrow 0.  \end{aligned}$
    \end{itemize}
To complete Step 2, since $\mathcal{W}_n\geq 0$, we first have
    \begin{align*}
        \mathcal{W}_n & \leq \frac{me^{\mu X_{1}^Tz}\big(e^{\frac{1}{2}\mu^{2}\|X_{1}\|_2^{2}} - e^{-\frac{1}{2}\mu^{2}\|X_{1}\|_2^{2}} \big)}{\mu^2\sum_{i=2}^{m} e^{-\frac{1}{2}\mu^{2}\|X_{i}\|_2^{2}}} \\
        &\leq \frac{m}{\mu^2}\cdot \frac{e^{\mu X_{1}^Tz}\big(e^{\frac{1}{2}\mu^{2}\|X_{1}\|_2^{2}} - e^{-\frac{1}{2}\mu^{2}\|X_{1}\|_2^{2}} \big)}{(m-1)e^{-\frac{\mu^2}{2(m-1)}\sum_{i=2}^m\|X_i\|_2^2}} \\
        &= \frac{m}{(m-1)\mu^{2}} \cdot e^{\frac{\mu^2}{2(m-1)}\sum_{i=2}^m\|X_i\|_2^2} \cdot e^{\mu X_{1}^Tz}\big(e^{\frac{1}{2}\mu^{2}\|X_{1}\|_2^{2}} - e^{-\frac{1}{2}\mu^{2}\|X_{1}\|_2^{2}} \big),
    \end{align*}
    where the second inequality uses the arithmetic-geometric inequality. Hence, 
    \begin{align*}
        &~\E \mathcal{W}^2_n \\
        \leq &~ \frac{m^2}{(m-1)^2\mu^{4}} \cdot \E e^{\frac{\mu^2}{(m-1)}\sum_{i=2}^m\|X_i\|_2^2} \cdot \E e^{2\mu X_{1}^Tz}\big(e^{\frac{1}{2}\mu^{2}\|X_{1}\|_2^{2}} - e^{-\frac{1}{2}\mu^{2}\|X_{1}\|_2^{2}} \big)^2 \\
        =&~\frac{m^2}{(m-1)^2\mu^{4}} \cdot \E e^{\frac{\mu^2}{(m-1)}\sum_{i=2}^m\|X_i\|_2^2} \cdot \E e^{2\mu^2\|X_1\|_2^2}\big(e^{\frac{1}{2}\mu^{2}\|X_{1}\|_2^{2}} - e^{-\frac{1}{2}\mu^{2}\|X_{1}\|_2^{2}} \big)^2 \\
        =&~\frac{m^2}{(m-1)^2\mu^{4}} \cdot \Big(1-\frac{2\mu^2}{n(m-1)}\Big)^{-\frac{n(m-1)}{2}} \cdot \bigg(\Big(1-\frac{6\mu^2}{n}\Big)^{-\frac{n}{2}}+\Big(1-\frac{2\mu^2}{n}\Big)^{-\frac{n}{2}}-2\Big(1-\frac{4\mu^2}{n}\Big)^{-\frac{n}{2}}\bigg).
    \end{align*}
    Here, the first equality is obtained by conditioning on $\|X_1\|_2$ and using $(X_1^Tz~|~\|X_1\|_2)\sim \mathcal{N}(0,\|X_1\|_2^2)$, and the second equality uses the moment-generating function of chi-squared distribution. In the regime where $n,m\rightarrow \infty$ and $\mu\rightarrow 0$, based on the asymptotic result $\big(1-\frac{t}{n}\big)^{-\frac{n}{2}}=1+\frac{t}{2}+\frac{1+o(1)}{8}t^2$ for $t=o(1)$, we can continue from the above to achieve
    \begin{align*}
    &~\E \mathcal{W}^2_n \\
    \leq &~\frac{1+o(1)}{\mu^4}\cdot \Big(1+3\mu^2+\frac{9+o(1)}{2}\mu^4+1+\mu^2+\frac{1+o(1)}{2}\mu^4-2(1+2\mu^2+(2+o(1))\mu^4)\Big)\\
    =& 1+o(1).
    \end{align*}
    
 To show (ii), with the bound \eqref{common:den} for the denominator of $\mathcal{P}^{(2)}$, it is sufficient to prove
 \begin{align*}
 \E\frac{me^{-\mu X_{1}^Tz-\mu^2\|X_1\|_1^2}\big(e^{\frac{1}{2}\mu^{2}\|X_{1}\|_2^{2}} - e^{-\frac{1}{2}\mu^{2}\|X_{1}\|_2^{2}} \big)}{\mu^2\sum_{i=1}^{m} e^{-\frac{1}{2}\mu^{2}\|X_{i}\|_2^{2}}}\leq (1+o(1)).
 \end{align*}
This can be quickly confirmed by using \eqref{part1:suf} together with the following simple argument
\begin{align*}
&~\E\frac{me^{-\mu X_{1}^Tz-\mu^2\|X_1\|_1^2}\big(e^{\frac{1}{2}\mu^{2}\|X_{1}\|_2^{2}} - e^{-\frac{1}{2}\mu^{2}\|X_{1}\|_2^{2}} \big)}{\mu^2\sum_{i=1}^{m} e^{-\frac{1}{2}\mu^{2}\|X_{i}\|_2^{2}}} \\
\leq &~\E\frac{me^{-\mu X_{1}^Tz}\big(e^{\frac{1}{2}\mu^{2}\|X_{1}\|_2^{2}} - e^{-\frac{1}{2}\mu^{2}\|X_{1}\|_2^{2}} \big)}{\mu^2\sum_{i=1}^{m} e^{-\frac{1}{2}\mu^{2}\|X_{i}\|_2^{2}}}=\E\frac{me^{\mu X_{1}^Tz}\big(e^{\frac{1}{2}\mu^{2}\|X_{1}\|_2^{2}} - e^{-\frac{1}{2}\mu^{2}\|X_{1}\|_2^{2}} \big)}{\mu^2\sum_{i=1}^{m} e^{-\frac{1}{2}\mu^{2}\|X_{i}\|_2^{2}}}.
\end{align*}

\end{proof}

\begin{lemma}\label{lem::2nd-order-lower-bound-low-SNR-p3}
    Under the same conditions of Lemma \ref{lem::2nd-order-lower-bound-low-SNR-p1&2}, the $\mathcal{P}_1^{(3)}$ defined in \eqref{eq::p_1^(3)} satisfies
    \begin{equation*}
        \E \mathcal{P}_1^{(3)} = o\Big(\frac{\mu^{2}}{m}\Big).
    \end{equation*}
\end{lemma}

\begin{proof}
We first introduce two notations:
\begin{align*}
&A:=\sum_{i=1}^{m}e^{- \frac{1}{2}\mu^{2}\|X_{i}\|_2^{2}}\big(e^{\mu X_{i}^{T}(\mu X_{1}+z)} + e^{-\mu X_{i}^{T}(\mu X_{1}+z)}\big) \\
&B:=\sum_{i=1}^{m}e^{- \frac{1}{2}\mu^{2}\|X_{i}\|_2^{2}}\big(e^{\mu X_{i}^{T}(\mu X_{1}-z)} + e^{-\mu X_{i}^{T}(\mu X_{1}-z)}\big)
\end{align*}
Note that $A$ is the denominator of $\mathcal{P}_1^{(3)}$, and $B$ is obtained by replacing $z$ in $A$ with $-z$. Given that flipping the sign of $z$ does not change the joint distribution, we have
\begin{align*}
\E \mathcal{P}_1^{(3)}&=\E\frac{e^{\mu x_{1}^{T}z - \frac{1}{2}\mu^{2}\|X_{1}\|_2^{2}}}{A}-\E\frac{e^{-\mu X_{1}^{T}z - \frac{1}{2}\mu^{2}\|X_{1}\|_2^{2}}}{A} \\
&=\E\frac{e^{\mu x_{1}^{T}z - \frac{1}{2}\mu^{2}\|X_{1}\|_2^{2}}}{A}-\E\frac{e^{\mu X_{1}^{T}z - \frac{1}{2}\mu^{2}\|X_{1}\|_2^{2}}}{B} \\
&= \E\frac{(B-A)e^{\mu x_{1}^{T}z - \frac{1}{2}\mu^{2}\|X_{1}\|_2^{2}}}{AB}:=\E\frac{\Delta_1}{AB}+\E\frac{\Delta_2}{AB},
\end{align*}
where we split the summation in $B-A$ into two parts $\sum_{i=1}^1$ and $\sum_{i=2}^m$, leading to
\begin{align}
\label{delta2:two:term}
&\Delta_1=(1-e^{-2\mu^2\|X_1\|^2_2})(1-e^{2\mu X_1^Tz}), \nonumber \\
&\Delta_2=e^{\mu x_{1}^{T}z - \frac{1}{2}\mu^{2}\|X_{1}\|_2^{2}}\sum_{i=2}^m(e^{-2\mu^2 X_i^TX_1}-1)e^{\mu X_i^T(\mu X_1+z)-\frac{1}{2}\mu^2\|X_i\|_2^2} \nonumber \\
& \quad \quad +e^{\mu x_{1}^{T}z - \frac{1}{2}\mu^{2}\|X_{1}\|_2^{2}}\sum_{i=2}^m(e^{2\mu^2 X_i^TX_1}-1)e^{-\mu X_i^T(\mu X_1+z)-\frac{1}{2}\mu^2\|X_i\|_2^2}:=\Delta_{21}+\Delta_{22}.
\end{align}

Let us first show $\E \frac{\Delta_{1}}{BA} = o\Big( \frac{\mu^{2}}{m} \Big)$. We use the same argument as in the proof of Lemma \ref{lem::2nd-order-lower-bound-low-SNR-p1&2} to bound the denominator: $A\geq 2m e^{-\frac{\mu^2}{2m}\sum_{i=1}^m\|X_i\|_2^2}, B \geq 2m e^{-\frac{\mu^2}{2m}\sum_{i=1}^m\|X_i\|_2^2}$. Furthermore, since $\Delta_1 \geq 0 \Leftrightarrow X_1^Tz\leq 0$, we have
    \begin{align*}
         \E \frac{\Delta_{1}}{AB} &\leq \E\frac{(1-e^{-2\mu^2\|X_1\|^2_2})(1-e^{2\mu X_1^Tz})\mathbbm{1}_{(X_{1}^{T}z\leq 0)}}{AB}\\ 
        &\leq  \frac{1}{4m^2}\cdot \E e^{\frac{\mu^{2}}{m}\sum_{i=2}^{m}\|X_{i}\|_2^{2}} \cdot \E \Big(\big(e^{ \frac{\mu^{2}}{m}\|X_{1}\|_2^{2}}- e^{ \frac{1-2m}{m}\mu^{2}\|X_{1}\|_2^{2}} \big) (1-e^{2\mu X_1^Tz})\mathbbm{1}_{(X_{1}^{T}z\leq 0)}  \Big)\\
        &\leq   \frac{1}{4m^{2}} \cdot \Big( 1- \frac{2\mu^{2}}{nm}\Big)^{-\frac{n(m-1)}{2}}\cdot \E \Big(\big(e^{ \frac{\mu^{2}}{m}\|X_{1}\|_2^{2}}- e^{ \frac{1-2m}{m}\mu^{2}\|X_{1}\|_2^{2}} \big) (-2\mu X_1^Tz)\mathbbm{1}_{(X_{1}^{T}z\leq 0)}  \Big),
    \end{align*}
    where in the last inequality we have used the moment-generating function of chi-squared distribution and the basic inequality $1-e^{-t} \leq t$, $\forall t\geq 0$. Given that $(X_1^Tz~|~\|X_1\|_2)\sim \mathcal{N}(0,\|X_1\|_2^2)$, the expectation in the last line can be computed as
    \begin{align*}
       &~ \E \Big(\big(e^{ \frac{\mu^{2}}{m}\|X_{1}\|_2^{2}}- e^{ \frac{1-2m}{m}\mu^{2}\|X_{1}\|_2^{2}} \big) (-2\mu X_1^Tz)\mathbbm{1}_{(X_{1}^{T}z\leq 0)}  \Big) \\
    =&~ \E \Big(\big(e^{ \frac{\mu^{2}}{m}\|X_{1}\|_2^{2}}- e^{ \frac{1-2m}{m}\mu^{2}\|X_{1}\|_2^{2}} \big) (-2\mu\|X_1\|_2)  \Big)\cdot \E\big(w \mathbbm{1}_{(w\leq 0)}\big) \quad \quad w\sim \mathcal{N}(0,1) \\
    =&~\frac{2\mu\Gamma(\frac{n+1}{2})}{\sqrt{\pi n}\Gamma(\frac{n}{2})}\Big(\Big(1-\frac{2\mu^2}{mn}\Big)^{-\frac{n+1}{2}}-\Big(1-\frac{2(1-2m)\mu^2}{mn}\Big)^{-\frac{n+1}{2}}\Big)=O(\mu^3),
    \end{align*}
    where the second equality uses the fact that for a chi-squared random variable $Q\sim \chi^2_n$,
    \begin{align}
    \label{chi:squre:mom}
    \E(e^{t Q}\sqrt{Q})=\frac{\sqrt{2}\Gamma(\frac{n+1}{2})}{\Gamma(\frac{n}{2})}(1-2t)^{-\frac{n+1}{2}}, \quad \forall t<\frac{1}{2}.
    \end{align}
    It thus follows that $\E \frac{\Delta_{1}}{AB}\leq O\Big(\frac{\mu^3}{m^2}\Big)=o\Big(\frac{\mu^2}{m}\Big)$.

    It remains to show $\E \frac{\Delta_{2}}{AB} = o \Big( \frac{\mu^{2}}{m} \Big)$. Referring to $\Delta_{21},\Delta_{22}$ in \eqref{delta2:two:term}, since flipping the signs of $\{X_i,i=2,\ldots, m\}$ does not change the joint distribution, we obtain
    \begin{align*}
        \frac{m}{\mu^{2}} \E \frac{\Delta_{2}}{AB} &= \frac{2m}{\mu^{2}} \E \frac{\Delta_{22}}{AB} \\
        &= \frac{2m(m-1)}{\mu^{2}} \cdot  \E\frac{e^{\mu x_{1}^{T}z - \frac{1}{2}\mu^{2}\|X_{1}\|_2^{2}}(e^{2\mu^2 X_2^TX_1}-1)e^{-\mu X_2^T(\mu X_1+z)-\frac{1}{2}\mu^2\|X_2\|_2^2}}{AB} \\
       & \overset{(a)}{\leq} \frac{m-1}{2m\mu^2} \cdot \E e^{\frac{\mu^2}{m}\sum_{i=3}^m\|X_i\|_2^2}\cdot \\
       & \quad \quad \E \Big(e^{\mu(X_1-X_2)^Tz}(e^{\mu^2X_2^TX_1}-e^{-\mu^2X_2^TX_1})e^{\frac{2-m}{2m}\mu^2(\|X_1\|_2^2+\|X_2\|_2^2)}\mathbbm{1}_{(X_2^TX_1\geq 0)}\Big) \\
       &\overset{(b)}{=}\frac{m-1}{2m\mu^2} \cdot \Big(1-\frac{2\mu^2}{mn}\Big)^{-\frac{(m-2)n}{2}} \cdot \E \Big((1-e^{-2\mu^2X_2^TX_1})e^{\frac{\mu^2}{m}(\|X_1\|_2^2+\|X_2\|_2^2)}\mathbbm{1}_{(X_2^TX_1\geq 0)}\Big) \\
       & \overset{(c)}{\leq} \frac{m-1}{2m\mu^2} \cdot \Big(1-\frac{2\mu^2}{mn}\Big)^{-\frac{(m-2)n}{2}} \cdot \E \Big(e^{\frac{\mu^2}{m}(\|X_1\|_2^2+\|X_2\|_2^2)}(2\mu^2X_2^TX_1)\mathbbm{1}_{(X_2^TX_1\geq 0)}\Big) \\
       &\overset{(d)}{=}\frac{m-1}{2m\mu^2} \cdot \Big(1-\frac{2\mu^2}{mn}\Big)^{-\frac{(m-2)n}{2}} \cdot \Big(\E\big(e^{\frac{\mu^2}{m}\|X_1\|^2_2}\|X_1\|_2\big)\Big)^2\cdot \E\Big(\frac{2\mu^2 X_2^TX_1}{\|X_2\|_2\|X_1\|_2}\mathbbm{1}_{(X_2^TX_1\geq 0)}\Big) \\
       &\overset{(e)}{=}\frac{m-1}{2m\mu^2} \cdot \Big(1-\frac{2\mu^2}{mn}\Big)^{-\frac{(m-2)n}{2}} \cdot \frac{2\Gamma^2(\frac{n+1}{2})}{n\Gamma^2(\frac{n}{2})}\Big(1-\frac{2\mu^2}{mn}\Big)^{-(n+1)}\cdot O\Big(\frac{\mu^2}{\sqrt{n}}\Big)=O(n^{-1/2}).
    \end{align*}
In the above, (a) applies the same bound as before: $A\geq 2m e^{-\frac{\mu^2}{2m}\sum_{i=1}^m\|X_i\|_2^2}, B \geq 2m e^{-\frac{\mu^2}{2m}\sum_{i=1}^m\|X_i\|_2^2}$; (b) uses the moment-generating function of chi-squared distribution and also computes expectation with respect to $z$ conditioning on $(X_1,X_2)$; (c) is due to the fact that $1-e^{-t} \leq t$, $\forall t\geq 0$; (d) holds by the mutual independence between $\{\|X_1\|_2,\|X_2\|_2, \frac{X_2^TX_1}{\|X_1\|_2\|X_2\|_2}\}$; (e) uses the earlier result \eqref{chi:squre:mom}.

\end{proof}

\begin{lemma}\label{lem::2nd-order-lower-bound-low-SNR-p^2}
    Under the same conditions of Lemma \ref{lem::2nd-order-lower-bound-low-SNR-p1&2}, the $\E_{\mu e_{2}} \mathcal{P}_1^{2}$ defined in \eqref{p12:expression} satisfies
    \begin{equation*}
        \E_{\mu e_{2}} \mathcal{P}_1^{2} \geq \frac{\mu^{2}}{m^{2}} \Big( 1+o(1) \Big).
    \end{equation*}
\end{lemma}
\begin{proof}
    Observe that the numerator in \eqref{p12:expression} is free from $(X_{3}, \ldots, X_{m})$. By conditioning on $(X_{1}, X_{2}, z)$, we can apply Jensen's inequality (on function $f(x):=\frac{1}{(x+c)^2}$) to first compute the conditional expectation inside the denominator, 
    \begin{align*}
        &~ \E \Big(\sum_{i=3}^{m} e^{ \mu X_{i}^{T}(\mu X_{2}+z) - \frac{\mu^{2}}{2}\|X_{i}\|_2^{2}} + e^{-\mu X_{i}^{T}(\mu X_{2}+z) - \frac{\mu^{2}}{2}\|X_{i}\|_2^{2}} \Big| (X_{1},X_{2},z) \Big) \\
        =&~ 2(m-2) \Big( 1+\frac{\mu^{2}}{n} \Big)^{-\frac{n}{2}} \exp\Big(\frac{\mu^2}{2(n+\mu^2)} \|\mu X_{2}+z\|_2^{2}\Big), \\
        \leq &~2(m-2)\exp\Big(\frac{\mu^2}{2n} \|\mu X_{2}+z\|_2^{2}\Big),
    \end{align*}
    where the equality uses the moment-generating function of noncentral chi-squared distribution. 
    Thus,
    \begin{align}
    \label{simplify:one}
         \E_{\mu e_{2}} \mathcal{P}_1^{2}  \geq & ~ \E \bigg[ \Big(e^{\mu X_{1}^{T}(\mu X_{2}+z) - \frac{1}{2}\mu^{2} \|X_{1}\|_2^{2}} - e^{ -\mu X_{1}^{T}(\mu X_{2} + z) - \frac{1}{2}\mu^{2}\|X_{1}\|_2^{2} } \Big)^{2} \cdot  
         \nonumber \\
        & \quad \quad   \Big( e^{ \mu X_{1}^{T}(\mu X_{2}+z) - \frac{1}{2}\mu^{2}\|X_{1}\|_2^{2}} + e^{-\mu X_{1}^{T}(\mu X_{2}+z) - \frac{1}{2}\mu^{2}\|X_{1}\|_2^{2}} \nonumber \\
        & \quad \quad + e^{\mu X_{2}^{T}z + \frac{1}{2}\mu^{2}\|X_{2} \|_2^{2}} + e^{ -\mu X_{2}^{T} z - \frac{3}{2}\mu^{2}\| X_{2} \|_2^{2}} \nonumber \\
        & \quad \quad + 2(m-2) e^{\frac{\mu^2}{2n} \|\mu X_{2}+z\|_2^{2}} \Big)^{-2}\bigg].
    \end{align}
    To simplify the above lower bound, we note that the numerator only depends on $X_1$ and $v:=\mu X_2+z$. We construct a random variable $\bar{v}:= -\frac{n}{\mu} X_{2}+z$ such that $(X_1,v,\bar{v})$ are mutually independent. It is clear that the denominator can be written as a function of $(X_1,v,\bar{v})$. Hence, we can again apply Jensen's inequality to compute the conditional expectation with respect to $\bar{v}$ (conditional on $(X_1,v)$) inside the denominator,
    \begin{align*}
    &~\E \big(\exp\big(\mu X_{2}^{T}z + \frac{1}{2}\mu^{2}\|X_{2} \|_2^{2}\big)~|~(X_1,v)\big) \\
    =&~\E \bigg(\exp\Big(-\frac{\mu^4}{2(n+\mu^2)^2}\|\bar{v}\|_2^2-\frac{n\mu^2}{(n+\mu^2)^2}v^T\bar{v}+\frac{\mu^4+2n\mu^2}{2(n+\mu^2)^2}\|v\|_2^2\Big)~|~(X_1,v)\bigg) \\
    =&~\Big(1+\frac{\mu^2}{n+\mu^2}\Big)^{-\frac{n}{2}}\cdot \exp\Big(\frac{\mu^2(3n+2\mu^2)}{2(n+\mu^2)(n+2\mu^2)}\cdot \|\mu X_2+z\|^2_2\Big) \\
    \leq&~\exp\Big(\frac{3\mu^2}{2n}\cdot \|\mu X_2+z\|^2_2\Big),
    \end{align*}
where the second equality is based on the result: for $\zeta\sim \mathcal{N}(0,\sigma^2 I_n)$ and constants $t>0, a\in \mathbb{R}^n$,
\[
\E \exp\Big(-\frac{t}{2}\|\zeta\|_2^2-a^T\zeta\Big)=(1+t\sigma^2)^{-\frac{n}{2}}\cdot \exp\Big(\frac{\sigma^2\|a\|_2^2}{2(1+t\sigma^2)}\Big).
\]
Similarly, we have
\begin{align*}
 &~\E \big(\exp\big(-\mu X_{2}^{T}z - \frac{3}{2}\mu^{2}\|X_{2} \|_2^{2}\big)~|~(X_1,v)\big) \\
  =&~\E \bigg(\exp\Big(-\frac{\mu^4}{2(n+\mu^2)^2}\|\bar{v}\|_2^2+\frac{n\mu^2+2\mu^4}{(n+\mu^2)^2}v^T\bar{v}-\frac{3\mu^4+2n\mu^2}{2(n+\mu^2)^2}\|v\|_2^2\Big)~|~(X_1,v)\bigg) \\ 
  =&~\Big(1+\frac{\mu^2}{n+\mu^2}\Big)^{-\frac{n}{2}}\cdot \exp\Big(-\frac{\mu^2}{2(n+\mu^2)}\cdot \|\mu X_2+z\|^2_2\Big) \leq 1.
\end{align*}
Combining the above two results and applying the bound $(e^{t}-e^{-t})^2\geq 4t^2, \forall t\in \mathbb{R}$ to the numerator of \eqref{simplify:one}, we obtain
\begin{align}
\label{simplify:two}
 \E_{\mu e_{2}} \mathcal{P}_1^{2}  \geq & ~\E\frac{4\mu^2(X_1^T(\mu X_2+z))^2e^{-\mu^2\|X_1\|_2^2}}{\big(e^{\mu X_1^T(\mu X_2+z)}+e^{-\mu X_1^T(\mu X_2+z)}+1+e^{\frac{3\mu^2}{2n}\|\mu X_2+z\|_2^2}+2(m-2)e^{\frac{\mu^2}{2n}\|\mu X_2+z\|_2^2}\big)^2}.
\end{align}
To further evaluate the lower bound above, we let $Q\in \R^{n\times n}$ be an orthogonal matrix (only depends on $(X_2,z)$) with the first row equal to $\frac{\mu X_2+z}{\|\mu X_2+ z\|_2}$. Denote $QX_1:=w=(w_1,w_{-1})$ where $w_1\in \mathbb{R}, w_{-1} \in  \mathbb{R}^{n-1}$ represent the first and remaining coordinates of $w$ respectively. It is direct to verify the following:
\begin{itemize}
\item $X_1^T(\mu X_2+z)=w_1\|\mu X_2+z\|_2$ 
\item $w\sim \mathcal{N}(0,\frac{1}{n}I_n)$
\item $(w_1, w_{-1}, \|\mu X_2 +z\|_2)$ are mutually independent
\item The ratio in \eqref{simplify:two} can be written as a function of $(w_1,w_{-1},\|\mu X_2+z\|_2)$.
\end{itemize}
Therefore, we can continue from \eqref{simplify:two} to have
\begin{align}
\label{simplify:three}
\E_{\mu e_{2}} \mathcal{P}_1^{2}  \geq & ~\frac{4\mu^2}{M}\cdot \E\big(\|\mu X_2+z\|_2^2\mathbbm{1}_{(\|\mu X_2+z\|_2\leq t_1\sqrt{n})}\big)\cdot \E\big(w_1^2e^{-\mu^2w_1^2}\mathbbm{1}_{(|w_1|\leq t_2/\sqrt{n})}\big)\cdot \E e^{-\mu^2\|w_{-1}\|_2^2},
\end{align}
where $t_1,t_2>0$ are constants that will be specified shortly, and 
\[
M:=\big(e^{\mu t_1 t_2}+e^{-\mu t_1t_2}+1+e^{\frac{3}{2}t_1^2\mu^2}+(2m-4)e^{\frac{1}{2}t_1^2\mu^2}\big)^2.
\]
Now, we set $t_1=t_2=\mu^{-\frac{1}{4}}$ and compute the three expectations in \eqref{simplify:three},
    \begin{itemize}
         \item Since $\frac{t_1^2n}{\mu^2/n+1}-n-2>0$, 
         \begin{align*}
          &~ \E \|\mu X_{2}+z\|_2^{2} \mathbbm{1}_{(\|\mu X_{2} + z\|_2\leq t_1\sqrt{n})} = n\Big( 1 + \frac{\mu^{2}}{n}\Big) \cdot \mathbb{P}\Big(\chi_{n+2}^{2} \leq \frac{t_1^{2}n}{\mu^{2}/n + 1} \Big) \\
        \geq &~ n\Big( 1 + \frac{\mu^{2}}{n}\Big) \cdot \bigg[1-\exp\Big(-\frac{(\mu^2/n+1)\big(\frac{t_1^2n}{\mu^2/n+1}-n-2\big)^2}{4nt_1^2}\Big) \bigg]=n(1+o(1)),
        \end{align*}
        where the inequality is due to Lemma \ref{lem::concentration-non-central-chisquare}. 
        \item $\E w_{1}^{2} e^{-\mu^{2} w_{1}^{2}} \mathbbm{1}_{( |w_{1}|\leq t_2/\sqrt{n} )} = \frac{2}{n+2\mu^{2}} \bigg[-\frac{t_2}{\sqrt{2\pi}} e^{-\frac{1}{2}(1+\frac{2\mu^{2}}{n})t_2^{2}} + \sqrt{\frac{n}{n+2\mu^2}}\int_{0}^{t_2\sqrt{1+\frac{2\mu^{2}}{n}}} \phi(x) d x \bigg]=\frac{1+o(1)}{n}$, where $\phi(\cdot)$ is the pdf of $\mathcal{N}(0,1)$.
       \item $\E e^{-\mu^2\|w_{-1}\|_2^2}=\Big(1+\frac{2\mu^2}{n}\Big)^{-\frac{n-1}{2}}=(1+o(1))$.
    \end{itemize}
Moreover, $t_1=t_2=\mu^{-\frac{1}{4}}$ implies that $M=(4+o(1))m^2$. Plugging these results into \eqref{simplify:three} shows $\E_{\mu e_{2}} \mathcal{P}_1^{2}\geq \frac{\mu^2}{m^2}(1+o(1))$. 

\end{proof}

%% file: proof_thm5.tex
\section{Proof of Theorem \ref{thm::second-order-med-snr-minimax}}\label{proof:thm:regime2}

\subsection{Lower bound}

 Due to the scale-invariance property as shown in Section \ref{sec:scaling}, without loss of generality, we can assume $\sigma=1$ and it is sufficient to prove
\begin{align}
\label{reduced:form:thm4:lower}
R(\Theta(k,\mu),1)\geq k\mu^2\Big(1-\frac{1+o(1)}{2}\cdot \frac{k}{p}\cdot e^{\mu^2}\Big).
\end{align}
The roadmap of proof is similar to that of Theorem \ref{thm::second-order-low-snr-minimax} in Section \ref{sec::regime1-lower-bound}. Let $\pi_{\pm IB}$ be the symmetric independent block prior described in Section \ref{sec::regime1-lower-bound}. Given that $R(\Theta(k,\mu),1)\geq B(\pi_{\pm IB}(\mu;p,k))$ from \eqref{baye:risk:lower:bound}, the lower bound in Theorem \ref{thm::second-order-med-snr-minimax} directly follows from the following proposition.

\begin{proposition}
    Assume model \eqref{model::gaussian-model} with $\sigma=1$. Suppose $n\rightarrow \infty$, $p/k \rightarrow \infty$ and $\log (p/k) /n \rightarrow 0$. If $\mu\rightarrow \infty, \mu = o\big(\sqrt{\log (p/k)}\big)$ and $\mu^{4}/n\rightarrow 0$, then the Bayes risk satisfies
    \begin{equation*}
        B(\pi_{\pm IB}(\mu;p,k)) \geq k \mu^{2} \Big(1- \frac{k}{2p}\cdot e^{\mu^{2}} \big(1+o(1)\big)\Big).
    \end{equation*}
\end{proposition}

\begin{proof}
Like in the proof of Proposition \ref{thm3:lower:prop4}, let $m=p/k$ and define the symmetric spike prior $\pi_{\pm S}(\mu;m)$ for $\beta \in \mathbb{R}^m$: select an index $I\in [m]$ uniformly at random and then set $\beta =\pm \mu e_I$ with equal probability. Based on the same argument of \eqref{eq::divide-into-k-blocks} and \eqref{eq:lb:complete:classicalminimax} in the proof of Proposition \ref{prop::first-order-block-prior}, the proof is completed by calculating the Bayes risk for $\pi_{\pm S}(\mu;m)$, as shown in the next lemma. 
\end{proof}

\begin{lemma}
    Consider model \eqref{model::gaussian-model} with $\sigma=1$ and $\beta\in \mathbb{R}^m$. Suppose $n\rightarrow \infty, m \rightarrow \infty$ and $\log (m) /n \rightarrow 0$. If $\mu\rightarrow \infty, \mu = o\big(\sqrt{\log (m)}\big)$ and $\mu^{4}/n\rightarrow 0$, then the Bayes risk satisfies
    \begin{equation*}
        B(\pi_{\pm S}(\mu;m)) \geq \mu^{2} - \frac{\mu^{2}e^{\mu^{2}}}{2m}\Big(1+o(1)\Big).
    \end{equation*}
\end{lemma}
\begin{proof}

With the same calculations, the inequality in \eqref{eq::2nd-order-lower-bound-low-SNR-spike-ineq} continues to hold:
\begin{align}
    B(\pi_{\pm S}(\mu; m)) \geq \mu^{2} \Big( 1 - 2 \E_{\mu e_{1}} \mathcal{P}_{1} + (m-1)\E_{\mu e_{2}} \mathcal{P}_{1}^{2} \Big). \label{eq::2nd-order-lower-bound-med-SNR-spike-ineq}
\end{align}
Here, $\E_{\mu e_{i}}$ denotes the expectation taken under the model $y=X\beta+z$ with $\beta=\mu e_i$, and 
\begin{align}
\label{p1:recall:def}
\mathcal{P}_1=\frac{\exp(\mu X_{1}^{T}y - \mu^{2}\|X_{1}\|_2^{2}/2) - \exp (-\mu X_{1}^{T}y - \mu^{2} \|X_{1}\|_2^{2}/2)}{\sum_{i=1}^{m}\Big(\exp(\mu X_{i}^{T}y - \mu^{2}\| X_{i}\|_2^{2}/2) + \exp (-\mu X_{i}^{T}y - \mu^{2} \|X_{i}\|_2^{2}/2)\Big)},
\end{align}
where $X_i$ is the $i$th column of $X$. Based on \eqref{eq::2nd-order-lower-bound-med-SNR-spike-ineq}, combining the upper bound for $ \E_{\mu e_{1}} \mathcal{P}_{1}$ in Lemma \ref{lem::regime2-E-p_m} and the lower bound for $\E_{\mu e_{2}} \mathcal{P}_{1}^{2}$ in Lemma \ref{lem::regime2-E-p_m^2} finishes the proof.

\end{proof}

In the rest of the proof, we state and prove Lemmas \ref{lem::regime2-E-p_m} and \ref{lem::regime2-E-p_m^2}. 

\begin{lemma}\label{lem::regime2-E-p_m}
Consider model \eqref{model::gaussian-model} with $\sigma=1$ and $\beta\in \mathbb{R}^m$. Suppose $n\rightarrow \infty, m \rightarrow \infty$ and $\log(m)/n\rightarrow 0$. If $\mu\rightarrow \infty, \mu = o\big(\sqrt{\log (m)}\big)$ and $\mu^{4}/n\rightarrow 0$, we have
    \begin{equation*}
        \E_{\mu e_{1}} \mathcal{P}_{1} \leq \frac{e^{\mu^{2}}}{2m} \big( 1+o(1) \big).
    \end{equation*}
\end{lemma}
\begin{proof}
Recalling $\mathcal{P}_1$ in \eqref{p1:recall:def}, we first have
    \begin{align}
    \label{simple:start:1}
        \E_{\mu e_{1}} \mathcal{P}_{1} \leq \E \frac{\exp(\mu X_{1}^{T}y - \mu^{2}\|X_{1}\|_2^{2}/2) }{\sum_{i=1}^{m}\Big(\exp(\mu X_{i}^{T}y - \mu^{2}\| X_{i}\|_2^{2}/2) + \exp (-\mu X_{i}^{T}y - \mu^{2} \|X_{i}\|_2^{2}/2)\Big)},
    \end{align}
    where $y=\mu X_{1}+z$. Define 
    \begin{align}
    \label{v:orthg:def}
    v:=\frac{\frac{n}{\mu}x_{1}-z}{\sqrt{n(1+n/\mu^{2})}}.
    \end{align}
    It is straightforward to confirm that $v \sim \calN(0,\frac{1}{n}I_n)$ and $(y,v,\{X_{j}\}_{j=2}^{m})$ are mutual independent. Given that $X_1=\frac{y+\sqrt{n(1+n/\mu^{2})}v}{\mu+n/\mu}$, we can write 
    \begin{eqnarray*}
        & \mu X_{1}^{T}y - \mu^{2}\|X_{1}\|_2^{2}/2 = \frac{\mu^{2}(2+\mu^{2}/n)}{2n(1+\mu^{2}/n)^2}\|y\|_2^{2} + \frac{\mu}{(1+\mu^{2}/n)^{3/2}}v^{T}y - \frac{\mu^{2}}{2(1+\mu^{2}/n)}\|v\|_2^{2}, \\
        & -\mu X_{1}^{T}y - \mu^{2}\|X_{1}\|_2^{2}/2 = - \frac{\mu^{2}(2+3\mu^{2}/n)}{2n(1+\mu^{2}/n)^2}\|y\|_2^{2} -\frac{\mu(1+2\mu^{2}/n)}{(1+\mu^{2}/n)^{3/2}}v^{T}y - \frac{\mu^{2}}{2(1+\mu^{2}/n)}\|v\|_2^{2}.
    \end{eqnarray*}
    Plugging the above into \eqref{simple:start:1} yields
    \begin{align}
        \E_{\mu e_{1}}\mathcal{P}_{1} \leq \E U, \label{eq::first-ineq-bound-E-p_m}
    \end{align}
    where
    \begin{align}
        U :=& \exp\Big(\frac{\mu^{2}(2+\mu^{2}/n)}{2n(1+\mu^{2}/n)^2}\|y\|_2^{2} + \frac{\mu}{(1+\mu^{2}/n)^{3/2}}v^{T}y - \frac{\mu^{2}}{2(1+\mu^{2}/n)}\|v\|_2^{2}\Big) \nonumber\\
        & \cdot \bigg[\sum_{j=2}^{m} \bigg( \exp\Big( \mu X_{j}^{T}y - \frac{\mu^{2}}{2}\|X_{j}\|_2^{2} \Big) + \exp \Big(-\mu X_{j}^{T}y - \frac{\mu^{2}}{2}\|X_{j}\|_2^{2} \Big) \bigg) \nonumber\\
        & + \exp \Big( \frac{\mu^{2}(2+\mu^{2}/n)}{2n(1+\mu^{2}/n)^2}\|y\|_2^{2} + \frac{\mu}{(1+\mu^{2}/n)^{3/2}}v^{T}y - \frac{\mu^{2}}{2(1+\mu^{2}/n)}\|v\|_2^{2} \Big) \nonumber\\
        & + \exp \Big( - \frac{\mu^{2}(2+3\mu^{2}/n)}{2n(1+\mu^{2}/n)^2}\|y\|_2^{2} -\frac{\mu(1+2\mu^{2}/n)}{(1+\mu^{2}/n)^{3/2}}v^{T}y - \frac{\mu^{2}}{2(1+\mu^{2}/n)}\|v\|_2^{2} \Big) \bigg]^{-1}. \label{eq::U-term-in-first-ineq-bound-E-p_m}
    \end{align}
Consider the following two events:
    \begin{align}
    \label{two:events:def}
   \mathcal{I}:=\bigg\{\frac{\|y\|_2^{2}}{n(1+\mu^{2}/n)} \geq \frac{2(1+\mu^{2}/n)\log 2}{\mu(2+\mu^{2}/n)}\bigg\}, ~\mathcal{II}:=\bigg\{\frac{v^{T}y}{\|y\|_2^{2}}n \sqrt{1+\frac{\mu^{2}}{n}} \geq -\frac{\mu(2+\mu^{2}/n)}{4} \bigg\}.
 \end{align}
    From \eqref{eq::first-ineq-bound-E-p_m}, we obtain
    \begin{eqnarray*}
        \E_{\mu e_{1}}\mathcal{P}_{1} \leq \E [U \mathbbm{1}_{ \mathcal{I}\cap \mathcal{II} }]  + \E [U \mathbbm{1}_{\mathcal{I}^{c} \cup \mathcal{II}^{c} }].
    \end{eqnarray*}
    Lemma \ref{lem::truncated-E-U} shows that
    \begin{equation*}
        \E [U \mathbbm{1}_{ \mathcal{I}\cap \mathcal{II} }] \leq \frac{e^{\mu^{2}}}{2m} \big(1+o(1)\big),
    \end{equation*}
    and Lemma \ref{lem::weak-convergence-E-p_m} proves
    \begin{equation*}
       \E [U \mathbbm{1}_{\mathcal{I}^{c} \cup \mathcal{II}^{c} }]=o\Big( \frac{e^{\mu^{2}}}{m} \Big).
    \end{equation*}
  The poof is thus completed.
\end{proof}

\begin{lemma}\label{lem::truncated-E-U}
Under the same conditions of Lemma \ref{lem::regime2-E-p_m}, it holds that
    \begin{equation*}
        \E [U \mathbbm{1}_{ \mathcal{I}\cap \mathcal{II} }] \leq \frac{e^{\mu^{2}}}{2m} \big(1+o(1)\big).
    \end{equation*}
Here, $U$ is introduced in \eqref{eq::U-term-in-first-ineq-bound-E-p_m} and $\mathcal{I},\mathcal{II}$ are defined in \eqref{two:events:def}.
\end{lemma}
\begin{proof}
Referring to the definition of $U$ in \eqref{eq::U-term-in-first-ineq-bound-E-p_m}, we first aim to find a lower bound for the denominator of $U$. Since $e^a+e^{-a}\geq e^{b}+e^{-b}, \forall |a|\geq |b|$, we have
\begin{align}
\label{denm:U:upper}
&~\sum_{j=2}^{m} \bigg( \exp\Big( \mu X_{j}^{T}y - \frac{\mu^{2}}{2}\|X_{j}\|_2^{2} \Big) + \exp \Big(-\mu X_{j}^{T}y - \frac{\mu^{2}}{2}\|X_{j}\|_2^{2} \Big) \bigg) \nonumber \\
\geq &~\sum_{j=2}^{m} \bigg( \exp\Big( \frac{\mu X_{j}^{T}y}{(1+\mu^2/n)^{3/2}} - \frac{\mu^{2}}{2}\|X_{j}\|_2^{2} \Big) + \exp \Big(-\frac{\mu X_{j}^{T}y}{(1+\mu^2/n)^{3/2}} - \frac{\mu^{2}}{2}\|X_{j}\|_2^{2} \Big) \bigg)  \nonumber \\
\geq &~\exp\Big(-\frac{\mu^{4}\max_{2\leq j\leq m} (\|v\|_2^{2} \vee \|X_{j}\|_2^{2})}{2n(1+\mu^{2}/n)}\Big) \nonumber \\
&~\cdot \sum_{j=2}^{m} \bigg( \exp\Big( \frac{\mu X_{j}^{T}y}{(1+\mu^2/n)^{3/2}} - \frac{\mu^{2}\|X_{j}\|_2^{2}}{2(1+\mu^2/n)} \Big) + \exp \Big(-\frac{\mu X_{j}^{T}y}{(1+\mu^2/n)^{3/2}} - \frac{\mu^{2}\|X_{j}\|_2^{2}}{2(1+\mu^2/n)} \Big) \bigg)
\end{align}
This result together with Lemma \ref{lem::sufficient-conditions} implies 
    \begin{align}
    \label{inter:events:bound1}
        \E [U \mathbbm{1}_{ \mathcal{I}\cap \mathcal{II} }] \leq~& \E \exp\Big(\frac{\mu^{2}(2+\mu^{2}/n)\|y\|_2^{2}}{2n(1+\mu^{2}/n)^2} + \frac{\mu^{4}\max_{2\leq j\leq m} (\|v\|_2^{2} \vee \|X_{j}\|_2^{2})}{2n(1+\mu^{2}/n)} \Big)  \\
        & \cdot \exp\Big(\frac{\mu v^{T}y}{(1+\mu^{2}/n)^{3/2}} - \frac{\mu^{2}\|v\|_2^{2}}{2(1+\mu^{2}/n)}\Big) \nonumber\\
        & \cdot\bigg[ \sum_{j=2}^{m} \bigg( \exp\Big( \frac{\mu X_{j}^{T}y}{(1+\mu^{2}/n)^{3/2}} - \frac{\mu^{2}\|X_{j}\|_2^{2}}{2(1+\mu^2/n)} \Big) + \exp\Big( -\frac{\mu X_{j}^{T}y}{(1+\mu^{2}/n)^{3/2}} - \frac{\mu^{2}\|X_{j}\|_2^{2}}{2(1+\mu^2/n)} \Big)\bigg) 
 \nonumber\\
        & + \exp \Big( \frac{\mu v^{T}y}{(1+\mu^{2}/n)^{3/2}} - \frac{\mu^{2} \|v\|_2^{2}}{2(1+\mu^{2}/n)} \Big) + \exp \Big( -\frac{\mu v^{T}y}{(1+\mu^{2}/n)^{3/2}} - \frac{\mu^{2}\|v\|_2^{2}}{2(1+\mu^{2}/n)} \Big) \bigg]^{-1}. \nonumber
    \end{align}
For any given $y\in \mathbb{R}^n$, define
    \begin{align}
    \label{givey:def:fg}
    g^y(t_1,\ldots, t_m)&:=\exp\Big(\frac{\mu^{2}(2+\mu^{2}/n)\|y\|_2^{2}}{2n(1+\mu^{2}/n)^2} + \frac{\mu^{4}\max_{1\leq j\leq m} \|t_j\|_2^{2}}{2n(1+\mu^{2}/n)} \Big), \quad \forall t_1,\ldots, t_m \in \mathbb{R}^n, \nonumber \\
        f^y_{\pm}(s)&:= \exp\Big( \pm \frac{\mu s^{T}y}{(1+\mu^{2}/n)^{3/2}} - \frac{\mu^{2}\|s\|_2^{2}}{2(1+\mu^{2}/n)}\Big), \quad \forall s \in \mathbb{R}^n. 
    \end{align}
Since $v, X_2,\ldots, X_m$ are independently and identically distributed and they are independent of $y$, we can calculate the upper bound in \eqref{inter:events:bound1} by conditioning on $y$,
\begin{align}
\label{inter:events:bound2}
 \E [U \mathbbm{1}_{ \mathcal{I}\cap \mathcal{II} }] &\leq \E_y\bigg[\E\Big( \frac{g^y(v,X_2,\ldots, X_m)\cdot f_+^y(v)}{f^y_+(v)+f^y_-(v)+\sum_{j=2}^m(f^y_+(X_j)+f^y_-(X_j))}\Big | y\Big)\bigg] \nonumber \\
 &\overset{(a)}{=}\frac{1}{2}\E_y\bigg[\E\Big( \frac{g^y(v,X_2,\ldots, X_m)\cdot (f_+^y(v)+f_-^y(v))}{f^y_+(v)+f^y_-(v)+\sum_{j=2}^m(f^y_+(X_j)+f^y_-(X_j))}\Big | y\Big)\bigg] \nonumber \\
&\overset{(b)}{=}\frac{1}{2m}\E_y\bigg[\E\Big(g^y(v,X_2,\ldots, X_m) \Big | y\Big)\bigg]  \nonumber \\
&\overset{(c)}{=} \frac{1}{2m} \cdot \E \exp\Big(\frac{\mu^{2}(2+\mu^{2}/n)\|y\|_2^{2}}{2n(1+\mu^{2}/n)^2} \Big)\cdot \E \exp \Big( \frac{\mu^{4}\max_{2\leq j\leq m} (\|v\|_2^{2} \vee \|X_{j}\|_2^{2})}{2n(1+\mu^{2}/n)} \Big)
\end{align}
Here, $(a)$ holds since flipping the sign of $v$ does not change the conditional expectation; $(b)$ is based on the equations
\begin{align*}
&~\E\Big( \frac{g^y(v,X_2,\ldots, X_m)\cdot (f_+^y(v)+f_-^y(v))}{f^y_+(v)+f^y_-(v)+\sum_{j=2}^m(f^y_+(X_j)+f^y_-(X_j))}\Big | y\Big) \\
=&~\E\Big( \frac{g^y(v,X_2,\ldots, X_m)\cdot (f_+^y(X_j)+f_-^y(X_j))}{f^y_+(v)+f^y_-(v)+\sum_{j=2}^m(f^y_+(X_j)+f^y_-(X_j))}\Big | y\Big), ~~j=2,\ldots, m,
\end{align*}
due to the exchangeability of $v,X_2,\ldots, X_m$; and $(c)$ is by the independence between $y$ and $(v,X_2,\ldots, X_m)$. It remains to compute the two expectations in \eqref{inter:events:bound2}. As $y\sim \mathcal{N}(0, (1+\mu^2/n)I_n)$, we use the moment-generating function of chi-squared distribution to obtain
\begin{align}
\label{exp:one:rel}
\E \exp\Big(\frac{\mu^{2}(2+\mu^{2}/n)\|y\|_2^{2}}{2n(1+\mu^{2}/n)^2} \Big)=\Big(1-\frac{\mu^2(2+\mu^2/n)}{n(1+\mu^2/n)}\Big)^{-\frac{n}{2}}=e^{\mu^2}(1+o(1)),
\end{align}
where in the last equation we have used the condition $\mu^4/n\rightarrow 0$ and the expansion $\log(1+x)=x+O(x^2)$ as $x\rightarrow 0$. Moreover, we apply Lemma \ref{lem::mgf-max-chi-square} to have that for some constant $c>0$,
\begin{align}
\label{exp:two:rel}
\E \exp \Big( \frac{\mu^{4}\max_{2\leq j\leq m} (\|v\|_2^{2} \vee \|X_{j}\|_2^{2})}{2n(1+\mu^{2}/n)} \Big) \leq \exp \Big( \frac{c\mu^{4}}{n(1+\mu^{2}/n)} \Big) \cdot \big(1+o(1)\big)=1+o(1),
\end{align}
where the last equality is due to $\mu^4/n\rightarrow 0$. Combining \eqref{inter:events:bound2}, \eqref{exp:one:rel} and \eqref{exp:two:rel} finishes the proof.

\end{proof}

\begin{lemma}\label{lem::sufficient-conditions}
    Assume $\mu \geq 1$ and $\mu^{2}/n<1$. For events $\mathcal{I}$ and $\mathcal{II}$ defined in \eqref{two:events:def}, the intersection $\mathcal{I}\cap \mathcal{II}$ implies that
    \begin{align}
    &\exp \Big( \frac{\mu^{2}(2+\mu^{2}/n)}{2n(1+\mu^{2}/n)^2}\|y\|_2^{2} + \frac{\mu}{(1+\mu^{2}/n)^{3/2}}v^{T}y - \frac{\mu^{2}}{2(1+\mu^{2}/n)}\|v\|_2^{2} \Big) \label{eq::truncation-event} \\
        & + \exp \Big( - \frac{\mu^{2}(2+3\mu^{2}/n)}{2n(1+\mu^{2}/n)^2}\|y\|_2^{2} -\frac{\mu(1+2\mu^{2}/n)}{(1+\mu^{2}/n)^{3/2}}v^{T}y - \frac{\mu^{2}}{2(1+\mu^{2}/n)}\|v\|_2^{2} \Big) \nonumber\\
        &\geq \exp \Big( \frac{\mu}{(1+\mu^{2}/n)^{3/2}}v^{T}y - \frac{\mu^{2}}{2(1+\mu^{2}/n)}\|v\|_2^{2} \Big) + \exp \Big( -\frac{\mu}{(1+\mu^{2}/n)^{3/2}}v^{T}y - \frac{\mu^{2}}{2(1+\mu^{2}/n)}\|v\|_2^{2} \Big). \nonumber
    \end{align}
\end{lemma}
\begin{proof}
Multiplying both sides of \eqref{eq::truncation-event} by $\exp\Big(\frac{\mu^2(2+\mu^2/n)\|y\|_2^2}{2n(1+\mu^2/n)^2}+ \frac{\mu v^{T}y}{(1+\mu^{2}/n)^{3/2}} + \frac{\mu^{2}\|v\|^{2}}{2(1+\mu^{2}/n)}\Big)$, we can obtain
        \begin{align}
       & \Big[ \exp \Big( \frac{\mu^{2}(2+\mu^{2}/n)\|y\|_2^2}{2n(1+\mu^{2}/n)^2} +\frac{2\mu v^{T}y}{(1+\mu^{2}/n)^{3/2}} \Big) - 1 \Big] \cdot \Big[ \exp\Big( \frac{\mu^{2}(2+\mu^{2}/n)\|y\|_2^2}{2n(1+\mu^{2}/n)^2}\Big) - 1 \Big] \nonumber\\
        & \geq 1 - \exp \Big( - \frac{\mu^{4}/n\|y\|_2^2}{n(1+\mu^{2}/n)^2} - \frac{2\mu^{3}/n}{(1+\mu^{2}/n)^{3/2}}v^{T}y \Big).\label{eq::sufficient-condition-equivalence}
    \end{align}
    To prove \eqref{eq::truncation-event}, it is equivalent to prove \eqref{eq::sufficient-condition-equivalence}. 
    
   When the event $\mathcal{II}$ holds, we have
    \begin{equation}
    \label{event:2:imply:one}
    \frac{v^{T}y}{\|y\|_2^{2}} \geq- \frac{\mu(2+\mu^{2}/n)}{4n\sqrt{1+\mu^{2}/n}},
    \end{equation}
    or equivalently
    \begin{align}
    \label{event:2:imply:two}
      \exp\Big( \frac{\mu^{2}(2+\mu^{2}/n)\|y\|_2^2}{2n(1+\mu^{2}/n)^2} +\frac{2\mu v^{T}y}{(1+\mu^{2}/n)^{3/2}}\Big) \geq 1. 
    \end{align}
    Moreover, since $\mu\geq 1$, the event $\mathcal{I}$ implies  
   \begin{equation}
        \exp\Big( \frac{\mu^{2}(2+\mu^{2}/n)\|y\|_2^2}{2n(1+\mu^{2}/n)^2}\Big) \geq 2. \label{eq::derive-condition-i}
    \end{equation}
Based on \eqref{event:2:imply:two} and \eqref{eq::derive-condition-i}, to prove\eqref{eq::sufficient-condition-equivalence} it is sufficient to show
\begin{align}
         & ~ \exp \Big( \frac{\mu^{2}(2+\mu^{2}/n)\|y\|_2^2}{2n(1+\mu^{2}/n)^2} +\frac{2\mu v^{T}y}{(1+\mu^{2}/n)^{3/2}} \Big) - 1 \nonumber\\
        \geq &~ 1 - \exp \Big( - \frac{\mu^{4}/n\|y\|_2^2}{n(1+\mu^{2}/n)^2} - \frac{2\mu^{3}/n}{(1+\mu^{2}/n)^{3/2}}v^{T}y \Big).\label{eq::sufficient-condition-step1}
    \end{align}
  
In order to prove \eqref{eq::sufficient-condition-step1}, it is direct to verify that \eqref{event:2:imply:one} together with the condition $\mu^2/n<1$ gives
    \begin{equation*}
        \frac{v^{T}y}{\|y\|^{2}} \geq -  \frac{\mu(2-\mu^{2}/n)}{4n\sqrt{1+\mu^{2}/n}(1-\mu^{2}/n)},
    \end{equation*}
    or equivalently
\begin{equation}
        \frac{\mu^{2}(2-\mu^{2}/n)\|y\|_2^2}{2n(1+\mu^{2}/n)^2} +\frac{2\mu(1-\mu^{2}/n)v^{T}y}{(1+\mu^{2}/n)^{3/2}} \geq 0. \label{eq::derive-condition-ii}
    \end{equation}

Hence, \eqref{eq::sufficient-condition-step1} can be obtained by the following inequalities:
    \begin{eqnarray*}
        &&\exp \Big( \frac{\mu^{2}(2+\mu^{2}/n)\|y\|_2^2}{2n(1+\mu^{2}/n)^2} +\frac{2\mu v^{T}y}{(1+\mu^{2}/n)^{3/2}} \Big) - 1 \\
        &\overset{(a)}{\geq}& \exp \Big( \frac{\mu^{4}/n\|y\|_2^2}{n(1+\mu^{2}/n)^2}+ \frac{2\mu^{3}/n}{(1+\mu^{2}/n)^{3/2}}v^{T}y \Big) - 1 \\
        &\overset{(b)}{\geq}& 1 - \exp \Big( - \frac{\mu^{4}/n\|y\|_2^2}{n(1+\mu^{2}/n)^2} - \frac{2\mu^{3}/n}{(1+\mu^{2}/n)^{3/2}}v^{T}y \Big),
    \end{eqnarray*}
where $(a)$ is due to \eqref{eq::derive-condition-ii}, and $(b)$ holds since $e^t+e^{-t}\geq 2, \forall t\in \mathbb{R}$.

\end{proof}

\begin{lemma}\label{lem::mgf-max-chi-square}
    Suppose $X_1, X_{2}, \ldots, X_{m} \simiid \calN(0,\frac{1}{n}I_{n})$. If $\mu^{2}/n \rightarrow 0$ and $(\log m)/n\rightarrow 0$, then there exists some constant $c>0$ such that
    \begin{equation*}
        \E \Big[\exp \Big( \frac{\mu^{4}}{2n(1+\mu^{2}/n)}\max_{1\leq j\leq m}  \|X_{j}\|_2^{2} \Big) \Big] \leq \exp \Big( \frac{c\mu^{4}}{n(1+\mu^{2}/n)} \Big) \cdot \big(1+o(1)\big).
    \end{equation*}
\end{lemma}
\begin{proof}
    Consider a constant $c>1/2$. We will describe our choice of $c$ later. Using Lemma \ref{lem:meanFromCDF} we obtain:
    \begin{align}
    \label{decomp:one:two}
    & \E \Big[\exp \Big( \frac{\mu^{4}}{2n(1+\mu^{2}/n)}\max_{1\leq j\leq m}  \|X_{j}\|_2^{2} \Big) \Big] \nonumber \\
     =&\int_0^{e^{\frac{c\mu^4}{n(1+\mu^2/n)}}}\mathbb{P}\Bigg(\exp \Big( \frac{\mu^{4}}{2n^2(1+\mu^{2}/n)}\max_{1\leq j\leq m}  \|\sqrt{n}X_{j}\|_2^{2} \Big)>t\Bigg)dt+ \nonumber \\
     &\int_{e^{\frac{c\mu^4}{n(1+\mu^2/n)}}}^{\infty}\mathbb{P}\Bigg(\exp \Big( \frac{\mu^{4}}{2n^2(1+\mu^{2}/n)}\max_{1\leq j\leq m}  \|\sqrt{n}X_{j}\|_2^{2} \Big)>t\Bigg)dt:=I+II.
        \end{align}
The term I admits a simple upper bound:
\begin{align}
\label{term:1:upper}
I\leq e^{\frac{c\mu^4}{n(1+\mu^2/n)}}.
\end{align}
We now focus on bounding term II. We apply the change of variable $x=\frac{2n(1+\mu^2/n)}{\mu^4}\log t$ to obtain
\begin{align}
\label{term:2:start}
II&=\frac{\mu^4}{2n(1+\mu^2/n)}\cdot \int_{2c}^{\infty}\mathbb{P}\Big(\max_{1\leq j\leq m}  \|\sqrt{n}X_{j}\|_2>\sqrt{nx}\Big)\cdot e^{\frac{\mu^4}{2n(1+\mu^2/n)}x}dx \nonumber \\
&\leq \frac{\mu^4}{2n(1+\mu^2/n)}\cdot \int_{2c}^{\infty}me^{-\frac{n}{2}(\sqrt{x}-1)^2}\cdot e^{\frac{\mu^4}{2n(1+\mu^2/n)}x}dx \nonumber \\
&=\frac{\mu^4}{2n(1+\mu^2/n)}\cdot \int_{2c}^{\infty}\exp \bigg\{-n\bigg(\Big(\frac{1}{2}-\frac{\mu^4}{2n^2(1+\mu^2/n)}\Big)x-\sqrt{x}+\frac{1}{2}-\frac{\log m}{n}\bigg) \bigg\}dx,
\end{align}
where the inequality above is due to the union bound and Lemma \ref{lem::chi-concentration}. Given that $\mu^{2}/n \rightarrow 0$ and $(\log m)/n\rightarrow 0$, it is straightforward to verify that the following holds,
\[
\Big(\frac{1}{2}-\frac{\mu^4}{2n^2(1+\mu^2/n)}\Big)x-\sqrt{x}+\frac{1}{2}-\frac{\log m}{n} \geq \frac{1}{4}x, \quad \forall x \geq 2c,
\]
as long as $n$ is sufficiently large, and the constant $c$ is chosen large enough (e.g. $c=32$). Therefore, we can continue from \eqref{term:2:start} to have
\begin{align}
\label{term:two:final}
II \leq \frac{\mu^4}{2n(1+\mu^2/n)}\cdot \int_{2c}^{\infty} \exp\Big\{-\frac{n}{4}x\Big\} dx=\frac{2\mu^4}{n^2(1+\mu^2/n)}e^{-\frac{c}{2}n}.
\end{align}
Putting together \eqref{decomp:one:two}, \eqref{term:1:upper} and \eqref{term:two:final} completes the proof.
\end{proof}

\begin{lemma}\label{lem::weak-convergence-E-p_m}
   Under the same conditions of Lemma \ref{lem::regime2-E-p_m}, it holds that
    \begin{equation*}
        \E [U \mathbbm{1}_{\mathcal{I}^{c} \cup \mathcal{II}^{c} }]=o\Big( \frac{e^{\mu^{2}}}{m} \Big).
    \end{equation*}
Here, $U$ is introduced in \eqref{eq::U-term-in-first-ineq-bound-E-p_m} and $\mathcal{I},\mathcal{II}$ are defined in \eqref{two:events:def}.
\end{lemma}

\begin{proof}
As in the proof of Lemma \ref{lem::truncated-E-U}, we first find a lower bound for the denominator of $U$. We continue from \eqref{denm:U:upper} to obtain
\begin{align*}
&~\sum_{j=2}^{m} \bigg( \exp\Big( \mu X_{j}^{T}y - \frac{\mu^{2}}{2}\|X_{j}\|_2^{2} \Big) + \exp \Big(-\mu X_{j}^{T}y - \frac{\mu^{2}}{2}\|X_{j}\|_2^{2} \Big) \bigg)\\
\geq &~\exp\Big(-\frac{\mu^{4}\max_{2\leq j\leq m} (\|v\|_2^{2} \vee \|X_{j}\|_2^{2})}{2n(1+\mu^{2}/n)}\Big) \cdot \sum_{j=2}^{m}  \exp\Big( \frac{\mu X_{j}^{T}y}{(1+\mu^2/n)^{3/2}} - \frac{\mu^{2}\|X_{j}\|_2^{2}}{2(1+\mu^2/n)} \Big).
\end{align*}
Also, the other part of the denominator has a simple lower bound:
\begin{align*}
& \exp \Big( \frac{\mu^{2}(2+\mu^{2}/n)}{2n(1+\mu^{2}/n)^2}\|y\|_2^{2} + \frac{\mu}{(1+\mu^{2}/n)^{3/2}}v^{T}y - \frac{\mu^{2}}{2(1+\mu^{2}/n)}\|v\|_2^{2} \Big) \nonumber\\
        & + \exp \Big( - \frac{\mu^{2}(2+3\mu^{2}/n)}{2n(1+\mu^{2}/n)^2}\|y\|_2^{2} -\frac{\mu(1+2\mu^{2}/n)}{(1+\mu^{2}/n)^{3/2}}v^{T}y - \frac{\mu^{2}}{2(1+\mu^{2}/n)}\|v\|_2^{2} \Big)  \\
\geq &~\exp\Big(-\frac{\mu^{4}\max_{2\leq j\leq m} (\|v\|_2^{2} \vee \|X_{j}\|_2^{2})}{2n(1+\mu^{2}/n)}\Big) \cdot \exp \Big(\frac{\mu v^{T}y}{(1+\mu^{2}/n)^{3/2}} - \frac{\mu^{2}\|v\|_2^{2}}{2(1+\mu^{2}/n)} \Big).
\end{align*}
Plugging the above two results into $U$ in \eqref{eq::U-term-in-first-ineq-bound-E-p_m} yields 
\begin{align}
\label{U:comple:upper}
&\E [U \mathbbm{1}_{\mathcal{I}^{c} \cup \mathcal{II}^{c} }] \nonumber \\
\leq &\E \exp\Big(\frac{\mu^{2}(2+\mu^{2}/n)\|y\|_2^{2}}{2n(1+\mu^{2}/n)^2} + \frac{\mu^{4}\max_{2\leq j\leq m} (\|v\|_2^{2} \vee \|X_{j}\|_2^{2})}{2n(1+\mu^{2}/n)} \Big)  \\
        & \cdot \exp\Big(\frac{\mu v^{T}y}{(1+\mu^{2}/n)^{3/2}} - \frac{\mu^{2}\|v\|_2^{2}}{2(1+\mu^{2}/n)}\Big)\mathbbm{1}_{\mathcal{I}^{c} \cup \mathcal{II}^{c} } \nonumber\\
        & \cdot\bigg[ \sum_{j=2}^{m}  \exp\Big( \frac{\mu X_{j}^{T}y}{(1+\mu^{2}/n)^{3/2}} - \frac{\mu^{2}\|X_{j}\|_2^{2}}{2(1+\mu^2/n)} \Big)  + \exp \Big( \frac{\mu v^{T}y}{(1+\mu^{2}/n)^{3/2}} - \frac{\mu^{2} \|v\|_2^{2}}{2(1+\mu^{2}/n)} \Big)  \bigg]^{-1}. \nonumber
\end{align}
Adopt the notation from \eqref{givey:def:fg}, and further define $\forall s\in \mathbb{R}^n$,
\begin{align}
\label{event:H:def}
\mathcal{H}^y(s):=\Bigg\{\frac{\|y\|_2^{2}}{n(1+\mu^{2}/n)} < \frac{2(1+\mu^{2}/n)\log 2}{\mu(2+\mu^{2}/n)}\bigg\} {\rm ~~or~~}\frac{s^{T}y}{\|y\|_2^{2}}n \sqrt{1+\frac{\mu^{2}}{n}} < -\frac{\mu(2+\mu^{2}/n)}{4}\Bigg\}.
\end{align}
Conditional on $y$, the upper bound in \eqref{U:comple:upper} can be rewritten as
\begin{align}
\label{U:comple:upper:more}
&~\E [U \mathbbm{1}_{\mathcal{I}^{c} \cup \mathcal{II}^{c} }] \nonumber \\
\leq &~\E_y\Bigg[\E\Big(\frac{g^y(v,X_2,\ldots, X_m)\cdot f_+^y(v)\mathbbm{1}_{\mathcal{H}^y(v)}}{f_+^y(v)+\sum_{j=2}^mf_+^y(X_j)}\Big|y\Big)\Bigg] \nonumber \\
\overset{(a)}{=}&~\frac{1}{m}\cdot \E_y\Bigg[\E\Big(\frac{g^y(v,X_2,\ldots, X_m)\cdot \big(f_+^y(v)\mathbbm{1}_{\mathcal{H}^y(v)}+\sum_{j=2}^mf_+^y(X_j)\mathbbm{1}_{\mathcal{H}^y(X_j)}\big)}{f_+^y(v)+\sum_{j=2}^mf_+^y(X_j)}\Big|y\Big)\Bigg] \nonumber \\
\overset{(b)}{\leq} &~\frac{1}{m} \cdot \sqrt{\E \exp\Big(\frac{\mu^{2}(2+\mu^{2}/n)\|y\|_2^{2}}{n(1+\mu^{2}/n)^2} \Big)\cdot \E \exp \Big( \frac{\mu^{4}\max_{2\leq j\leq m} (\|v\|_2^{2} \vee \|X_{j}\|_2^{2})}{n(1+\mu^{2}/n)} \Big)} \nonumber \\
&~ \cdot \sqrt{\E \Big(\frac{f_+^y(v)\mathbbm{1}_{\mathcal{H}^y(v)}+\sum_{j=2}^mf_+^y(X_j)\mathbbm{1}_{\mathcal{H}^y(X_j)}}{f_+^y(v)+\sum_{j=2}^mf_+^y(X_j)}\Big)^2},
\end{align}
where $(a)$ is due to the identities that for $j=2,\ldots, m$,
\[
\E\Big(\frac{g^y(v,X_2,\ldots, X_m)\cdot f_+^y(v)\mathbbm{1}_{\mathcal{H}^y(v)}}{f_+^y(v)+\sum_{j=2}^mf_+^y(X_j)}\Big|y\Big)=\E\Big(\frac{g^y(v,X_2,\ldots, X_m)\cdot f_+^y(X_j)\mathbbm{1}_{\mathcal{H}^y(X_j)}}{f_+^y(v)+\sum_{j=2}^mf_+^y(X_j)}\Big|y\Big),
\]
thanks to the exchangeability among $(v,X_2,\ldots, X_m)$; $(b)$ holds by Cauchy–Schwarz inequality. Using arguments similar to \eqref{exp:one:rel} and \eqref{exp:two:rel} in the proof of Lemma \ref{lem::truncated-E-U}, we have that as $\mu^4/n\rightarrow 0$,
\begin{align}
\label{similar:two:exp}
&\E \exp\Big(\frac{\mu^{2}(2+\mu^{2}/n)\|y\|_2^{2}}{n(1+\mu^{2}/n)^2} \Big)=e^{2\mu^2}(1+o(1)), \nonumber \\
&\E \exp \Big( \frac{\mu^{4}\max_{2\leq j\leq m} (\|v\|_2^{2} \vee \|X_{j}\|_2^{2})}{n(1+\mu^{2}/n)} \Big) \leq 1+o(1).
\end{align}
Based on \eqref{U:comple:upper:more} and \eqref{similar:two:exp}, the proof will be completed if we can further show
\[
\E \Big(\frac{f_+^y(v)\mathbbm{1}_{\mathcal{H}^y(v)}+\sum_{j=2}^mf_+^y(X_j)\mathbbm{1}_{\mathcal{H}^y(X_j)}}{f_+^y(v)+\sum_{j=2}^mf_+^y(X_j)}\Big)^2=o(1).
\]
This is done in the next lemma.
\end{proof}

\begin{lemma}
\label{bounded:dct:zero}
Suppose $n\rightarrow \infty, m \rightarrow \infty$ and $\log(m)/n\rightarrow 0$. If $\mu\rightarrow \infty, \mu = o\big(\sqrt{\log (m)}\big)$ and $\mu^{2}/n\rightarrow 0$, it holds that
\begin{align*}
\E\Big(\frac{\sum_{j=1}^mf_+^y(X_j)\mathbbm{1}_{\mathcal{H}^y(X_j)}}{\sum_{j=1}^mf_+^y(X_j)}\Big)^2=o(1).
\end{align*}
Here, $f_+^y(\cdot), \mathcal{H}^y(\cdot)$ are defined in \eqref{givey:def:fg} and \eqref{event:H:def}; $X_1,X_2,\ldots, X_m \overset{i.i.d.}{\sim}\mathcal{N}(0,\frac{1}{n}I_n)$ and they are independent of $y\sim \mathcal{N}(0, (1+\mu^2/n)I_n)$.
\end{lemma}

\begin{proof}
Since $0\leq \frac{\sum_{j=1}^mf_+^y(X_j)\mathbbm{1}_{\mathcal{H}^y(X_j)}}{\sum_{j=1}^mf_+^y(X_j)} \leq 1$, according to the dominated convergence theorem, it is sufficient to show $\frac{\sum_{j=1}^mf_+^y(X_j)\mathbbm{1}_{\mathcal{H}^y(X_j)}}{\sum_{j=1}^mf_+^y(X_j)} \overset{P}{\rightarrow} 0$. Towards this end, define
\begin{align*}
\mathcal{A}:=&\frac{1}{m}\Big(\frac{1+2\mu^2/n}{1+\mu^2/n}\Big)^{\frac{n}{2}}\exp\Big(\frac{-\mu^2/n\|y\|_2^2}{2(1+\mu^2/n)^2(1+2\mu^2/n)}\Big) \sum_{j=1}^{m}  \exp\Big( \frac{\mu X_{j}^{T}y}{(1+\mu^{2}/n)^{3/2}} - \frac{\mu^{2}\|X_{j}\|_2^{2}}{2(1+\mu^2/n)} \Big)    \\
\mathcal{B}:=& \frac{1}{m}\Big(\frac{1+2\mu^2/n}{1+\mu^2/n}\Big)^{\frac{n}{2}}\exp\Big(\frac{-\mu^2/n\|y\|_2^2}{2(1+\mu^2/n)^2(1+2\mu^2/n)}\Big) \sum_{j=1}^{m}  \exp\Big( \frac{\mu X_{j}^{T}y}{(1+\mu^{2}/n)^{3/2}} - \frac{\mu^{2}\|X_{j}\|_2^{2}}{2(1+\mu^2/n)} \Big)\mathbbm{1}_{\mathcal{H}^y(X_j)}   
\end{align*}
Then we can rewrite $\frac{\sum_{j=1}^mf_+^y(X_j)\mathbbm{1}_{\mathcal{H}^y(X_j)}}{\sum_{j=1}^mf_+^y(X_j)}=\frac{\mathcal{B}}{\mathcal{A}}$. In the rest of the proof, we will prove $\mathcal{A}\overset{P}{\rightarrow} 1$ and $\mathcal{B}\overset{P}{\rightarrow} 0$.

Regarding $\mathcal{A}$, using the moment-generating function of noncentral chi-squared distribution, we can obtain that with $\lambda=\frac{n\|y\|_2^2}{\mu^2(1+\mu^2/n)}$,
\begin{align*}
\E\mathcal{A}&=\E\big[\E\big(\mathcal{A}\big| y\big)\big] \\
&=\E\Bigg[\Big(\frac{1+2\mu^2/n}{1+\mu^2/n}\Big)^{\frac{n}{2}}\exp\Big(\frac{\|y\|_2^2}{2(1+\mu^2/n)(1+2\mu^2/n)}\Big) \cdot \E\Big(\exp\Big(\frac{-\mu^2 \chi^2_n(\lambda)}{2n(1+\mu^2/n)}\Big)\Big|y\Big)  \Bigg] \\
&=1.
\end{align*}
Hence, to prove $\mathcal{A}\overset{P}{\rightarrow} 1$, it is sufficient to show ${\rm Var}(\mathcal{A})\rightarrow 0$. We use the moment-generating function of noncentral chi-squared distribution again to compute the variance,
\begin{align*}
{\rm Var}(\mathcal{A})&=\E\big[{\rm Var}(\mathcal{A}|y)\big]+{\rm Var}\big[\E(\mathcal{A}|y)\big] \\
&=\E\Bigg[\frac{1}{m}\Big(\frac{1+2\mu^2/n}{1+\mu^2/n}\Big)^{n}\exp\Big(\frac{\|y\|_2^2}{(1+\mu^2/n)(1+2\mu^2/n)}\Big) \cdot {\rm Var}\Big(\exp\Big(\frac{-\mu^2 \chi^2_n(\lambda)}{2n(1+\mu^2/n)}\Big)\Big|y\Big)  \Bigg] \\
&\leq \E\Bigg[\frac{1}{m}\Big(\frac{1+2\mu^2/n}{1+\mu^2/n}\Big)^{n}\exp\Big(\frac{\|y\|_2^2}{(1+\mu^2/n)(1+2\mu^2/n)}\Big) \cdot \E\Big(\exp\Big(\frac{-\mu^2 \chi^2_n(\lambda)}{n(1+\mu^2/n)}\Big)\Big|y\Big)  \Bigg] \\
&=\frac{1}{m}\Big(\frac{1+2\mu^2/n}{1+\mu^2/n}\Big)^{n}\Big(\frac{1+3\mu^2/n}{1+\mu^2/n}\Big)^{-\frac{n}{2}}\cdot \E\exp\Big(\frac{\mu^2/n\|y\|_2^2}{(1+\mu^2/n)(1+2\mu^2/n)(1+3\mu^2/n)}\Big) \\
&=\frac{1}{m}\Big(\frac{1+2\mu^2/n}{1+\mu^2/n}\Big)^{n}\Big(\frac{1+3\mu^2/n}{1+\mu^2/n}\Big)^{-\frac{n}{2}}\Big(1-\frac{2\mu^2/n}{(1+2\mu^2/n)(1+3\mu^2/n)}\Big)^{-\frac{n}{2}}\\
&=\frac{1}{m}\exp\Big(\mu^2(1+o(1))\Big)=o(1),
\end{align*}
where in the second-to-last equality we have used $\log(1+x)=x+o(x)$ as $x\rightarrow 0$ and the condition $\mu^2/n\rightarrow 0$, and the last equality holds by the condition $m\rightarrow \infty$ and $\mu^2=o(\log (m))$.

For $\mathcal{B}$, we will show a stronger result $\E \mathcal{B} \rightarrow 0$. Recalling the definition of $\mathcal{H}^y(\cdot)$ in \eqref{event:H:def}, we have
\begin{align*}
\E \mathcal{B}&=\E\Big(\frac{1+2\mu^2/n}{1+\mu^2/n}\Big)^{\frac{n}{2}}\exp\Big(\frac{-\mu^2/n\|y\|_2^2}{2(1+\mu^2/n)^2(1+2\mu^2/n)}\Big)  \exp\Big( \frac{\mu X_{1}^{T}y}{(1+\mu^{2}/n)^{3/2}} - \frac{\mu^{2}\|X_{1}\|_2^{2}}{2(1+\mu^2/n)}\Big)\mathbbm{1}_{\mathcal{H}^y(X_1)} \\
&\leq \E\Big(\frac{1+2\mu^2/n}{1+\mu^2/n}\Big)^{\frac{n}{2}}\exp\Big(\frac{-\mu^2/n\|y\|_2^2}{2(1+\mu^2/n)^2(1+2\mu^2/n)}\Big)  \exp\Big( \frac{\mu X_{1}^{T}y}{(1+\mu^{2}/n)^{3/2}} - \frac{\mu^{2}\|X_{1}\|_2^{2}}{2(1+\mu^2/n)}\Big)\mathbbm{1}_{\mathcal{H}_1}+  \\
&\quad \E\Big(\frac{1+2\mu^2/n}{1+\mu^2/n}\Big)^{\frac{n}{2}}\exp\Big(\frac{-\mu^2/n\|y\|_2^2}{2(1+\mu^2/n)^2(1+2\mu^2/n)}\Big)  \exp\Big( \frac{\mu X_{1}^{T}y}{(1+\mu^{2}/n)^{3/2}} - \frac{\mu^{2}\|X_{1}\|_2^{2}}{2(1+\mu^2/n)}\Big)\mathbbm{1}_{\mathcal{H}_2} \\
&:=\mathcal{E}_1+\mathcal{E}_2,
\end{align*}
where
\[
\mathcal{H}_1:=\Bigg\{\frac{\|y\|_2^{2}}{n(1+\mu^{2}/n)} < \frac{2(1+\mu^{2}/n)\log 2}{\mu(2+\mu^{2}/n)}\Bigg\},~~\mathcal{H}_2:=\Bigg\{\frac{X_1^{T}y}{\|y\|_2^{2}}n \sqrt{1+\frac{\mu^{2}}{n}} < -\frac{\mu(2+\mu^{2}/n)}{4}\Bigg\}.
\]

We first bound $\mathcal{E}_1$. Note that the event $\mathcal{H}_1$ only depends on $y$ and $y$ is independent of $X_1$. We can thus first compute the conditional expectation (conditioning on $y$) with respect to $X_1$. This is already done when we computed $\E\mathcal{A}$ and it is equal to one. Hence, with $\tau=1-\frac{2(1+\mu^2/n)\log 2}{\mu(2+\mu^2/n)}$,
\begin{align*}
\mathcal{E}_1&=\mathbb{P}(\mathcal{H}_1)=\mathbb{P}\Big(\|y\|_2^{2} < \frac{2n(1+\mu^{2}/n)^2\log 2}{\mu(2+\mu^{2}/n)}\Big) \\
&=\mathbb{P}\Big(\chi^2_n< \frac{2n(1+\mu^{2}/n)\log 2}{\mu(2+\mu^{2}/n)}\Big)=\mathbb{P}\big(\chi^2_n<n(1-\tau)\big)\leq \exp\Big(\frac{n}{2}\big(\tau+\log(1-\tau)\big)\Big)=o(1),
\end{align*}
where the last inequality is due to Lemma \ref{lem::chi-square-concentration} and the last equality holds since $\tau \rightarrow 1$ under the condition $\mu\rightarrow \infty, \mu^2/n\rightarrow 0$.

It remains to show $\mathcal{E}_2\rightarrow 0$. Define $\tilde{X}_1=A X_1$ where $A$ is an orthogonal matrix (measurable with respect to $y$) whose first row equals $y/\|y\|_2$. Then it is straightforward to verify that $(y, \tilde{X}_1)\overset{d}{=}(y,X_1)$ and $X_1^Ty=\tilde{X}_{1,1}\|y\|_2$ with $\tilde{X}_1=(\tilde{X}_{1,1},\tilde{X}_{1,-1})$. Using the mutual independence between $(y,\tilde{X}_{1,1},\tilde{X}_{1,-1})$, we can proceed with
\begin{align*}
\mathcal{E}_2&=\Big(\frac{1+2\mu^2/n}{1+\mu^2/n}\Big)^{\frac{n}{2}}\E\exp\Big(-\frac{\mu^{2}\|\tilde{X}_{1,-1}\|_2^{2}}{2(1+\mu^2/n)}\Big) \\
&~\cdot \E\exp\Big(\frac{-\mu^2/n\|y\|_2^2}{2(1+\mu^2/n)^2(1+2\mu^2/n)}\Big)  \exp\Big( \frac{\mu \tilde{X}_{1,1}\|y\|_2}{(1+\mu^{2}/n)^{3/2}} - \frac{\mu^{2}\tilde{X}_{1,1}^{2}}{2(1+\mu^2/n)}\Big)\mathbbm{1}_{\mathcal{H}_3} \\
&\overset{(a)}{=} \Big(\frac{1+2\mu^2/n}{1+\mu^2/n}\Big)^{\frac{1}{2}}\E_y\Bigg[\exp\Big(\frac{-\mu^2/n\|y\|_2^2}{2(1+\mu^2/n)^2(1+2\mu^2/n)}\Big) \\
& \hspace{3.7cm} \cdot \E\bigg(\exp\Big( \frac{\mu \tilde{X}_{1,1}\|y\|_2}{(1+\mu^{2}/n)^{3/2}} - \frac{\mu^{2}\tilde{X}_{1,1}^{2}}{2(1+\mu^2/n)}\Big)\mathbbm{1}_{\mathcal{H}_3}\Big| y\bigg)\Bigg]\\
&\overset{(b)}{=} \mathbb{P}\bigg(Z\leq -\frac{\mu[(2+\mu^2/n)(1+2\mu^2/n)+4]}{4\sqrt{n}(1+\mu^2/n)(1+2\mu^2/n)^{1/2}}\|y\|_2\bigg)\\
&\leq \mathbb{P}\bigg(Z\leq -\frac{\mu[(2+\mu^2/n)(1+2\mu^2/n)+4]}{4\sqrt{n}(1+\mu^2/n)(1+2\mu^2/n)^{1/2}}\|y\|_2, \|y\|_2\geq \frac{1}{2}\sqrt{n+\mu^2}\bigg)+\mathbb{P}\Big(\|y\|_2< \frac{1}{2}\sqrt{n+\mu^2}\Big) \\
&\overset{(c)}{\leq} \mathbb{P}\bigg(Z\leq -\frac{\mu[(2+\mu^2/n)(1+2\mu^2/n)+4]}{8(1+\mu^2/n)^{1/2}(1+2\mu^2/n)^{1/2}}\bigg)+\exp\Big(\frac{n}{2}\big(\frac{3}{4}+\log\frac{1}{4}\big)\Big) \overset{(d)}{\rightarrow} 0,
\end{align*}
where $\mathcal{H}_3=\Big\{\tilde{X}_{1,1}<-\frac{\mu(2+\mu^2/n)}{4n\sqrt{1+\mu^2/n}}\|y\|_2\Big\}$ and $Z\sim \mathcal{N}(0,1)$ independent of $y$. To obtain $(a)$ we have used the moment-generating function of $\chi^2_{n-1}$ to compute the expectation; $(b)$ holds by applying Lemma \ref{lem:incomp_quadexp_Gaussian}; $(c)$ is due to Lemma \ref{lem::chi-square-concentration}; $(d)$ holds under the condition $\mu\rightarrow \infty, \mu^2/n\rightarrow 0$.
\end{proof}

\begin{lemma}\label{lem::regime2-E-p_m^2}
    Consider model \eqref{model::gaussian-model} with $\sigma=1$ and $\beta\in \mathbb{R}^m$. Suppose $n\rightarrow \infty, m \rightarrow \infty$ and $\log(m)/n\rightarrow 0$. If $\mu\rightarrow \infty, \mu = o\big(\sqrt{\log (m)}\big)$ and $\mu^{4}/n\rightarrow 0$, we have
    \begin{equation*}
        \E_{\mu e_{2}} \mathcal{P}_{1}^{2} \geq \frac{e^{\mu^{2}}}{2m^{2}}\big(1+o(1)\big),
    \end{equation*}
    where $\E_{\mu e_{2}}$ denotes the expectation taken under the model $y=X\beta+z$ with $\beta=\mu e_2$, and $\mathcal{P}_{1}$ is defined in \eqref{p1:recall:def}.
\end{lemma}

\begin{proof}
Define
\[
v:=-\frac{n}{\mu}X_{2}+z.
\]
Since $y=\mu X_2+z$, it is direct to verify that $(v,y,X_1,X_3, \ldots, X_n)$ are mutually independent and $X_2=\frac{y-v}{\mu+n/\mu}$. We can then rewrite $\mathcal{P}_{1}$ in \eqref{p1:recall:def} as a function of $(v,y,X_1,X_3, \ldots, X_m)$,    
        \begin{align}
        \E_{\mu e_{2}} \mathcal{P}_{1}^{2} =~& \E \bigg[ \exp\bigg(\mu X_{1}^{T}y - \frac{\mu^{2}}{2}\|X_{1}\|_2^{2} \bigg) - \exp\bigg( -\mu X_{1}^{T}y - \frac{\mu^{2}}{2} \|X_{1}\|_2^{2}\bigg) \bigg]^{2} \nonumber \\
        & \cdot \bigg[ \sum_{j\neq 2} \bigg(\exp\Big(\mu X_{j}^{T}y - \frac{\mu^{2}}{2}\|X_{j}\|_2^{2} \Big) + \exp\Big(-\mu X_{j}^{T}y - \frac{\mu^{2}}{2}\|X_{j}\|_2^{2} \Big)\bigg) \nonumber \\
        &\quad + \exp\bigg(  \frac{\frac{\mu^2}{n}(1+ \frac{\mu^2}{2n})}{(1+ \frac{\mu^2}{n})^2} \|y\|_2^2 - \frac{\mu^4}{ 2 n^2 (1+ \frac{\mu^2}{n})^2 } \|{v}\|_2^2- \frac{\frac{\mu^2}{n} v^Ty}{(1+ \frac{\mu^2}{n})^2}  \bigg)  \nonumber \\
        &\quad + \exp\bigg( - \frac{\frac{\mu^2}{n}(1+ \frac{3\mu^2}{2n})}{(1+ \frac{\mu^2}{n})^2} \|y\|_2^2 - \frac{\mu^4}{ 2 n^2 (1+ \frac{\mu^2}{n})^2 } \|{v}\|_2^2+\frac{\frac{\mu^2}{n}\left ( 1 + \frac{2\mu^2}{n}\right ) v^Ty}{(1+ \frac{\mu^2}{n})^2}  \bigg) \Bigg ]^{-2}
        \nonumber \\
        \geq ~& \E \bigg[ \exp\bigg(\mu X_{1}^{T}y - \frac{\mu^{2}}{2}\|X_{1}\|_2^{2} \bigg) - \exp\bigg( -\mu X_{1}^{T}y - \frac{\mu^{2}}{2} \|X_{1}\|_2^{2}\bigg) \bigg]^{2}\nonumber  \\
        & \cdot \bigg[ 2(m-2)\Big(1+\frac{\mu^{2}}{n} \Big)^{-\frac{n}{2}} \exp \bigg(   \frac{\mu^2\|y\|_2^{2}}{2n(1+\mu^{2}/n)}\bigg) \nonumber \\
        & \quad + \exp \bigg( \mu X_{1}^{T}y - \frac{\mu^{2}}{2}\|X_{1}\|_2^{2} \bigg) + \exp\bigg(-\mu X_{1}^{T}y - \frac{\mu^{2}}{2}\|X_{1}\|_2^{2} \bigg) \nonumber \\
        & \quad + \Big(1+ \frac{1}{1+\frac{n}{\mu^{2}}} \Big)^{-\frac{n}{2}} \exp \Big(\frac{ 3\big(\frac{n}{\mu^{2}} \big)^{2} + 5 \frac{n}{\mu^{2}} + 2 }{2\big(1+\frac{n}{\mu^{2}} \big)^{2}  \big(2+\frac{n}{\mu^{2}} \big)} \|y\|_2^{2} \Big)\nonumber \\
        &\quad + \Big(1+ \frac{1}{1+\frac{n}{\mu^{2}}} \Big)^{-\frac{n}{2}} \exp \Big( - \frac{1}{2\big(1+\frac{n}{\mu^{2}} \big)} \|y\|_2^{2} \Big) \bigg]^{-2},\nonumber \\
         = ~& \E \bigg[ \exp\bigg(\mu X_{1,1}\|y\|_2 - \frac{\mu^{2}}{2}\|X_{1}\|_2^{2} \bigg) - \exp\bigg( -\mu X_{1,1}\|y\|_2 - \frac{\mu^{2}}{2} \|X_{1}\|_2^{2}\bigg) \bigg]^{2}\nonumber  \\
        & \cdot \bigg[ 2(m-2)\Big(1+\frac{\mu^{2}}{n} \Big)^{-\frac{n}{2}} \exp \bigg(   \frac{\mu^2\|y\|_2^{2}}{2n(1+\mu^{2}/n)}\bigg) \nonumber \\
        & \quad + \exp \bigg( \mu X_{1,1}\|y\|_2 - \frac{\mu^{2}}{2}\|X_{1}\|_2^{2} \bigg) + \exp\bigg(-\mu X_{1,1}\|y\|_2 - \frac{\mu^{2}}{2}\|X_{1}\|_2^{2} \bigg) \nonumber \\
        & \quad + \Big(1+ \frac{1}{1+\frac{n}{\mu^{2}}} \Big)^{-\frac{n}{2}} \exp \Big(\frac{ 3\big(\frac{n}{\mu^{2}} \big)^{2} + 5 \frac{n}{\mu^{2}} + 2 }{2\big(1+\frac{n}{\mu^{2}} \big)^{2}  \big(2+\frac{n}{\mu^{2}} \big)} \|y\|_2^{2} \Big)\nonumber \\
        &\quad + \Big(1+ \frac{1}{1+\frac{n}{\mu^{2}}} \Big)^{-\frac{n}{2}} \exp \Big( - \frac{1}{2\big(1+\frac{n}{\mu^{2}} \big)} \|y\|_2^{2} \Big) \bigg]^{-2},\label{eq:pm_2:tempup2}
\end{align}
where to obtain the inequality we have taken the conditional expectation (conditioning on $(y, X_1)$), applied Jensen's inequality (viewing $\mathcal{P}_1^2$ as a convex function $f(x)=\frac{c_1}{(x+c_2)^2}$ of $x>0$), and computed the expectations with respect to $(v, X_3,\ldots, X_m)$ using the moment-generating function of noncentral chi-squared distribution; the last equality holds by an orthogonal transformation (previously used in the proof of Lemma \ref{bounded:dct:zero}) $\tilde{X}_1=A X_1$ with the first row of $A$ equal to $y/\|y\|_2$.

Define the following three events:
\begin{align*}
&\mathcal{A}_1:=\Big\{|X_{1,1}|\leq  \frac{3\mu\|y\|_2}{n(1+\mu^{2}/n)}\Big\},\\ &\mathcal{A}_2:=\Big\{\|X_{1,-1}\|_2^{2}\geq 1-t, {\rm~for~some~constant~} t\in (0,1) \Big \}, \\
&\mathcal{A}_3:=\Big\{\frac{\|y\|_2^{2}}{n(1+\mu^{2}/n)}\leq c, {\rm ~for~some~constant~} c>0\Big\}.
\end{align*}
We aim to find an upper bound for the denominator (inside the square) of \eqref{eq:pm_2:tempup2} on $\mathcal{A}_1\cap \mathcal{A}_2\cap \mathcal{A}_3$. We first have
\begin{align*}
\Big(1+\frac{\mu^2}{n}\Big)^{-\frac{n}{2}}\leq \exp\Big(-\frac{\mu^2}{2}\Big(1-\frac{\mu^2}{2n}\Big)\Big), \quad {\rm~since~}\log(1+x)\geq x-\frac{x^2}{2},~~ \forall x>0.
\end{align*}
Then, on $\mathcal{A}_1\cap \mathcal{A}_2$ we obtain
\begin{align*}
\exp \bigg( \mu X_{1,1}\|y\|_2 - \frac{\mu^{2}}{2}\|X_{1}\|_2^{2} \bigg) + \exp\bigg(-\mu X_{1,1}\|y\|_2 - \frac{\mu^{2}}{2}\|X_{1}\|_2^{2} \bigg) \leq 2\exp\Big(\frac{3\mu^2\|y\|_2^2}{n(1+\mu^2/n)}-\frac{(1-t)\mu^2}{2}\Big).
\end{align*}
Also, since $\mu^2/n\rightarrow 0$, it is straightforward to confirm that 
\begin{align*}
&~\Big(1+ \frac{1}{1+\frac{n}{\mu^{2}}} \Big)^{-\frac{n}{2}}\Bigg[ \exp \Big(\frac{ 3\big(\frac{n}{\mu^{2}} \big)^{2} + 5 \frac{n}{\mu^{2}} + 2 }{2\big(1+\frac{n}{\mu^{2}} \big)^{2}  \big(2+\frac{n}{\mu^{2}} \big)} \|y\|_2^{2} \Big)+  \exp \Big( - \frac{1}{2\big(1+\frac{n}{\mu^{2}} \big)} \|y\|_2^{2} \Big) \Bigg] \\
\leq&~2\exp\Big(\frac{3\mu^2\|y\|_2^2}{n(1+\mu^2/n)}-\frac{(1-t)\mu^2}{2}\Big) \quad {\rm ~for~large~enough~} n
\end{align*}
Collecting the above three results, we obtain that on $\mathcal{A}_1\cap \mathcal{A}_2\cap \mathcal{A}_3$, when $n$ is sufficiently large, the denominator is upper bounded by
\begin{align*}
&~2(m-2)\exp\Big(\frac{\mu^2\|y\|_2^{2}}{2n(1+\mu^{2}/n)} - \frac{\mu^{2}}{2}\Big(1-\frac{\mu^2}{2n}\Big) \Big) + 4 \exp\Big(\frac{3\mu^{2}\|y\|_2^{2}}{n(1+\mu^{2}/n)} - \frac{(1-t)\mu^{2}}{2} \Big) \\
\leq&~ 2(m-2+\sqrt{m})\exp\Big(\frac{\mu^2\|y\|_2^{2}}{2n(1+\mu^{2}/n)} - \frac{\mu^{2}}{2}\Big(1-\frac{\mu^2}{2n}\Big) \Big),
\end{align*}
where the last inequality holds if and only if $\frac{5\mu^2\|y\|_2^2}{2n(1+\mu^2/n)}+\log 2+\frac{\mu^2}{2}\Big(t-\frac{\mu^2}{2n}\Big)\leq \frac{1}{2}\log m$, which is satisfied on $\mathcal{A}_3$ since $\frac{5c}{2}\mu^2+\log 2+\frac{\mu^2}{2}\Big(t-\frac{\mu^2}{2n}\Big)\leq \frac{1}{2}\log m$ under the condition $m\rightarrow \infty, \mu^2=o(\log m)$.
The derived upper bound of the denominator together with \eqref{eq:pm_2:tempup2} implies
    \begin{align*}
        \E_{\mu e_{2}}\mathcal{P}_1^{2} &\geq \E \frac{\Big[\exp\bigg(\mu X_{1,1}\|y\|_2 - \frac{\mu^{2}}{2}\|X_{1}\|_2^{2} \bigg) - \exp\bigg( -\mu X_{1,1}\|y\|_2 - \frac{\mu^{2}}{2} \|X_{1}\|_2^{2}\bigg)\Big]^{2}\mathbbm{1}_{\mathcal{A}_1\cap \mathcal{A}_2\cap \mathcal{A}_3} }{\Big[2(m-2+\sqrt{m})\exp\Big(\frac{\mu^2\|y\|_2^{2}}{2n(1+\mu^{2}/n)} - \frac{\mu^{2}}{2}\Big(1-\frac{\mu^2}{2n}\Big) \Big)\Big]^{2}} \\
        &\geq \E \frac{2\exp\Big(2\mu X_{1,1}\|y\|_2 - \mu^{2}\|X_{1}\|_2^{2} \Big) - 2\exp\Big( - \mu^{2} \|X_{1}\|_2^{2}\Big)}{\Big[ 2(m-2+\sqrt{m})\exp\Big(\frac{\mu^2\|y\|_2^{2}}{2n(1+\mu^{2}/n)} - \frac{\mu^{2}}{2}\Big(1-\frac{\mu^2}{2n}\Big) \Big)\Big]^{2}} \\
        & \quad - \E \frac{2\exp\Big(2\mu X_{1,1}\|y\|_2 - \mu^{2}\|X_{1}\|_2^{2} \Big) \mathbbm{1}_{\mathcal{A}^c_1\cup \mathcal{A}^c_2\cup \mathcal{A}^c_3} }{ \Big[ 2(m-2+\sqrt{m})\exp\Big(\frac{\mu^2\|y\|_2^{2}}{2n(1+\mu^{2}/n)} - \frac{\mu^{2}}{2}\Big(1-\frac{\mu^2}{2n}\Big) \Big)\Big]^{2}}:=\mathcal{E}_1-\mathcal{E}_2.
    \end{align*}
    In the rest of the proof, we will show that $\mathcal{E}_{1}=\frac{e^{\mu^{2}}}{2m^{2}}\cdot \big( 1+o(1)\big)$ and $\mathcal{E}_{2}=o\Big(\frac{e^{\mu^{2}}}{m^{2}} \Big)$. 
    
    We first calculate $\mathcal{E}_{1}$, starting with
    \begin{eqnarray*}
        && \E \exp \Big( 2\mu X_{1,1}\|y\|_2 - \mu^{2}\|X_{1}\|_2^{2} -\frac{\mu^2\|y\|^{2}}{n(1+\mu^{2}/n)} + \mu^{2}\Big(1-\frac{\mu^2}{2n}\Big) \Big) \\
        &=& \Big( 1+\frac{2\mu^{2}}{n} \Big)^{-\frac{n}{2}} \E \exp \Big(\frac{\mu^2\|y\|_2^2}{n(1+\mu^2/n)(1+2\mu^2/n)} + \mu^{2}\Big(1-\frac{\mu^2}{2n}\Big) \Big) \\
        &=& \Big( 1+\frac{2\mu^{2}}{n} \Big)^{-\frac{n}{2}} \Big(1 - \frac{2\mu^2}{n(1+2\mu^2/n)} \Big)^{-\frac{n}{2}} e^{\mu^{2}(1-\frac{\mu^2}{2n}) } \\
        &=& e^{-\mu^2+\mu^2+\mu^{2} + O(\mu^{4}/n)} = e^{\mu^{2}} \cdot \big(1+o(1)\big),
    \end{eqnarray*}
    where we have used the moment-generating function of noncentral chi-squared distribution in the first two equalities, and the last two equalities use the condition $\mu^{4}/n = o(1)$. The negative term in $\mathcal{E}_{1}$ is of smaller order, since
    \begin{align*}
        & \E \exp \Big(-\mu^{2}\|X_{1}\|_2^{2} -\frac{\mu^2\|y\|^{2}}{n(1+\mu^{2}/n)} + \mu^{2}\Big(1-\frac{\mu^2}{2n}\Big) \Big)  \\
        =& \Big(1+\frac{2\mu^{2}}{n} \Big)^{-\frac{n}{2}} \cdot \Big(1+ \frac{2\mu^{2}}{n} \Big)^{-\frac{n}{2}} \cdot e^{\mu^{2}(1-\frac{\mu^2}{2n})} = e^{-2\mu^2(1+o(1))+\mu^{2}(1-\frac{\mu^2}{2n})}=o(e^{\mu^2}),
    \end{align*}
    where the last two equalities use $\frac{\mu^2}{n}\rightarrow 0, \mu\rightarrow \infty$.
    
    To calculate $\mathcal{E}_{2}$, we first bound the term
    \begin{align*}
        & \E \exp\Big(2\mu X_{1,1}\|y\|_2 - \mu^{2}\|X_{1}\|_2^{2} -\frac{\mu^2\|y\|_2^{2}}{n(1+\mu^{2}/n)} \Big) \mathbbm{1}_{\mathcal{A}_1^c} \\
        =& \E\exp\Big(-\mu^2\|X_{1,-1}\|_2^2\Big) \cdot \E \exp\Big(-\frac{\mu^2\|y\|_2^{2}}{n(1+\mu^{2}/n)}+2\mu X_{1,1}\|y\|_2-\mu^2X_{1,1}^2\Big) \mathbbm{1}_{|X_{1,1}|>  \frac{3\mu\|y\|_2}{n(1+\mu^{2}/n)}} \\
        \overset{(a)}{=}& \Big(1+\frac{2\mu^{2}}{n}\Big)^{-\frac{n}{2}} \E\Bigg[\exp\Big(\frac{\mu^2\|y\|_2^2}{n(1+\mu^2/n)(1+2\mu^2/n)}\Big) \\
        & \hspace{2.6cm} \cdot \mathbb{P}_Z \bigg( Z \leq -\frac{(5+8\mu^2/n)\mu \|y\|_2}{\sqrt{n}(1+\mu^2/n)(1+2\mu^2/n)^{1/2}} {\rm ~or~} Z \geq \frac{(1+4\mu^2/n)\mu \|y\|_2}{\sqrt{n}(1+\mu^2/n)(1+2\mu^2/n)^{1/2}}  \bigg) \Bigg]\\
        \overset{(b)}{\leq} & \Big(1+\frac{2\mu^{2}}{n}\Big)^{-\frac{n}{2}} \E\Bigg[\exp\Big(\frac{\mu^2\|y\|_2^2}{n(1+\mu^2/n)(1+2\mu^2/n)}\Big)\cdot 2\exp\Big(-\frac{(1+4\mu^2/n)^2\mu^2\|y\|_2^2}{2n(1+\mu^2/n)^2(1+2\mu^2/n)}\Big)\Bigg]\\
        =& 2\exp \Big[ -\mu^{2}\Big(1+o(1)\Big) + \frac{1}{2}\mu^{2}\Big(1+o(1)\Big) \Big]  = o(1),
    \end{align*}
    where $Z\sim \mathcal{N}(0,1)$ independent of $y$; to obtain $(a)$ we have used Lemma \ref{lem:incomp_quadexp_Gaussian}, and $(b)$ is by the Gaussian tail bound $\mathbb{P}(Z>t)\leq e^{-\frac{t^2}{2}}, \forall t>0$. We proceed to bound the second one,
    \begin{align*}
        & ~\E \exp\Big(2\mu X_{1,1}\|y\|_2 - \mu^{2}\|X_{1}\|_2^{2} -\frac{\mu^2\|y\|_2^{2}}{n(1+\mu^{2}/n)} \Big) \mathbbm{1}_{\mathcal{A}_2^c} \\
        \leq &~\E \exp\Big(2\mu X_{1,1}\|y\|_2 - \mu^{2}X_{1,1}^2 \Big) \cdot \E\exp\Big( - \mu^{2}\|X_{1,-1}\|_2^{2}\Big)\mathbbm{1}_{\|X_{1,-1}\|_2^2<1-t}  \\
        \overset{(c)}{=}& \Big(1+\frac{2\mu^{2}}{n}\Big)^{-\frac{1}{2}} \E \exp\Big(  \frac{2\mu^2\|y\|_2^{2}}{n(1+ 2\mu^{2}/n)} \Big) \cdot \E\exp\Big( - \mu^{2}\|X_{1,-1}\|_2^{2}\Big)\mathbbm{1}_{\|X_{1,-1}\|_2^2<1-t}  \\
        \overset{(d)}{=}& \Big(1+\frac{2\mu^{2}}{n}\Big)^{-\frac{1}{2}} \Big(1 -\frac{4\mu^2(1+\mu^{2}/n)}{n(1+2\mu^{2}/n)}\Big)^{-\frac{n}{2}} \cdot \E\exp\Big( - \mu^{2}\|X_{1,-1}\|_2^{2}\Big)\mathbbm{1}_{\|X_{1,-1}\|_2^2<1-t}  \\
        \overset{(e)}{=}& \Big(1+\frac{2\mu^{2}}{n}\Big)^{-\frac{n}{2}}\Big(1 -\frac{4\mu^2(1+\mu^{2}/n)}{n(1+2\mu^{2}/n)}\Big)^{-\frac{n}{2}} \cdot \mathbb{P}\bigg( \chi^{2}_{n-1} <  n\Big(1+\frac{2\mu^{2}}{n}\Big)(1-t)\bigg) \\
        \overset{(f)}{\leq}& \Big(1+\frac{2\mu^{2}}{n}\Big)^{-\frac{n}{2}}\Big(1 -\frac{4\mu^2(1+\mu^{2}/n)}{n(1+2\mu^{2}/n)}\Big)^{-\frac{n}{2}}  \cdot \exp \bigg[ \frac{n-1}{2} \bigg(\log\Big(1-t' \Big) + t' \bigg) \bigg] \\
        \overset{(g)}{\leq}&e^{\mu^2(1+o(1))-c'n} =o(1), \quad \quad {\rm ~where~}t':=1-\frac{n}{n-1}\Big(1+\frac{2\mu^{2}}{n}\Big)(1-t)>0 {\rm ~when~} n {\rm~is~large}.
    \end{align*}
    Here, $(c)$ and $(d)$ use the moment-generating function of chi-squared distributions; $(e)$ is obtained by applying the identity in \eqref{chi:square:form}; $(f)$ is due to Lemma \ref{lem::chi-square-concentration}; $(g)$ holds by choosing $t\in (0,1)$ sufficiently close to $1$ so that $e^{\frac{n-1}{2}(\log(1-t' ) + t' )}\leq e^{-c'n}$ for some constant $c'>0$; and the last equality follows under the condition $\mu^2/n\rightarrow 0$. It remains to bound the third term,
    \begin{align*}
        & \E \exp\Big(2\mu X_{1,1}\|y\|_2 - \mu^{2}\|X_{1}\|_2^{2} - \frac{\mu^2\|y\|_2^{2}}{n(1+\mu^{2}/n)} \Big) \mathbbm{1}_{\mathcal{A}_3^c} \\
        =& \E\exp\Big(-\mu^2\|X_{1,-1}\|_2^2\Big) \cdot \E \exp\Big(2\mu X_{1,1}\|y\|_2 - \mu^{2}X_{1,1}^{2} - \frac{\mu^2\|y\|_2^{2}}{n(1+\mu^{2}/n)} \Big) \mathbbm{1}_{\|y\|_2^{2} > cn(1+\mu^{2}/n)} \\
        =& \Big(1+\frac{2\mu^{2}}{n}\Big)^{-\frac{n}{2}} \E  \exp\Big(\frac{\mu^2\|y\|_2^2}{n(1+\mu^2/n)(1+2\mu^2/n)}\Big) \mathbbm{1}_{\|y\|_2^{2} > cn(1+\mu^{2}/n)} \\
        \overset{(h)}{=}& \Big(1+\frac{2\mu^{2}}{n}\Big)^{-\frac{n}{2}} \cdot \Big( 1-\frac{2\mu^2}{n(1+2\mu^2/n)}\Big)^{-\frac{n}{2}} \cdot \mathbb{P} \bigg[\chi^{2}_{n} > nc\Big(1-\frac{2\mu^2}{n(1+2\mu^2/n)}\Big) \bigg] \\
        \overset{(i)}{=}& \exp \Big(O\Big(\frac{\mu^{4}}{n}\Big) \Big) \cdot \exp \bigg[ - \frac{n}{2} \Big[ t''-1 - \log t'' \Big]\bigg] = o(1), \quad t'':= c\Big(1-\frac{2\mu^2}{n(1+2\mu^2/n)}\Big),
    \end{align*}
    where $(h)$ follows by applying the identity in \eqref{chi:square:form}; $(i)$ is by Lemma \ref{lem::chi-square-concentration}; and the last equality holds by choosing $c$ large enough so that $t''-1-\log t''>c''$ for some constant $c''>0$, together with the condition $\mu^4/n\rightarrow 0$.
\end{proof}

\subsection{Upper bound}
\label{app::subsec::enet}

\begin{proof}
Recall the soft thresholding function: for a given $\chi\geq 0$,
\begin{align*}
\eta(u,\chi)&=\argmin_{t\in \mathbb{R}}\frac{1}{2}(t-u)^2+\chi|t| \\
&={\rm sign}(u)\cdot (|u|-\chi)_+, \quad \forall u \in \mathbb{R}.
\end{align*}
For a vector $\nu=(\nu_1,\ldots,\nu_p)$, the notation $\eta(\nu,\chi)$ represents the soft thresholding function applied componentwise to each element $\nu_i$. Then, the estimator $\betahy$ in \eqref{elastic:reg:est} can be rewritten as follows:
\begin{align}
\label{els:reform}
\betahy(\lambda, \gamma)&=\argmin_{b\in \mathbb{R}^p}\|b-X^Ty\|_2^2+\lambda\|b\|_1+\gamma \|b\|^2_2 \nonumber \\
&=\argmin_{b\in \mathbb{R}^p}\frac{1}{2}\Big\|b-\frac{X^Ty}{1+\gamma}\Big\|_2^2+\frac{\lambda}{2(1+\gamma)}\|b\|_1 \nonumber \\
&=\eta\Big(\frac{X^Ty}{1+\gamma}, \frac{\lambda}{2(1+\gamma)}\Big)=\frac{1}{1+\gamma}\eta(X^Ty,\lambda/2),
\end{align}
where in the last equality we have used the invariance property $\eta(a\nu,a\chi)=a\eta(\nu,\chi), \forall a\geq 0$. Dividing $\betahy(\lambda, \gamma)$ in \eqref{els:reform} by $\sigma$, it is equivalent to prove
\begin{align}
\label{equi:form:rescale}
\sup_{\beta \in\Theta(k,\mu)}\E_{\beta} \|\betahy(\lambda, \gamma) - \beta \|_2^{2} \leq k\mu^{2} - \frac{2+o(1)}{\sqrt{2\pi}}\cdot\frac{k^2}{p}\cdot \mu e^{\mu^2},
\end{align}
under unit noise variance $\sigma=1$ and the tuning $\lambda=4\mu,\gamma=\frac{p}{2k\mu^2}e^{-\frac{3\mu^2}{2}}-1$. To simplify notations, we write $\mathbb{E}$ for $\E_{\beta}$ and $\betahy$ for $\betahy(\lambda, \gamma)$ throughout the rest of the proof.

Let $\tilde{X} \in \mathbb{R}^{n\times p}$ be an independent copy of $X$, and write out their columns as
\begin{align*}
X=
\begin{pmatrix}
X_1&X_2 & \cdots & X_p
\end{pmatrix}
, \quad \tilde{X}=
\begin{pmatrix}
\tilde{X}_1& \tilde{X}_2 & \cdots & \tilde{X}_p
\end{pmatrix}
.
\end{align*}
We first obtain a key decomposition,
\begin{align}
\label{key:de:elastic}
(X^Ty)_i&=(X^TX\beta+X^Tz)_i \nonumber \\
&=\beta_i+ \underbrace{\Big(\sum_{j\neq i}X_i^TX_j\beta_j+X_i^T\tilde{X}_i\beta_i+X_i^Tz\Big)}_{:=w_i}+\underbrace{(\|X_i\|_2^2-1-X_i^T\tilde{X}_i)\beta_i}_{:=\Delta_i}
\end{align}
Define $w=(w_1,\ldots, w_p), \Delta=(\Delta_1,\ldots, \Delta_p)$. To calculate the risk of $\betahy$, we consider the following intermediate estimator:
\begin{equation}
\label{inter:med:est}
    \betaTE = \frac{1}{1+\gamma} \eta(\beta+w, \lambda/2),
\end{equation}
Note that $\betaTE$ is not a valid estimator since it relies on $\beta$. We only use it in the proof to help evaluate the risk of $\betahy$. Indeed, as will be shown shortly,  the squared distance between $\betahy$  and $\betaTE$ is negligible, and hence their risks are asymptotically the same. We can then focus on $\betaTE$ which is more amenable to risk calculation.

To this end, \eqref{key:de:elastic} allows us to rewrite $\betahy$ as $\betahy=\frac{1}{1+\gamma} \eta(\beta+w+\Delta, \lambda/2)$. Combining this with \eqref{inter:med:est}, we have $\forall \beta\in \Theta(k,\mu)$,
\begin{align}
\label{close:distance}
\E\|\betahy-\betaTE\|_2^2 &\leq \frac{1}{(1+\gamma)^2}\E\|\Delta\|_2^2 \nonumber \\
&=\frac{4k^2\mu^4}{p^2}e^{3\mu^2}\|\beta\|_2^2 \cdot \E(\|X_i\|_2^2-1-X_i^T\tilde{X}_i)^2 \nonumber \\
&\leq \frac{8k^2\mu^4}{p^2}e^{3\mu^2}\|\beta\|_2^2 \cdot \big({\rm Var}(\|X_i\|_2^2)+{\rm Var}(X_i^T\tilde{X}_i)\big) \nonumber \\
&= \frac{24k^2\mu^4}{np^2}e^{3\mu^2}\|\beta\|_2^2 \leq \frac{24k^3\mu^6}{np^2}e^{3\mu^2}.
\end{align}
Here, the first inequality holds by the nonexpansiveness of the soft thresholding function; in the first equality we used the tuning $\gamma=\frac{p}{2k\mu^2}e^{-\frac{3\mu^2}{2}}-1$; the second inequality is based on the basic inequality $(a-b)^2\leq 2(a^2+b^2)$; the last inequality is due to $\|\beta\|_2^2\leq k\mu^2$ for any $\beta\in \Theta(k,\mu)$.

Now we can calculate the risk of $\betahy$ in the following way, $\forall \beta\in \Theta(k,\mu)$,
\begin{align}
\label{risk:eval:map}
&~\E\|\betahy-\beta\|_2^2 \nonumber \\
=&~\E\|\betaTE-\beta\|_2^2+\E\|\betahy-\betaTE\|_2^2+2\E\langle \betahy-\betaTE, \betaTE-\beta \rangle \nonumber \\
\leq &~ \E\|\betaTE-\beta\|_2^2+ \frac{24k^3\mu^6}{np^2}e^{3\mu^2}+ \frac{4\sqrt{6}k^{3/2}\mu^3}{n^{1/2}p}e^{3\mu^2/2}\cdot \sqrt{\E\|\betaTE-\beta\|_2^2},
\end{align}
where we have used Cauchy–Schwarz inequality and \eqref{close:distance} in the last step. It remains to calculate the risk of $\betaTE$. It is direct to verify that for each $i=1,2\ldots, p$, 
\begin{align}
\label{dis:handy:form}
w_i | X_i \sim \mathcal{N}\Big(0, \frac{(\|\beta\|_2^2+n)\|X_i\|_2^2}{n}\Big), \quad X_i \sim \mathcal{N}\Big(0,\frac{1}{n}I_n\Big).
\end{align}
Hence, $\betaTE$ in \eqref{inter:med:est} can be viewed as a denoising estimator under a sequence model with identically distributed errors $w_1,\ldots, w_p$. We calculate such risk in Lemma \ref{lem::estimate-betaTE}. Note the conditions $\frac{k\mu^4}{n}\rightarrow 0$ and $\frac{\mu^6}{n}\rightarrow 0$ of Lemma \ref{lem::estimate-betaTE} are satisfied since $\frac{k(\log(p/k))^2}{n}=o(1), (p/k)^{\alpha}\leq n$. We can then plug \eqref{risk:oracle:form} into \eqref{risk:eval:map} to obtain conclude
\begin{align*}
\sup_{\beta \in\Theta(k,\mu)}\E \|\betahy - \beta \|_2^{2} \leq &~k\mu^{2} - \frac{2+o(1)}{\sqrt{2\pi}}\cdot\frac{k^2}{p}\cdot \mu e^{\mu^2}+\frac{24k^3\mu^6}{np^2}e^{3\mu^2}+\\
&~\frac{(4\sqrt{6}+o(1))k^2\mu^4}{n^{1/2}p}e^{3\mu^2/2} \\
=&~k\mu^{2} - \frac{2+o(1)}{\sqrt{2\pi}}\cdot\frac{k^2}{p}\cdot \mu e^{\mu^2},
\end{align*}
where the last step holds because
\begin{align*}
&\frac{k^3\mu^6}{np^2}e^{3\mu^2} \ll \frac{k^2}{p} \mu e^{\mu^2} \quad {\rm ~~when~} \mu\rightarrow 0, \mu= o\big( \sqrt{\log(p/k)}\big), \\
&\frac{k^2\mu^4}{n^{1/2}p}e^{3\mu^2/2} \ll \frac{k^2}{p} \mu e^{\mu^2} \quad {\rm ~~when~} \mu\rightarrow 0, \mu= o\big( \sqrt{\log(p/k)}\big), (p/k)^{\alpha}\leq n.
\end{align*}

\end{proof}

\begin{lemma}\label{lem::estimate-betaTE}
    Assume model \eqref{model::gaussian-model} with $\sigma = 1$. Suppose $\frac{k\mu^4}{n}\rightarrow 0$ and $\frac{\mu^6}{n}\rightarrow 0$. Then, in Regime II where $\mu\rightarrow \infty$ and $\mu = o\big( \sqrt{\log(p/k)}\big)$, the estimator $\betaTE(\lambda,\gamma)$ with tuning $\lambda=4\mu,\gamma=\frac{p}{2k\mu^2}e^{-\frac{3\mu^2}{2}}-1$ satisfies
    \begin{equation}
    \label{risk:oracle:form}
        \sup_{\beta \in\Theta(k,\mu)}\E_{\beta} \|\betaTE(\lambda, \gamma) - \beta \|_2^{2} \leq k\mu^{2} - \frac{2+o(1)}{\sqrt{2\pi}}\cdot\frac{k^2}{p}\cdot \mu e^{\mu^2}.
    \end{equation}
\end{lemma}
\begin{proof}
We write $\mathbb{E}$ for $\E_{\beta}$ and $\betaTE$ for $\betaTE(\lambda, \gamma)$ for convenience. 

Based on \eqref{inter:med:est},\eqref{dis:handy:form} and \eqref{1d:risk:general}, we can re-express the risk of $\betaTE$ as
\begin{align}
\label{break:one}
\sup_{\beta \in\Theta(k,\mu)}\E \|\betaTE - \beta \|_2^{2}=\sup_{0\leq \delta \leq \mu} \sup_{\|\beta\|_0\leq k, \|\beta\|_2^2=k\delta^2 }\sum_{i=1}^pg(\beta_i; \lambda/2, \gamma),
\end{align}
where the distribution of $\omega$ in \eqref{1d:risk:general} is specified by $\omega|\theta \sim \mathcal{N}(0,\sigma_{\theta}^2), \theta\sim \mathcal{N}(0, n^{-1}I_n)$ and $\sigma_{\theta}^2=\frac{(\|\beta\|_2^2+n)\|\theta\|_2^2}{n}$. Since $\|\beta\|_2^2$ is fixed at the same level in the first supremum of \eqref{break:one}, the distribution of $\omega$ does not depend on $\beta$ anymore. Thus, we can apply Lemma \ref{1d:risk:prop} to conclude that the first supremum in \eqref{break:one} is attained by a vector $\beta$ that has $k$ identical non-zero components equal to $\delta$:
\begin{align}
\label{break:two}
\sup_{\beta \in\Theta(k,\mu)}\E \|\betaTE - \beta \|_2^{2}=\sup_{0\leq \delta \leq \mu} \underbrace{\big[(p-k)g(0;\lambda/2,\gamma)+kg(\delta;\lambda/2,\gamma)\big]}_{:=f(\delta)}
\end{align}
Note that the function $g$ above itself depends on $\delta$ as well, through the distribution of $\omega$. To facilitate further analysis, let us write out $f(\delta)$ with tuning $\lambda=4\mu, \gamma=\frac{p}{2k\mu^2}e^{-\frac{3\mu^2}{2}}-1$ in a more explicit form,
\begin{align}
\label{break:three}
f(\delta)&=4(p-k)k^2p^{-2}\mu^4e^{3\mu^2}\E\eta^2(\omega,2\mu)+4k^3p^{-2}\mu^4e^{3\mu^2}\E\eta^2(\delta+\omega,2\mu) \nonumber \\
&\quad \quad -4\delta k^2\mu^2p^{-1}e^{3\mu^2/2}\E\eta(\delta+\omega,2\mu)+k\delta^2,
\end{align}
where the distribution of $\omega$ is specified by $\omega|\theta \sim \mathcal{N}(0,\sigma_{\theta}^2), \theta\sim \mathcal{N}(0, n^{-1}I_n)$ and $\sigma_{\theta}^2=\frac{(k\delta^2+n)\|\theta\|_2^2}{n}$. 

To evaluate $f(\delta)$, we bound the three expectations in \eqref{break:three}. Let $a_{\theta}^2=\frac{(k\mu^2+n)\|\theta\|_2^2}{n}$, and $\phi(\cdot)$ and $\Phi(\cdot)$ denote the pdf and cdf of $\mathcal{N}(0,1)$, respectively. We obtain $\forall \delta \in [0,\mu]$,
\begin{align}
\label{exp:one}
\E\eta^2(\omega,2\mu)&=\E\Big(2\sigma_{\theta}^2\Big[(1+4\mu^2\sigma_{\theta}^{-2})(1-\Phi(2\mu\sigma_{\theta}^{-1}))-2\mu\sigma_{\theta}^{-1}\phi(2\mu\sigma_{\theta}^{-1})\Big]\Big) \nonumber \\
&\leq \E\Big((\sigma_{\theta}^5\mu^{-3}/2+3\sigma_{\theta}^7\mu^{-5}/16)\phi(2\mu\sigma_{\theta}^{-1})\Big) \nonumber \\
&\leq \E\Big((a_{\theta}^5\mu^{-3}/2+3a_{\theta}^7\mu^{-5}/16)\phi(2\mu a_{\theta}^{-1})\Big) \nonumber \\
&\leq \frac{1+o(1)}{2\sqrt{2\pi}\mu^3}e^{-2\mu^2\big(1+O(\frac{k\mu^2+\mu\sqrt{n}}{n})\big)}=\frac{1+o(1)}{2\sqrt{2\pi}\mu^3}e^{-2\mu^2}.
\end{align}
Here, the first equality is obtained by calculating the expectation conditional on $\theta$, using the explicit expression of the soft thresholding function $\eta$ defined in \eqref{soft:thd:f}; the first inequality uses the non-asymptotic Gaussian tail bound from Lemma \ref{lem::gaussian-tail-mills-ratio}: $1-\Phi(t)\leq (t^{-1}-t^{-3}+3t^{-5})\phi(t), \forall t>0$; the second inequality holds since the function inside the expectation is increasing in $\sigma_{\theta}$ and $\sigma_{\theta}\leq a_{\theta}$; the third inequality applies \eqref{chi:upper} of Lemma \ref{concen:chi:exp} with $c_1=5/2 {\rm ~or~} 7/2, c_2=\frac{2n^2\mu^2}{k\mu^2+n} $, under the conditions $\mu\rightarrow \infty, \mu=o(\sqrt{\log(p/k)}), \frac{k\mu^2}{n}=o(1)$; the last step is due to $\frac{k\mu^4}{n}\rightarrow 0$ and $\frac{\mu^6}{n}\rightarrow 0$.

Using similar arguments, we can bound the other two expectations:
\begin{align}
\label{exp:two}
\E\eta^2(\delta+\omega,2\mu)=&~\E\Big((\sigma_{\theta}^2+(2\mu-\delta)^2)(1-\Phi((2\mu-\delta)\sigma_{\theta}^{-1}))-\sigma_{\theta}(2\mu-\delta)\phi((2\mu-\delta)\sigma_{\theta}^{-1})\Big) + \nonumber \\
&~\E\Big((\sigma_{\theta}^2+(2\mu+\delta)^2)(1-\Phi((2\mu+\delta)\sigma_{\theta}^{-1}))-\sigma_{\theta}(2\mu+\delta)\phi((2\mu+\delta)\sigma_{\theta}^{-1})\Big) \nonumber \\
\leq &~\E\Big((2\sigma_{\theta}^5(2\mu-\delta)^{-3}+3\sigma_{\theta}^7(2\mu-\delta)^{-5})\phi((2\mu-\delta)\sigma_{\theta}^{-1})\Big)+ \nonumber \\
&~\E\Big((2\sigma_{\theta}^5(2\mu+\delta)^{-3}+3\sigma_{\theta}^7(2\mu+\delta)^{-5})\phi((2\mu+\delta)\sigma_{\theta}^{-1})\Big) \nonumber \\
\leq&~\E\Big((2a_{\theta}^5\mu^{-3}+3a_{\theta}^7\mu^{-5})\phi(\mu a_{\theta}^{-1})\Big)+\E\Big((2a_{\theta}^5(2\mu)^{-3}+3a_{\theta}^7(2\mu)^{-5})\phi(2\mu a_{\theta}^{-1})\Big) \nonumber \\
\leq&~ \frac{2+o(1)}{\sqrt{2\pi}\mu^3}e^{-\frac{\mu^2}{2}},
\end{align}
 and 
 \begin{align}
 \label{exp:three}
\E\eta(\delta+\omega,2\mu)=&~\E\Big(\sigma_{\theta}\phi((2\mu-\delta)\sigma_{\theta}^{-1})-(2\mu-\delta)(1-\Phi((2\mu-\delta)\sigma_{\theta}^{-1}))\Big)- \nonumber \\
&~\E\Big(\sigma_{\theta}\phi((2\mu+\delta)\sigma_{\theta}^{-1})-(2\mu+\delta)(1-\Phi((2\mu+\delta)\sigma_{\theta}^{-1}))\Big) \nonumber \\
 \geq &~\E\Big((\sigma_{\theta}^3(2\mu-\delta)^{-2}-3\sigma_{\theta}^5(2\mu-\delta)^{-4})\phi((2\mu-\delta)\sigma_{\theta}^{-1})\Big)- \nonumber \\
&~\E\Big(\sigma_{\theta}^3(2\mu+\delta)^{-2}\phi((2\mu+\delta)\sigma_{\theta}^{-1})\Big)  \nonumber \\
\geq&~\E\Big(\|\theta\|_2^3(2\mu-\delta)^{-2}\phi((2\mu-\delta)\|\theta\|_2^{-1})\Big)-\E\Big(3a_{\theta}^5\mu^{-4}\phi(\mu a_{\theta}^{-1})\Big)- \nonumber \\
&~\E\Big(a_{\theta}^3(2\mu)^{-2}\phi(2\mu a_{\theta}^{-1})\Big) \nonumber \\
\geq &~\E\Big(\|\theta\|_2^3(2\mu-\delta)^{-2}\phi((2\mu-\delta)\|\theta\|_2^{-1})\Big)-\frac{3+o(1)}{\sqrt{2\pi}\mu^4}e^{-\frac{\mu^2}{2}} \nonumber \\
\geq &~\frac{(2\mu-\delta)^{-2}}{\sqrt{2\pi}}\Big(\frac{n-2}{n}\Big)^{3/2}e^{-\frac{n(2\mu-\delta)^2}{2(n-2)}}-\frac{3+o(1)}{\sqrt{2\pi}\mu^4}e^{-\frac{\mu^2}{2}}.
 \end{align}
To obtain \eqref{exp:three}, we need a bit more effort since $\E\eta(\delta+\omega,2\mu)$ is the difference between two expectations. In particular, to derive the first inequality above, we have used the two-sided Gaussian tail bounds: $(t^{-1}-t^{-3})\phi(t)\leq 1-\Phi(t)\leq (t^{-1}-t^{-3}+3t^{-5})\phi(t), \forall t>0$ (see Lemma \ref{lem::gaussian-tail-mills-ratio}), and the last inequality is due to \eqref{chi:lower} of Lemma \ref{concen:chi:exp}.

Now, combining \eqref{break:two}-\eqref{exp:three} with some simple calculations gives
\begin{align}
&~\sup_{\beta \in\Theta(k,\mu)}\E \|\betaTE - \beta \|_2^{2} \nonumber \\
\leq &~\frac{2+o(1)}{\sqrt{2\pi}}\cdot \frac{k^2}{p}\cdot \mu e^{\mu^2}+ \nonumber \\
&~\sup_{0\leq \delta \leq \mu}\underbrace{k\delta^2-4(2\pi)^{-1/2}k^2p^{-1}\mu^2e^{3\mu^2/2}\big(\frac{n-2}{n}\big)^{3/2}(2\mu-\delta)^{-2}\delta e^{-\frac{n(2\mu-\delta)^2}{2(n-2)}}}_{:=h(\delta)} \label{close:finish}
\end{align}
It is straightforward to verify (by calculating the derivative) that $h(\delta)$ is increasing over $[1,\mu]$ under the condition $\mu\rightarrow \infty, \mu=o(\sqrt{\log(p/k)})$. Hence, 
\begin{align*}
\sup_{0\leq \delta \leq \mu} h(\delta) &=\Big(\sup_{0\leq \delta \leq 1} h(\delta)\Big) \vee \Big(\sup_{1 \leq \delta \leq \mu} h(\delta)\Big) \\
&\leq k \vee h(\mu)=k\mu^2-\frac{4+o(1)}{\sqrt{2\pi}}\cdot \frac{k^2}{p}\cdot \mu e^{\mu^2}.
\end{align*}
This together with \eqref{close:finish} completes the proof.
\end{proof}

\begin{lemma}
\label{1d:risk:prop}
For a given pair $(\chi_1,\chi_2)$, define
\begin{align}
\label{1d:risk:general}
g(u;\chi_1,\chi_2) = \E \Big(\frac{1}{1+\chi_2} \eta (u + \omega, \chi_1 ) - u\Big)^{2}, \quad \forall u \in \mathbb{R},
\end{align}
where $\omega$ follows a mixture of Gaussians: there exists a random variable $\theta$ such that $\omega|\theta \sim \mathcal{N}(0,\sigma^2_{\theta})$ with $\E\sigma^2_{\theta}<\infty$ and $\sigma_{\theta}>0$. For any pair $(\chi_1,\chi_2)$ satisfying $\chi_1>0,\chi_2\geq 0$, it holds that 
\begin{enumerate}[label=(\roman*)]
    \item $g(u;\chi_1,\chi_2)$, as a function of $u$, is symmetric, and increasing over $u \in [0,+\infty)$.
    \item $\max_{(x,y):x^{2}+y^{2} = c^{2}} [g(x;\chi_1,\chi_2) + g(y;\chi_1,\chi_2)] = 2g(c/\sqrt{2};\chi_1,\chi_2)$, \quad $\forall c>0.$
\end{enumerate}
\end{lemma}

\begin{proof}
This lemma generalizes Lemma \ref{lem::elastic-net-risk} from the Gaussian to a mixture of Gaussians setting. Adopting the notation \eqref{1d:risk:f}, we can write
\begin{align*}
g(u;\chi_1,\chi_2)&=\E\Big\{\E \Big[\Big(\frac{1}{1+\chi_2} \eta (u + \omega, \chi_1 ) - u\Big)^{2}\Big | \theta\Big] \Big\}\\
&=\E\Big\{\sigma^2_{\theta} \cdot \E\Big[ \Big(\frac{1}{1+\chi_2} \eta (u\sigma^{-1}_{\theta} + \omega\sigma^{-1}_{\theta}, \chi_1\sigma^{-1}_{\theta} ) - u\sigma^{-1}_{\theta}\Big)^{2}\Big | \theta\Big]\Big\}\\
&=\E\big( \sigma^2_{\theta} r(u\sigma^{-1}_{\theta}; \chi_1\sigma^{-1}_{\theta},\chi_2)\big).
\end{align*}
Since $0<\sigma_{\theta}^{-1}<\infty$, Lemma \ref{lem::elastic-net-risk} Part (i) implies that  $\sigma^2_{\theta} r(u\sigma^{-1}_{\theta}; \chi_1\sigma^{-1}_{\theta},\chi_2)$, as a function of $u$, is symmetric, and increasing over $u \in [0,+\infty)$. So is its expectation, i.e. $g(u;\chi_1,\chi_2)$.

Moreover, we have
\begin{align*}
&~\max_{(x,y):x^{2}+y^{2} = c^{2}} \big[g(x;\chi_1,\chi_2) + g(y;\chi_1,\chi_2)\big] \\
=&~\max_{(x,y):x^{2}+y^{2} = c^{2}} \E \Big(\sigma^2_{\theta} \cdot\big[r(x\sigma^{-1}_{\theta}; \chi_1\sigma^{-1}_{\theta},\chi_2) + r(y\sigma^{-1}_{\theta}; \chi_1\sigma^{-1}_{\theta},\chi_2)\big]\Big) \\
\leq &~ \E \Big(\sigma^2_{\theta} \cdot \max_{(x,y):x^{2}+y^{2} = c^{2}}\big[r(x\sigma^{-1}_{\theta}; \chi_1\sigma^{-1}_{\theta},\chi_2) + r(y\sigma^{-1}_{\theta}; \chi_1\sigma^{-1}_{\theta},\chi_2)\big]\Big) \\
= &~\E \Big(\sigma^2_{\theta} \cdot 2r(\sigma^{-1}_{\theta}c/\sqrt{2}; \chi_1\sigma^{-1}_{\theta},\chi_2) \Big)=2g(c/\sqrt{2};\chi_1,\chi_2),
\end{align*}
where the inequality is due to Lemma \ref{lem::elastic-net-risk} Part (ii). The upper bound $2g(c/\sqrt{2};\chi_1,\chi_2)$ is attained by $x=y=c/\sqrt{2}$.

\end{proof}
 
\begin{lemma}
\label{concen:chi:exp}
For a chi-squared random variable $Q\sim \chi^2_n$ with $n>2$, it holds that
\begin{align}
&\E\big(Q^{c_1}e^{-\frac{c_2}{Q}}\big)\geq (n-2)^{c_1}e^{\frac{c_2}{n-2}}, \quad \forall c_1,c_2>0,\label{chi:lower}
\end{align}
and for all $c_1\geq \frac{1}{4},0<c_2<\frac{n^2}{36}$,
\begin{align}
&\E\big(Q^{c_1}e^{-\frac{c_2}{Q}}\big)\leq (n+6\sqrt{c_2})^{c_1}e^{-\frac{c_2}{n+6\sqrt{c_2}}}+(C\sqrt{c_1}+\sqrt{n})^{2c_1}e^{-\frac{2c_2}{n+6\sqrt{c_2}}}. \label{chi:upper}
\end{align}
\end{lemma}

\begin{proof}
Since $h(t):=t^{-c_1}e^{-c_2t}$ is convex over $(0,\infty)$, Jensen's inequality yields,
\begin{align*}
\E\big(Q^{c_1}e^{-\frac{c_2}{Q}}\big)&=\E h(Q^{-1}) \\
&\geq h(\E Q^{-1})=h((n-2)^{-1})=(n-2)^{c_1}e^{\frac{c_2}{n-2}},
\end{align*}
where the second equality uses the expectation of inverse-chi-squared distribution. Regarding the upper bound \eqref{chi:upper}, we will use the following concentration inequality (see Example 2.11 in \cite{wainwright2019high}),
\begin{align}
\label{chi:concen:use}
\mathbb{P}(Q-n\geq t)\leq e^{-\frac{t^2}{8n}}, ~~t\in (0,n). 
\end{align}
We have $\forall t\in (0,n)$,
\begin{align*}
\E\big(Q^{c_1}e^{-\frac{c_2}{Q}}\big)&=\E\big(Q^{c_1}e^{-\frac{c_2}{Q}}\mathbbm{1}_{Q<n+t}\big)+\E\big(Q^{c_1}e^{-\frac{c_2}{Q}}\mathbbm{1}_{Q\geq n+t}\big) \\
&\leq (n+t)^{c_1}e^{-\frac{c_2}{n+t}}+\E\big(Q^{c_1}\mathbbm{1}_{Q\geq n+t}\big) \\
&\leq (n+t)^{c_1}e^{-\frac{c_2}{n+t}}+\sqrt{\E Q^{2c_1}}\cdot \sqrt{\mathbb{P}(Q-n\geq t)} \\
&\leq (n+t)^{c_1}e^{-\frac{c_2}{n+t}}+(C\sqrt{c_1}+\sqrt{n})^{2c_1}e^{-\frac{t^2}{16n}},
\end{align*}
for a absolute constant $C>0$. Here, the first inequality holds since $Q^{c_1}e^{-\frac{c_2}{Q}}$ is increasing in $Q$; the second inequality applies Cauchy–Schwarz inequality; the third inequality uses \eqref{chi:concen:use}, the Minkowski inequality $(\E\sqrt{Q}^{4c_1})^{\frac{1}{4c_1}}\leq (\E|\sqrt{Q}-\sqrt{n}|^{4c_1})^{\frac{1}{4c_1}}+\sqrt{n}$, and the fact that $\sqrt{Q}-\sqrt{n}$ has a constant sub-Gaussian norm so that its scaled moment $(4c_1)^{-1/2}(\E|\sqrt{Q}-\sqrt{n}|^{4c_1})^{\frac{1}{4c_1}}$ is bounded by a constant (see Proposition 2.5.2 and Theorem 3.1.1 in \cite{vershynin2018high}). Since $c_2<\frac{n^2}{36}$, we set $t=6\sqrt{c_2}<n$ so that $\frac{t^2}{16n}>\frac{2c_2}{n+t}$. We can then continue from the above inequality to obtain
\[
\E\big(Q^{c_1}e^{-\frac{c_2}{Q}}\big)\leq (n+6\sqrt{c_2})^{c_1}e^{-\frac{c_2}{n+6\sqrt{c_2}}}+(C\sqrt{c_1}+\sqrt{n})^{2c_1}e^{-\frac{2c_2}{n+6\sqrt{c_2}}}.
\]

\end{proof}

\newpage

%% file: proof_prop_ridge_reg2.tex
\section{Proof of Proposition~\ref{prop::ridge-suboptimality-second-regime}}
\label{ssec:proof:ridge:suboptimal}

It is straightforward to verify that the ridge estimator $\hat{\beta}^R(\lambda)=(X^TX+\lambda I)^{-1}X^Ty$ satisfies the following scale-invariance property,
\[
\sup_{\beta\in \Theta(k,\tau)}\E\|\hat{\beta}^R(\lambda)-\beta\|_2^2=\sigma^2\cdot \sup_{\beta\in \Theta(k,\mu)}\E\|\hat{\beta}^R(\lambda)-\beta\|_2^2.
\]
Hence, without loss of generality, we can assume $\sigma=1$ and prove
\begin{align*}
\inf_{\lambda >0}\sup_{\beta\in \Theta(k,\mu)}\E\|\hat{\beta}^R(\lambda)-\beta\|_2^2\geq k\mu^2\Big(1-\frac{k\mu^2}{p}(1+o(1))\Big).
\end{align*}
To prove the above, we consider a specific parameter value $\beta^*\in \Theta(k,\mu)$ with $\beta^*_i = \mu$ for $i \in \{1, 2, ..., k\}$ and $\beta^*_i = 0$ otherwise. It is then sufficient to prove
\begin{align}
\label{reduce:sigma:one}
\inf_{\lambda >0}\E\|\hat{\beta}^R(\lambda)-\beta^*\|_2^2\geq k\mu^2\Big(1-\frac{k\mu^2}{p}(1+o(1))\Big).
\end{align}
We start with the decomposition derived in \eqref{eq::ridge-estimator-risk-decomposition},
\begin{equation}
    \label{eq::ridge-decomp-again}
    \E\|\hat{\beta}^R(\lambda) - \beta^*\|_2^2 = \E \| (X^{T}X+\lambda I)^{-1}\lambda \beta^* \|_2^{2}+\E\|(X^{T}X+\lambda I)^{-1}X^{T}z\|_2^{2}. 
\end{equation}

Let \(X^TX = Q\Lambda Q^T\) denote the spectral decomposition of the matrix $X^TX$, where $\sigma_1\geq \cdots \geq \sigma_p$ denote the diagonal elements of $\Lambda$ and \(q_i\) denotes the \(i\)th column of $Q$ for \(i \in \{1, 2, ..., p\}\). Since $n X^TX$ is a standard Wishart matrix, each eigenvector $q_i$ is uniformly distributed on the unit sphere in \(\mathbb{R}^p\) \cite{bai2010spectral}. We aim to obtain lower bounds for the two expectations on the right-hand side of \eqref{eq::ridge-decomp-again}. Regarding the first term, we have
\begin{align}\label{eq:risk:ridge:part1:suboptimal}
\E \| (X^{T}X+\lambda I)^{-1}\lambda \beta^* \|_2^{2} 
   &= \lambda^2\E\|Q(\Lambda + \lambda I)^{-1}Q^T\beta^* \|^2\nonumber \\
   &=\lambda^2\E\|(\Lambda + \lambda I)^{-1}Q^T\beta^* \|^2\nonumber \\
   &= \E \left [\lambda^2\sum^n_{i=1}\frac{(q_i^T\beta^*)^2}{(\sigma_{i} + \lambda)^2} + \sum^p_{i=n+1}(q_i^T\beta^*)^2\mathbbm{1}_{p>n} \right ] \nonumber \\
   &\geq \sum^p_{i=n+1}\E(q_i^T\beta^*)^2\mathbbm{1}_{p>n}=0\vee \frac{(p-n)k\mu^2}{p}, 
\end{align}
where to obtain the last equality we have used the results that $\|\beta^*\|_2^2=k\mu^2$ and the spherical distribution $\sqrt{p}q_i$ is isotropic. Now consider the second term $\E\|(X^{T}X+\lambda I)^{-1}X^{T}z\|_2^{2}$. Then,
\begin{align}
    &\E\|(X^{T}X+\lambda I)^{-1}X^{T}z\|_2^{2} \ge \E\bigg[\frac{1}{(\lambda + \sigma_1)^2} \|X^Tz\|_2^2 \bigg]\nonumber\\
    =& \E \Bigg [\frac{1}{(\lambda + \sigma_1)^2}\|X^Tz\|_2^2 \mathbbm{1}_{\left \{\sigma_1 \le \left(2 + \sqrt{\frac{p}{n}}\right)^2\right \}} \Bigg ] + \E \left [\frac{1}{(\lambda + \sigma_1)^2}\|X^Tz\|_2^2 \mathbbm{1}_{\left \{\sigma_1 > \left(2 + \sqrt{\frac{p}{n}}\right)^2\right \}} \right ]\nonumber\\
    \ge &\frac{1}{\bigg( \lambda + \Big( 1 + \sqrt{\frac{p}{n}} \Big )^2\bigg )^2} \cdot  \E \left [ \|X^Tz\|_2^2 \mathbbm{1}_{\left \{\sigma_1 \le \left(2 + \sqrt{\frac{p}{n}}\right)^2\right \}} \right ]\nonumber\\
    =& \frac{1}{\bigg( \lambda + \Big( 1 + \sqrt{\frac{p}{n}} \Big )^2\bigg )^2} \cdot \E \left [ \E \left [\|X^Tz\|_2^2 \mathbbm{1}_{\left \{\sigma_1 \le \left(2 + \sqrt{\frac{p}{n}}\right)^2\right \}}  \bigg | X\right ]\right ]\nonumber\\
    =& \frac{1}{\bigg( \lambda + \Big( 1 + \sqrt{\frac{p}{n}} \Big )^2\bigg )^2} \cdot \E \Bigg [\|X\|_F^2 \mathbbm{1}_{\left \{\sigma_1 \le \left(2 + \sqrt{\frac{p}{n}}\right)^2\right \}}  \Bigg ]\nonumber\\
    =& \frac{1}{\bigg( \lambda + \Big( 1 + \sqrt{\frac{p}{n}} \Big )^2\bigg )^2} \cdot \left ( \E \|X\|_F^2 - \E \left [\|X\|_F^2 \mathbbm{1}_{\left \{\sigma_1 > \left(2 + \sqrt{\frac{p}{n}}\right)^2\right \}}  \right ] \right )\nonumber\\
    \stackrel{(i)}{\ge}& \frac{1}{\bigg( \lambda + \Big( 1 + \sqrt{\frac{p}{n}} \Big )^2\bigg )^2} \cdot \left ( \E \|X\|_F^2 - (\E \|X\|_F^4)^{\frac{1}{2}} \left ( \mathbb{P} \left ( \sigma_1 > \left (2 + \sqrt{\frac{p}{n}} \right )^2\right )\right )^{\frac{1}{2}} \right )\nonumber\\
    \stackrel{(ii)}{\ge}&\frac{1}{\bigg( \lambda + \Big( 1 + \sqrt{\frac{p}{n}} \Big )^2\bigg )^2} \cdot \left ( p -\sqrt{p \left ( p + \frac{2}{n}\right )}\cdot \sqrt{2e^{-n/2}} \right ), \label{eq::ridge-decomp-second-term-lower bound}
\end{align}
where to obtain $(i)$ we have used the Cauchy-Schwarz inequality, and $(ii)$ holds by using the fact \(n\|X\|^2_F \sim \chi^2_{np}\) and applying \cref{lem::eval-conc}. Combining \eqref{eq::ridge-decomp-again}-\eqref{eq::ridge-decomp-second-term-lower bound} gives us
\begin{align}
\label{lower:bound:one:form}
&~\E\|\hat{\beta}^R(\lambda) - \beta^*\|_2^2 \nonumber \\
\geq &~\underbrace{0\vee \frac{(p-n)k\mu^2}{p}+\frac{1}{\bigg( \lambda + \Big( 1 + \sqrt{\frac{p}{n}} \Big )^2\bigg )^2} \cdot \left ( p -\sqrt{p \left ( p + \frac{2}{n}\right )}\cdot \sqrt{2e^{-n/2}} \right )}_{:=f(\lambda)}.
\end{align}
Next, we develop a different type of lower bound for $\E\|\hat{\beta}^R(\lambda) - \beta^*\|_2^2$ which will be used together with $f(\lambda)$ for the final part of the proof. Since \(g(x) := \frac{1}{(1+x)^2} - (1 -2x) \ge 0, \forall x > 0\), it is straightforward to see that
\begin{align*}
    \bigg(\frac{1}{\lambda}X^{T}X+I\bigg)^{-2} - \bigg( I - \frac{2}{\lambda}X^{T}X \bigg) 
    &= Q\bigg[ \bigg( \frac{1}{\lambda} \Lambda +I \bigg)^{-2} - \bigg(I - \frac{2}{\lambda}\Lambda  \bigg) \bigg] Q^{T}  \\
    &= Q \diag \bigg[ h\Big(\frac{\sigma_{1}}{\lambda}\Big), ~\ldots~, h\Big( \frac{\sigma_{p}}{\lambda} \Big) \bigg]Q^T \geq \mathbf{0}_{p\times p}.
\end{align*}
As a result, we obtain
\begin{align}
    \E\|\hat{\beta}^R(\lambda) - \beta^*\|_2^2 &= \E \| (X^{T}X+\lambda I)^{-1}\lambda \beta^* \|_2^{2}+\E\|(X^{T}X+\lambda I)^{-1}X^{T}z\|_2^{2}\nonumber\\
    &\ge \E \left [(\beta^*)^T\left (I - \frac{2}{\lambda}X^TX \right )\beta^* + \frac{1}{\lambda^2}z^TX\left (I - \frac{2}{\lambda}X^TX \right )X^Tz \right]\nonumber\\
    &= \|\beta^*\|_2^2 - \frac{2}{\lambda}\E\|X\beta^*\|_2^2 + \frac{1}{\lambda^2}\E\|X^Tz\|_2^2 - \frac{2}{\lambda^3}\E\|XX^Tz\|_2^2\nonumber\\
    &= \underbrace{k\mu^2\left (1 - \frac{2}{\lambda} \right ) + \frac{p}{\lambda^2} - \frac{2}{\lambda^3}\left (p + \frac{p^2}{n} + \frac{p}{n} \right)}_{:=g(\lambda)},\label{eq:lowerbound_firststep:linf}
\end{align}
where in the last equality we have used the results that $\E X^TX=I_p, \E(XX^T)^2=\frac{p(n+p+1)}{n^2}I_n$. The result for $\E(X^TX)^2$ has been derived in the proof of Lemma \ref{ridge:upper:lemma}. For simplicity, we skip similar calculations (switching the role of $n$ and $p$) for $\E(XX^T)^2$.

Now based on \eqref{lower:bound:one:form} and \eqref{eq:lowerbound_firststep:linf}, we have
\begin{align*}
\inf_{\lambda >0}\E\|\hat{\beta}^R(\lambda) - \beta^*\|_2^2 \geq \inf_{\lambda > 0} f(\lambda) \vee g(\lambda).
\end{align*}
In the rest of the proof, we evaluate $\inf_{\lambda > 0} f(\lambda) \vee g(\lambda)$. We first argue that there exists an optimal tuning $\lambda_n^* \in (0,\infty)$ such that $\inf_{\lambda > 0} f(\lambda) \vee g(\lambda)=f(\lambda_n^*) \vee g(\lambda_n^*)$. We prove this statement by showing the following for $\bar{\lambda}=p(k\mu^2)^{-1}$,
\begin{align*}
\lim_{\lambda \rightarrow 0}f(\lambda) \vee g(\lambda)>f(\bar{\lambda}) \vee g(\bar{\lambda}), \quad \lim_{\lambda \rightarrow \infty}f(\lambda) \vee g(\lambda)>f(\bar{\lambda}) \vee g(\bar{\lambda}).
\end{align*}
Under the condition $(k\log(p/k))/n=o(1), k/p=o(1), \mu^2=o(\log(p/k))$, it holds that $n\gg k\log(p/k) \gg k\mu^2, p\gg k\log(p/k)\gg k\mu^2$ and $\bar{\lambda}\gg p/n$. Hence, it is straightforward to compute $f(\bar{\lambda}) \vee g(\bar{\lambda})=k\mu^2-\frac{k^2\mu^4}{p}(1+o(1))$, and 
\begin{align*}
\lim_{\lambda \rightarrow \infty}f(\lambda) \vee g(\lambda)&=k\mu^2 > f(\bar{\lambda}) \vee g(\bar{\lambda}), \\
\lim_{\lambda \rightarrow 0}f(\lambda) \vee g(\lambda)&=0\vee \Big(k\mu^2-\frac{nk\mu^2}{p}\Big)+\frac{p(1+o(1))}{\Big(1+\sqrt{\frac{p}{n}}\Big)^4} \\
&\geq
\begin{cases}
\frac{p(1+o(1))}{16} & \text{when~} p\leq n\\
 k\mu^2-\frac{nk\mu^2}{p}+\frac{n^2(1+o(1))}{16p}&\text{when~}p > n
\end{cases} \\
& > f(\bar{\lambda}) \vee g(\bar{\lambda}).
\end{align*}
We proceed to show $\frac{\lambda_n^*}{(1+\sqrt{\frac{p}{n}})^2}\rightarrow \infty$. Otherwise, $\frac{\lambda_n^*}{(1+\sqrt{\frac{p}{n}})^2}\rightarrow C \in [0,\infty)$ (take a subsequence if necessary). As a result,
\begin{align*}
\inf_{\lambda > 0} f(\lambda) \vee g(\lambda)&=f(\lambda_n^*) \vee g(\lambda_n^*) \\
&\geq f(\lambda_n^*)=0\vee \Big(k\mu^2-\frac{nk\mu^2}{p}\Big)+\frac{p(1+o(1))}{(1+C)^2(1+\sqrt{\frac{p}{n}})^4}  \\
&\geq
\begin{cases}
\frac{p(1+o(1))}{16(1+C)^2} & \text{when~} p\leq n\\
 k\mu^2-\frac{nk\mu^2}{p}+\frac{n^2(1+o(1))}{16(1+C)^2p}&\text{when~}p > n
\end{cases} \\
& > f(\bar{\lambda}) \vee g(\bar{\lambda}),
\end{align*}
leading to a contradiction. Now that $\frac{\lambda_n^*}{(1+\sqrt{\frac{p}{n}})^2}\rightarrow \infty$, we can obtain
\begin{align*}
\inf_{\lambda > 0} f(\lambda) \vee g(\lambda)&=f(\lambda_n^*) \vee g(\lambda_n^*) \\
&\geq g(\lambda_n^*) =k\mu^2-\frac{2k\mu^2}{\lambda_n^*}+\frac{p}{(\lambda_n^*)^2}-\frac{2}{(\lambda_n^*)^3}\Big(p+\frac{p^2}{n}+\frac{p}{n}\Big) \\
&=k\mu^2-\frac{2k\mu^2}{\lambda_n^*}+\frac{p(1+o(1))}{(\lambda_n^*)^2} \\
&\geq \inf_{\lambda>0} \Big(k\mu^2-\frac{2k\mu^2}{\lambda}+\frac{p(1+o(1))}{\lambda^2}\Big)=k\mu^2-\frac{k^2\mu^4}{p}(1+o(1)).
\end{align*}

%% file: reference.bib
@article{ghosh2021exponential,
  title={Exponential Tail Bounds for Chisquared Random Variables},
  author={Ghosh, Malay},
  journal={Journal of Statistical Theory and Practice},
  volume={15},
  number={2},
  pages={1--6},
  year={2021},
  publisher={Springer}
}

@book{vershynin2018high,
  title={High-dimensional probability: An introduction with applications in data science},
  author={Vershynin, Roman},
  volume={47},
  year={2018},
  publisher={Cambridge university press}
}

@INPROCEEDINGS{6283602,

  author={Jalali, Shirin and Maleki, Arian and Baraniuk, Richard},

  booktitle={2012 IEEE International Symposium on Information Theory Proceedings}, 

  title={Minimum complexity pursuit: Stability analysis}, 

  year={2012},

  volume={},

  number={},

  pages={1857-1861},

  doi={10.1109/ISIT.2012.6283602}}

@book{wainwright2019high,
  title={High-dimensional statistics: A non-asymptotic viewpoint},
  author={Wainwright, Martin J},
  volume={48},
  year={2019},
  publisher={Cambridge university press}
}

@article{vershynin2010introduction,
  title={Introduction to the non-asymptotic analysis of random matrices},
  author={Vershynin, Roman},
  journal={arXiv preprint arXiv:1011.3027},
  year={2010}
}

@book{johnstone19,
  title={Gaussian estimation: Sequence and wavelet models},
  author={Iain M. Johnstone},
  year={2019},
}

@article{hastie2020best,
   title={Best subset, forward stepwise or lasso? Analysis and recommendations based on extensive comparisons},
   author={Hastie, Trevor and Tibshirani, Robert and Tibshirani, Ryan},
   journal={Statistical Science},
   volume={35},
   number={4},
   pages={579--592},
   year={2020},
   publisher={Institute of Mathematical Statistics}
 }

@article{wang2020bridge,
   title={Which bridge estimator is the best for variable selection?},
   author={Wang, Shuaiwen and Weng, Haolei and Maleki, Arian},
   journal={The Annals of Statistics},
   volume={48},
   number={5},
   pages={2791--2823},
   year={2020},
   publisher={Institute of Mathematical Statistics}
 }

@article{donoho1992maximum,
  title={Maximum entropy and the nearly black object},
  author={Donoho, David L and Johnstone, Iain M and Hoch, Jeffrey C and Stern, Alan S},
  journal={Journal of the Royal Statistical Society: Series B (Methodological)},
  volume={54},
  number={1},
  pages={41--67},
  year={1992},
  publisher={Wiley Online Library}
}

@article{raskutti2011minimax,
   title={Minimax rates of estimation for high-dimensional linear regression over $\ell_q$-balls},
   author={Raskutti, Garvesh and Wainwright, Martin J and Yu, Bin},
   journal={IEEE transactions on information theory},
   volume={57},
   number={10},
   pages={6976--6994},
   year={2011},
   publisher={IEEE}
 }

@article{bellec2018slope,
  title={Slope meets lasso: improved oracle bounds and optimality},
  author={Bellec, Pierre C and Lecu{\'e}, Guillaume and Tsybakov, Alexandre B},
  journal={The Annals of Statistics},
  volume={46},
  number={6B},
  pages={3603--3642},
  year={2018},
  publisher={JSTOR}
}

@article{su2016slope,
  title={SLOPE is adaptive to unknown sparsity and asymptotically minimax},
  author={Su, Weijie and Candes, Emmanuel},
  year={2016}
}

@article{bogdan2015slope,
  title={SLOPE—adaptive variable selection via convex optimization},
  author={Bogdan, Ma{\l}gorzata and Van Den Berg, Ewout and Sabatti, Chiara and Su, Weijie and Cand{\`e}s, Emmanuel J},
  journal={The annals of applied statistics},
  volume={9},
  number={3},
  pages={1103},
  year={2015},
  publisher={NIH Public Access}
}

@incollection{donoho1997universal,
  title={Universal near minimaxity of wavelet shrinkage},
  author={Donoho, David L and Johnstone, Iain M and Kerkyacharian, G and Picard, Dominique},
  booktitle={Festschrift for Lucien Le Cam},
  pages={183--218},
  year={1997},
  publisher={Springer}
}

@article{verzelen2012minimax,
  title={Minimax risks for sparse regressions: Ultra-high dimensional phenomenons},
  author={Verzelen, Nicolas},
  year={2012}
}

@article{hoerl1970ridge,
  title={Ridge regression: Biased estimation for nonorthogonal problems},
  author={Hoerl, Arthur E and Kennard, Robert W},
  journal={Technometrics},
  volume={12},
  number={1},
  pages={55--67},
  year={1970},
  publisher={Taylor \& Francis}
}

@article{tibshirani1996regression,
  title={Regression shrinkage and selection via the lasso},
  author={Tibshirani, Robert},
  journal={Journal of the Royal Statistical Society Series B: Statistical Methodology},
  volume={58},
  number={1},
  pages={267--288},
  year={1996},
  publisher={Oxford University Press}
}

@article{candes2007dantzig,
  title={The Dantzig selector: Statistical estimation when p is much larger than n},
  author={Candes, Emmanuel and Tao, Terence},
  year={2007}
}

@book{hastie2009elements,
  title={The elements of statistical learning: data mining, inference, and prediction},
  author={Hastie, Trevor and Tibshirani, Robert and Friedman, Jerome H and Friedman, Jerome H},
  volume={2},
  year={2009},
  publisher={Springer}
}

@book{hastie2015statistical,
  title={Statistical learning with sparsity: the lasso and generalizations},
  author={Hastie, Trevor and Tibshirani, Robert and Wainwright, Martin},
  year={2015},
  publisher={CRC press}
}

@book{buhlmann2011statistics,
  title={Statistics for high-dimensional data: methods, theory and applications},
  author={B{\"u}hlmann, Peter and Van De Geer, Sara},
  year={2011},
  publisher={Springer Science \& Business Media}
}

@book{fan2020statistical,
  title={Statistical foundations of data science},
  author={Fan, Jianqing and Li, Runze and Zhang, Cun-Hui and Zou, Hui},
  year={2020},
  publisher={CRC press}
}

@article{beale1967discarding,
  title={The discarding of variables in multivariate analysis},
  author={Beale, Evelyn Martin Lansdowne and Kendall, Maurice George and Mann, DW},
  journal={Biometrika},
  volume={54},
  number={3-4},
  pages={357--366},
  year={1967},
  publisher={Oxford University Press}
}

@article{hocking1967selection,
  title={Selection of the best subset in regression analysis},
  author={Hocking, Ronald R and Leslie, RN},
  journal={Technometrics},
  volume={9},
  number={4},
  pages={531--540},
  year={1967},
  publisher={Taylor \& Francis}
}

@article{bickel2009simultaneous,
  title={Simultaneous analysis of LASSO and Dantzig selector},
  author={Bickel, P. J. and Ritov, Y. and Tsybakov, A.},
  journal={The Annals of Statistics},
  volume={37},
  number={4},
  pages={1705--1732},
  year={2009}
}

@article{guo2022signaltonoise,
  title = {Signal-to-noise ratio aware minimaxity and higher-order asymptotics},
  author = {Guo, Yilin and Weng, Haolei and Maleki, Arian},
  journal = {IEEE Transactions on Information Theory},
  year = {2023},
  eprint = {2211.05954},
  archiveprefix = {arXiv},
  primaryclass = {math.ST},
}

@misc{guo2024minimaxlr,
  title = {A note on the minimax risk of sparse linear regression},
  author = {Guo, Yilin and Ghosh, Shubhangi and Weng, Haolei and Maleki, Arian},
  year = {2024},
  eprint = {2405.05344},
  archiveprefix = {arXiv},
  primaryclass = {stat.ME},
}

@article{zou2005regularization,
  title={Regularization and variable selection via the elastic net},
  author={Zou, Hui and Hastie, Trevor},
  journal={Journal of the Royal Statistical Society Series B: Statistical Methodology},
  volume={67},
  number={2},
  pages={301--320},
  year={2005},
  publisher={Oxford University Press}
}

@book{durrett2019probability,
  title={Probability: theory and examples},
  author={Durrett, Rick},
  volume={49},
  year={2019},
  publisher={Cambridge university press}
}

@book{bai2010spectral,
  title={Spectral analysis of large dimensional random matrices},
  author={Bai, Zhidong and Silverstein, Jack W},
  volume={20},
  year={2010},
  publisher={Springer}
}

@article{mazumder2023subset,
  title={Subset selection with shrinkage: Sparse linear modeling when the SNR is low},
  author={Mazumder, Rahul and Radchenko, Peter and Dedieu, Antoine},
  journal={Operations Research},
  volume={71},
  number={1},
  pages={129--147},
  year={2023},
  publisher={INFORMS}
}

@article{hazimeh2020fast,
  title={Fast best subset selection: Coordinate descent and local combinatorial optimization algorithms},
  author={Hazimeh, Hussein and Mazumder, Rahul},
  journal={Operations Research},
  volume={68},
  number={5},
  pages={1517--1537},
  year={2020},
  publisher={INFORMS}
}

@article{le2017does,
  title={Does $\ell_p$-minimization outperform $\ell_1$-minimization?},
  author={Le, Zheng and Arian, Maleki and Haolei, Weng and Xiaodong, Wang and Teng, Long},
  journal={IEEE Trans. Inf. Theory},
  volume={63},
  number={11},
  pages={6896--6935},
  year={2017}
}

@article{weng2018overcoming,
  title={OVERCOMING THE LIMITATIONS OF PHASE TRANSITION BY HIGHER ORDER ANALYSIS OF REGULARIZATION TECHNIQUES},
  author={Weng, Haolei and Maleki, Arian and Zheng, Le},
  journal={The Annals of Statistics},
  volume={46},
  number={6A},
  pages={3099--3129},
  year={2018},
  publisher={JSTOR}
}

@article{wang2022does,
  title={Does SLOPE outperform bridge regression?},
  author={Wang, Shuaiwen and Weng, Haolei and Maleki, Arian},
  journal={Information and Inference: A Journal of the IMA},
  volume={11},
  number={1},
  pages={1--54},
  year={2022},
  publisher={Oxford University Press}
}
